\documentclass[12pt,letter]{amsart}
\usepackage{amsmath,amssymb,amsthm,amscd,amsxtra,amsfonts,mathrsfs,graphicx,enumerate}
\input xy
\xyoption{all}
\newtheorem{Theorem}{Theorem}[section]
\newtheorem{Lemma}[Theorem]{Lemma}
\newtheorem{Proposition}[Theorem]{Proposition}
\newtheorem{Standard Theorem}[Theorem]{Standard Theorem}
\newtheorem{Corollary}[Theorem]{Corollary}
\newtheorem{Example}[Theorem]{Example}
\newtheorem{Remark}[Theorem]{Remark}

\newtheorem{Definition}[Theorem]{Definition}

\newtheorem{Definition and Lemma}[Theorem]{Definition and Lemma}

\newtheorem{Lemma and Definition}[Theorem]{Lemma/Definition}
\newtheorem{Notation}[Theorem]{Notation}
\newtheorem*{Proposition A}{Proposition A}
\newtheorem*{Theorem B}{Theorem B}
\newtheorem*{Theorem C}{Theorem C}
\newtheorem*{Theorem D}{Theorem D}
\newtheorem*{Theorem E}{Theorem E}
\advance\evensidemargin-.5in
\advance\oddsidemargin-.5in
\advance\textwidth1in

\begin{document}
 \author{Charlie Beil}
 \thanks{The author was supported in part by the Simons Foundation, a DOE grant, and the PFGW grant, which he gratefully acknowledges.}
 \address{Simons Center for Geometry and Physics, State University of New York, Stony Brook, NY 11794-3636, USA}
 \email{cbeil@scgp.stonybrook.edu}
 \title{On the noncommutative geometry of square superpotential algebras}
 \date{}

\begin{abstract}
A superpotential algebra is \textit{square} if its quiver admits an embedding into a two-torus such that the image of its underlying graph is a square grid, possibly with diagonal edges in the unit squares; examples are provided by dimer models in physics.  Such an embedding reveals much of the algebras representation theory through a device we introduce called an \textit{impression}.  Let $A$ be a square superpotential algebra, $Z$ its center, and $\mathfrak{m}$ the maximal ideal at the origin of $\operatorname{Spec}Z$.  Using an impression, we 
\begin{itemize}
 \item give a classification of all simple $A$-modules up to isomorphism, and give algebraic and homological characterizations of the simple $A$-modules of maximal $k$-dimension;
 \item show that $Z$ is a 3-dimensional normal toric domain and $Z_{\mathfrak{m}}$ is Gorenstein, by determining transcendence bases and $Z$-regular sequences; and
 \item show that $A_{\mathfrak{m}}$ is a noncommutative crepant resolution of $Z_{\mathfrak{m}}$, and thus a local Calabi-Yau algebra.
\end{itemize}

A particular class of square superpotential algebras, the $Y^{p,q}$ algebras, is considered in detail.  We show that the Azumaya and smooth loci of the centers coincide, and propose that each ramified maximal ideal sitting over the singular locus is the exceptional locus of a blowup shrunk to zero size.
\end{abstract}
\maketitle
\tableofcontents

\section{Introduction}

\subsection{Overview} \label{Overview}

Superpotential algebras are a class of quiver algebras that have arisen in string theory and have found mathematical interest in their own right.  We will consider a special class of these algebras, which we call \textit{square superpotential algebras}; the quiver of such an algebra admits an embedding into a two-torus such that the image of its underlying graph is a square grid, possibly with diagonal edges in the unit squares.  

We briefly state the main results of the paper.  In section \ref{Impressions} we introduce a device called an \textit{impression} and establish a few key properties.  An impression $(\tau,B)$ of a representable algebra $A$ is a closely related commutative algebra $B$ that contains the center $Z$ as a subalgebra, together with an algebra monomorphism $\tau: A \hookrightarrow M_d(B)$ with certain properties.
An impression is useful because, in contrast to the definition of an order, it explicitly determines both the center of $A$ and all simple $A$-module isoclasses of maximal $k$-dimension--what we call the \textit{large modules}.  In favorable cases the large module isoclasses are parameterized by the smooth locus of the algebras center.\footnote{If $A$ is a finitely-generated $k$-algebra, module-finite over its center $Z$ (hence noetherian \cite[Theorem 4.2.1]{Smith}), then the maximal $k$-dimension $d$ of the simple $A$-modules is finite \cite[Theorem 4.2.2]{Smith}.  If $A$ is also prime and $k$ is algebraically closed then the `Azumaya locus' parameterizes the isoclasses of large modules \cite[Proposition 3.1.a]{BGood}.  Le Bruyn \cite[Theorem 1]{Le Bruyn} and Brown and Goodearl \cite[Theorem 3.8]{BGood} showed that if $A$ is additionally Auslander-regular, Cohen-Macaulay, and if the compliment of the Azumaya locus has codimension at least 2 in $\operatorname{Max}Z$, then the Azumaya and smooth loci coincide.}  Specifically we show

\begin{Proposition A}
Let $(\tau, B)$ be an impression of a finitely-generated algebra $A$ module-finite over its center, with $B$ prime.  If $V$ is a large $A$-module then there is some $\mathfrak{q} \in \operatorname{Max}B$ such that $V \cong (B/\mathfrak{q})^d$, where $av := \tau_{\mathfrak{q}}(a) v$.
\end{Proposition A}

%
We then prove some general results that will be useful in our analysis of square superpotential algebras, such as the following theorem (see also Proposition \ref{strange} and Theorem \ref{projective dimensions}) .

\begin{Theorem B}
Let $A=kQ/I$ be a quiver algebra that admits a pre-impression $(\tau,B)$ such that $\tau(e_i) = E_{ii}$ and $\bar{\tau}(e_iAe_i) = \bar{\tau}(e_jAe_j) \subset B$ for each $i,j \in Q_0$.  
Then $A$ and its center $Z$ are noetherian rings, $A$ is a finitely-generated $Z$-module, and
$$Z = k \left[ \sum_{i \in Q_0} \gamma_i \in \bigoplus_{i \in Q_0}e_iAe_i \ | \ \bar{\tau}(\gamma_i)= \bar{\tau}(\gamma_j) \text{ for each } i,j \in Q_0 \right].$$
In particular, $Z \cong Ze_i = e_iAe_i$ for each $i \in Q_0$.
\end{Theorem B}

In section \ref{Impressions of square superpotential algebras} we determine an impression of square superpotential algebras.  In sections \ref{Gorenstein centers}, \ref{Classification of simples}, and \ref{Noncommutative crepant resolutions} we use this impression to prove the following two theorems.  Let $A$ be a square superpotential algebra, $Z$ its center, $\mathfrak{m}$ the origin of $\operatorname{Max}Z$, and $(\tau,B = k[x_1,x_2,y_1,y_2])$ an impression of $A$.

\begin{Theorem C}
$Z$ is a 3-dimensional normal toric domain and the localization $Z_{\mathfrak{m}}$ at the origin $\mathfrak{m}$ of $\operatorname{Max}Z$ is Gorenstein.  Furthermore, $A_{\mathfrak{m}} := Z_{\mathfrak{m}} \otimes_Z A$ is a noncommutative crepant resolution of $Z_{\mathfrak{m}}$, and consequently a local Calabi-Yau algebra of dimension 3.
\end{Theorem C}

\begin{Theorem D}
Let $A$ be a square superpotential algebra with impression $(\tau,B)$, and let $V$ be a simple $A$-module.  Set $\mathfrak{p} := \operatorname{ann}_AV$ and $\mathfrak{m} := \mathfrak{p} \cap Z \in \operatorname{Max}Z$.  Then $\operatorname{dim}_ke_iV \leq 1$ for each $i \in Q_0$.  
Furthermore, one of the following holds.
\begin{enumerate}
 \item $V$ is a vertex simple $A$-module, in which case $A/\mathfrak{p} \cong V$ as $A$-modules.
 \item $V$ is supported on a single cycle $c$ in $A$ up to cyclic permutation.  
   \begin{enumerate}
     \item If $Q$ is not McKay then $\bar{\tau}(c)$ is divisible by precisely two of $x_1,x_2,y_1,y_2$.
     \item If $Q$ is McKay with $\tau$ defined in Proposition \ref{triangular}, then $\bar{\tau}(c)$ divisible by precisely one of $x,y,z$.
   \end{enumerate}
 \item $V$ is a large $A$-module, in which case
  \begin{enumerate} 
   \item $A/\mathfrak{p} \cong V^{|Q_0|}$ as $A$-modules;
   \item there is a point $\mathfrak{q} \in \operatorname{Max}B$ such that $V \cong (B/\mathfrak{q})^{|Q_0|}$, where the module structure of $(B/\mathfrak{q})^{|Q_0|}$ is given by $av := \tau_{\mathfrak{q}}(a)v$; and
   \item the projective dimension of $V$ is determined by $\mathfrak{m}$: $$\operatorname{pd}_A(V) = \operatorname{pd}_{A_{\mathfrak{m}}}(A_{\mathfrak{m}}/\mathfrak{p}_{\mathfrak{m}}) = \operatorname{pd}_{Z_{\mathfrak{m}}}(Z_{\mathfrak{m}}/\mathfrak{m}_{\mathfrak{m}}).$$
  \end{enumerate}
\end{enumerate}
\end{Theorem D}

A special class of square superpotential algebras conjecturally related to Sasaki-Einstein manifolds, the $Y^{p,q}$ algebras, is considered in section \ref{The Ypq algebras}.  
Recall that the Azumaya locus $U$ of a prime finitely-generated algebra $A$ over an algebraically closed field $k$, module-finite over its center $Z$, is the open dense set of points $\mathfrak{m} \in \operatorname{Max}Z$ such that $A/A\mathfrak{m} \cong M_d(k)$, where $d$ is the $k$-dimension of the large modules \cite[Theorem 4.2.7]{Smith} (or equivalently the PI degree of $A$ \cite[Proposition 13.7.14]{MR}).  $U$ then consists of the points in $\operatorname{Max}Z$ whose `noncommutative residue fields' have full rank.

\begin{Theorem E}
Let $A$ be a (non-localized) $Y^{p,q}$ algebra.  Then the following hold.
\begin{enumerate}
 \item If $p \not = q$ and $V$ is a simple $A$-module, then $V$ is either a vertex simple module or a large module.
 \item The Azumaya locus of $A$ coincides with the smooth locus of $Z$.
 \item $A$ is homologically homogeneous of global dimension 3.
\end{enumerate}
\end{Theorem E}

Finally, we introduce a proposal regarding `point-like' exceptional loci by using symplectic reduction on the impression of the $Y^{p,q}$ algebras.

Section \ref{Endomorphism rings} is based on joint work with Alex Dugas, and I thank him for kindly allowing me to publish it here.  Questions regarding dimer models and noncommutative crepant resolutions have also been studied \cite{UY}, \cite{Mozgovoy}, \cite{Bocklandt}, and \cite{Broomhead}.

I would like to give very special thanks to my advisors David Morrison and David Berenstein for all of their encouragement and guidance.  
I would also like to thank Alex Dugas and Ken Goodearl for many useful discussions.  
I am grateful to a long list of people who have made helpful comments: Tom Howard, Birge Huisgen-Zimmermann, Paul Smith, Alastair King, James McKernan, Bill Jacob, Raphael Flauger, James Sparks, Bernhard Keller, Susan Siera, and Alastair Craw.
I am especially grateful to an anonymous referee for their careful reading and valuable comments.  
Also thanks to Coral, Aidan, Kael, Tea Rose, Leonard, and my parents and family for their wonderful support. 

\textbf{Conventions:} $k$ denotes an algebraically closed field of characteristic zero.  All algebras are unital and finitely-generated over $k$.  By module we mean \textit{left} module.  For brevity the term \textit{quiver algebra} is used in place of path algebra modulo relations.  By a cycle in a quiver we mean an oriented cycle.  The set of paths of length $n$ in a quiver $Q$ is denoted $Q_n$, and the set of all ($k$-linear combinations of) paths of length greater than or equal to $n$ is denoted $Q_{\geq n}$ (respectively $kQ_{\geq n}$).  $\operatorname{h}(p)$ and $\operatorname{t}(p)$ denote the head vertex and tail vertex of a path $p$, respectively.  Path concatenation is read right to left (following the composition of maps).  By a \textit{cyclic proper subpath} we mean a subpath of nonzero length.  The term \textit{superpotential algebra} is synonymous with \textit{vacualgebra} and \textit{quiver with potential}.

\subsection{Square superpotential algebras}

A superpotential algebra is a type of quiver algebra where the relations are derived from certain equations of motion in a physical theory.  A quiver algebra is a quotient of a path algebra, which is an algebra whose basis consists of all paths in a quiver, including the vertices, and multiplication is given by path concatenation: the product of two paths is their concatenation if it is defined, and zero otherwise.  A representation of (or module over) a quiver algebra is obtained by associating a vector space to each vertex of the quiver, representing each arrow by a linear map from the vector space at its tail to the vector space at its head, and requiring these linear maps satisfy the relations of the algebra.

We now define a superpotential algebra.  Let $Q$ be a quiver and $kQ$ its path algebra.  Two paths $p$ and $p'$ are \textit{cyclically equivalent} if $p$ is a cyclic permutation of the arrows of $p'$, so all non-cyclic paths are cyclically equivalent to zero.  The \textit{trace space} of $kQ$, denoted $\operatorname{tr}(kQ)$, is the $k$-vector space spanned by the paths of $Q$ up to cyclic equivalence, and an element of $\operatorname{tr}(kQ)$ is called a \textit{superpotential}.  For each $a \in Q_1$, define a $k$-linear map $\partial_a : \operatorname{tr}(kQ) \rightarrow kQ$ as follows: for each path $b_n \cdots b_1 \in Q_{\geq 1}$ with $b_1, \ldots, b_n \in Q_1$, set
$$\partial_a\left(b_n \cdots b_1 \right) := \sum_{1 \leq j \leq n} \delta(a,b_j) e_{\operatorname{t}(b_j)} b_{j-1} \cdots b_1 b_n \cdots b_{j+1},$$
for each $e \in Q_0$, set $\partial_a e := 0$, and extend $k$-linearly to $\operatorname{tr}(kQ)$.  For $W \in \operatorname{tr}(kQ)$, set
$$\partial W := \left\langle \partial_a W \ | \ a \in Q_1 \right\rangle.$$
The \textit{superpotential algebra} with quiver $Q$ and superpotential $W$ is then the quiver algebra $kQ/\partial W$.  In this paper we are interested in a particularly simple class of superpotential algebras that arise from what are called brane tilings--specifically, brane boxes and brane diamonds--in string theory (see \cite{FHHU} and references therein).  Embedding these relatively simple superpotential algebras into a two-torus is standard; what we introduce here that is new is a relationship between a particular choice of embedding (when it exists) and the representation theory of the corresponding algebras.

\begin{Definition} \label{square} \rm{
Let $Q$ be a quiver.  Suppose there are non-colinear elements $u,v \in \mathbb{Z}^2 \subset \mathbb{R}^2$ such that the underlying graph $\bar{Q}$ of $Q$ embeds into the two-torus $\mathbb{R}^2 / (\mathbb{Z}u \oplus \mathbb{Z}v)$ with the property that the preimage of $\bar{Q}$ under the quotient map
$$\pi: \mathbb{R}^2 \rightarrow \mathbb{R}^2/(\mathbb{Z}u\oplus \mathbb{Z}v)$$
is a square grid with vertex set $\pi^{-1}\left(\bar{Q_0}\right) = \mathbb{Z}^2$, and with at most one diagonal edge in each unit square.
Further suppose that $Q$ has an orientation where each unit square with no diagonal and each triangle with two unit length sides forms an oriented cycle; 
we call these the \textit{unit cycles} of $Q$.  Let $\Gamma_{c}$ (resp.\ $\Gamma_{cc}$) denote the clockwise (resp.\ counterclockwise) unit cycles up to cyclic equivalence.  We then call the quiver algebra $A = kQ/\partial W$ with superpotential
\begin{equation}\label{square superpotential}
W = \sum_{d \in \Gamma_c} d - \sum_{d' \in \Gamma_{cc}} d' \in \operatorname{tr}(kQ)
\end{equation}
a \textit{square} superpotential algebra.
} \end{Definition} 

It will be useful to consider the covering quiver $\widetilde{Q}$ (or \textit{periodic quiver} in the physics literature) of $Q$, whose underlying graph is $\pi^{-1}\left(\bar{Q}\right)$.  By abuse of notation we will write $\pi: \widetilde{Q} \rightarrow Q$ for the corresponding projection of quivers.

\begin{figure}
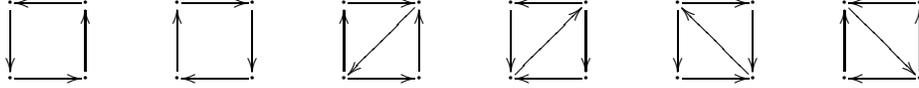

$$\xy (-5,-5)*{\cdot}="1";(5,-5)*{\cdot}="2";(-5,5)*{\cdot}="3";(5,5)*{\cdot}="4";
{\ar"1";"2"};{\ar@{->}"2";"4"};{\ar@{->}"4";"3"};{\ar@{->}"3";"1"};
\endxy
\ \ \ \ \ \ \ \
\xy (-5,-5)*{\cdot}="1";(5,-5)*{\cdot}="2";(-5,5)*{\cdot}="3";(5,5)*{\cdot}="4";
{\ar"1";"3"};{\ar"3";"4"};{\ar"4";"2"};{\ar"2";"1"};
\endxy
\ \ \ \ \ \ \ \
\xy (-5,-5)*{\cdot}="1";(5,-5)*{\cdot}="2";(-5,5)*{\cdot}="3";(5,5)*{\cdot}="4";
{\ar"1";"2"};{\ar"2";"4"};{\ar"1";"3"};{\ar"3";"4"};{\ar"4";"1"};
\endxy
\ \ \ \ \ \ \ \
\xy (-5,-5)*{\cdot}="1";(5,-5)*{\cdot}="2";(-5,5)*{\cdot}="3";(5,5)*{\cdot}="4";
{\ar"4";"2"};{\ar"2";"1"};{\ar"4";"3"};{\ar"3";"1"};{\ar"1";"4"};
\endxy
\ \ \ \ \ \ \ \
\xy (-5,-5)*{\cdot}="1";(5,-5)*{\cdot}="2";(-5,5)*{\cdot}="3";(5,5)*{\cdot}="4";
{\ar"1";"2"};{\ar"3";"1"};{\ar"3";"4"};{\ar"4";"2"};{\ar"2";"3"};
\endxy
\ \ \ \ \ \ \ \
\xy (-5,-5)*{\cdot}="1";(5,-5)*{\cdot}="2";(-5,5)*{\cdot}="3";(5,5)*{\cdot}="4";
{\ar"2";"1"};{\ar"1";"3"};{\ar"2";"4"};{\ar"4";"3"};{\ar"3";"2"};
\endxy
$$
\caption{ 
The 6 possible `building blocks' for the quiver of a square superpotential algebra.}
\label{squares}
\end{figure}

\begin{Example} \label{conifold} \rm{The two quiver algebras given in figure \ref{1} are perhaps the simplest square superpotential algebras.  The center of the second example is the coordinate ring for the conifold.  The quivers on the right are drawn in the plane.
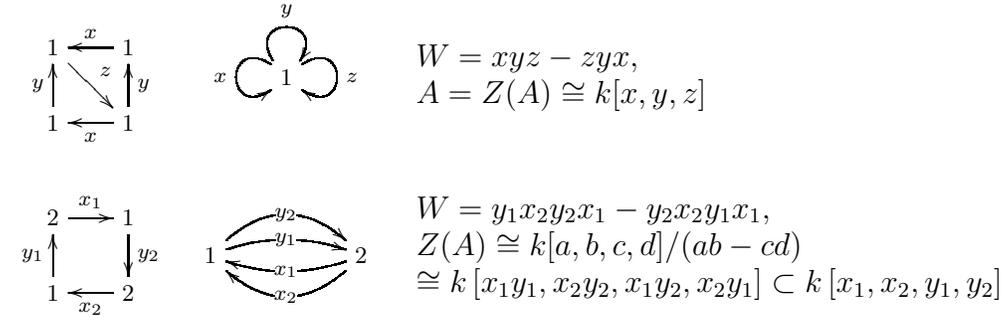
\begin{figure}
$$\begin{array}{ccl}
\xy (-5,-5)*+{\text{\scriptsize{$1$}}}="1";(-5,5)*+{\text{\scriptsize{$1$}}}="2";(5,-5)*+{\text{\scriptsize{$1$}}}="3";(5,5)*+{\text{\scriptsize{$1$}}}="4";
{\ar^y"1";"2"};{\ar^z"2";"3"};{\ar^x"3";"1"};{\ar_y"3";"4"};{\ar_x"4";"2"};\endxy
& \xymatrix{\text{\scriptsize{$1$}} \ar@(ul,dl)[]_x \ar@(ul,ur)[]^y \ar@(ur,dr)[]^z}
& \begin{array}{l} W = xyz -zyx, \\ A = Z(A) \cong k[x,y,z] \end{array}
\\ \\
\xy (-5,-5)*+{\text{\scriptsize{$1$}}}="1";(-5,5)*+{\text{\scriptsize{$2$}}}="2";(5,-5)*+{\text{\scriptsize{$2$}}}="3";(5,5)*+{\text{\scriptsize{$1$}}}="4";
{\ar^{y_1}"1";"2"};{\ar^{x_1}"2";"4"};{\ar^{y_2}"4";"3"};{\ar^{x_2}"3";"1"};
\endxy
& \xy (-10,0)*+{\text{\scriptsize{$1$}}}="1";(10,0)*+{\text{\scriptsize{$2$}}}="2";
{\ar@/^/|-{y_1}"1";"2"};{\ar@/^1.3pc/|-{y_2}"1";"2"};
{\ar@/^/|-{x_1}"2";"1"};{\ar@/^1.3pc/|-{x_2}"2";"1"};
\endxy
& \begin{array}{l} W = y_1x_2y_2x_1 - y_2x_2y_1x_1, \\ Z(A) \cong k[a,b,c,d]/(ab-cd) \\ \cong k\left[x_1y_1,x_2y_2,x_1y_2,x_2y_1 \right] \subset k\left[x_1,x_2,y_1,y_2 \right] \end{array}
\end{array}$$
\caption{Examples of square superpotential algebras: $\mathbb{A}_k^3$ and the conifold.}
\label{1}
\end{figure}
}\end{Example} 

\begin{Example} \label{Ypq} \rm{The $Y^{p,q}$ algebras form a class of square superpotential algebras.  
In string theory they are conjecturally related to a class of Sasaki-Einstein manifolds, namely the $Y^{p,q}$ manifolds.
This conjecture is based on a matching of symmetries, where certain `global symmetries' of the algebras are identified with isometries of the manifolds (see \cite{BFHMS}).
The $Y^{p,q}$ quivers are constructed by vertically stacking $p$ of any the three graphs given in figure \ref{2},
\begin{figure}
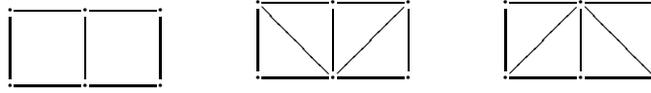

$$\begin{array}{ccc}
\xy (-10,-5)*{\cdot}="1";(-10,5)*{\cdot}="2";(0,-5)*{\cdot}="3";(0,5)*{\cdot}="4";(10,-5)*{\cdot}="5";(10,5)*{\cdot}="6";
{\ar@{-}"1";"2"};{\ar@{-}"3";"4"};{\ar@{-}"5";"6"};{\ar@{-}"1";"3"};{\ar@{-}"3";"5"};{\ar@{-}"2";"4"};{\ar@{-}"4";"6"};
\endxy \ \ \
&
\ \ \
\xy (-10,-5)*{\cdot}="1";(-10,5)*{\cdot}="2";(0,-5)*{\cdot}="3";(0,5)*{\cdot}="4";(10,-5)*{\cdot}="5";(10,5)*{\cdot}="6";
{\ar@{-}"1";"2"};{\ar@{-}"3";"4"};{\ar@{-}"5";"6"};{\ar@{-}"1";"3"};{\ar@{-}"3";"5"};{\ar@{-}"2";"4"};{\ar@{-}"4";"6"};
{\ar@{-}"2";"3"};{\ar@{-}"3";"6"};
\endxy
\ \ \
&
\ \ \
\xy (-10,-5)*{\cdot}="1";(-10,5)*{\cdot}="2";(0,-5)*{\cdot}="3";(0,5)*{\cdot}="4";(10,-5)*{\cdot}="5";(10,5)*{\cdot}="6";
{\ar@{-}"1";"2"};{\ar@{-}"3";"4"};{\ar@{-}"5";"6"};{\ar@{-}"1";"3"};{\ar@{-}"3";"5"};{\ar@{-}"2";"4"};{\ar@{-}"4";"6"};
{\ar@{-}"4";"1"};{\ar@{-}"4";"5"};
\endxy
\end{array}$$
\caption{The building blocks for the $Y^{p,q}$ quivers.}
\label{2}
\end{figure}
identifying vertices $(0,j) = (2,j)$ and $(i,0)=(i+i_0,p)$ for each $i,j$ and some $i_0 \in \{0,1\}$, and choosing a compatible orientation.  The label $q$ is given by
$$q = p- \# \left\{ 
\xy (-6,-3)*{\cdot}="1";(-6,3)*{\cdot}="2";(0,-3)*{\cdot}="3";(0,3)*{\cdot}="4";(6,-3)*{\cdot}="5";(6,3)*{\cdot}="6";
{\ar@{-}"1";"2"};{\ar@{-}"3";"4"};{\ar@{-}"5";"6"};{\ar@{-}"1";"3"};{\ar@{-}"3";"5"};{\ar@{-}"2";"4"};{\ar@{-}"4";"6"};
\endxy
\right\} - 2 \cdot \# \left\{ 
\xy (-6,-6)*{\cdot}="1";(-6,0)*{\cdot}="2";(0,-6)*{\cdot}="3";(0,0)*{\cdot}="4";(6,-6)*{\cdot}="5";(6,0)*{\cdot}="6";
(-6,6)*{\cdot}="7";(0,6)*{\cdot}="8";(6,6)*{\cdot}="9";
{\ar@{-}"1";"2"};{\ar@{-}"3";"4"};{\ar@{-}"5";"6"};{\ar@{-}"1";"3"};{\ar@{-}"3";"5"};{\ar@{-}"2";"4"};{\ar@{-}"4";"6"};
{\ar@{-}"3";"2"};{\ar@{-}"3";"6"};
{\ar@{-}"2";"7"};{\ar@{-}"4";"8"};{\ar@{-}"6";"9"};{\ar@{-}"7";"8"};{\ar@{-}"8";"9"};
{\ar@{-}"2";"8"};{\ar@{-}"6";"8"};
\endxy
\right\}.$$
Some examples are given in figure \ref{Ypq figure}.
}\end{Example}
\begin{figure}[t]
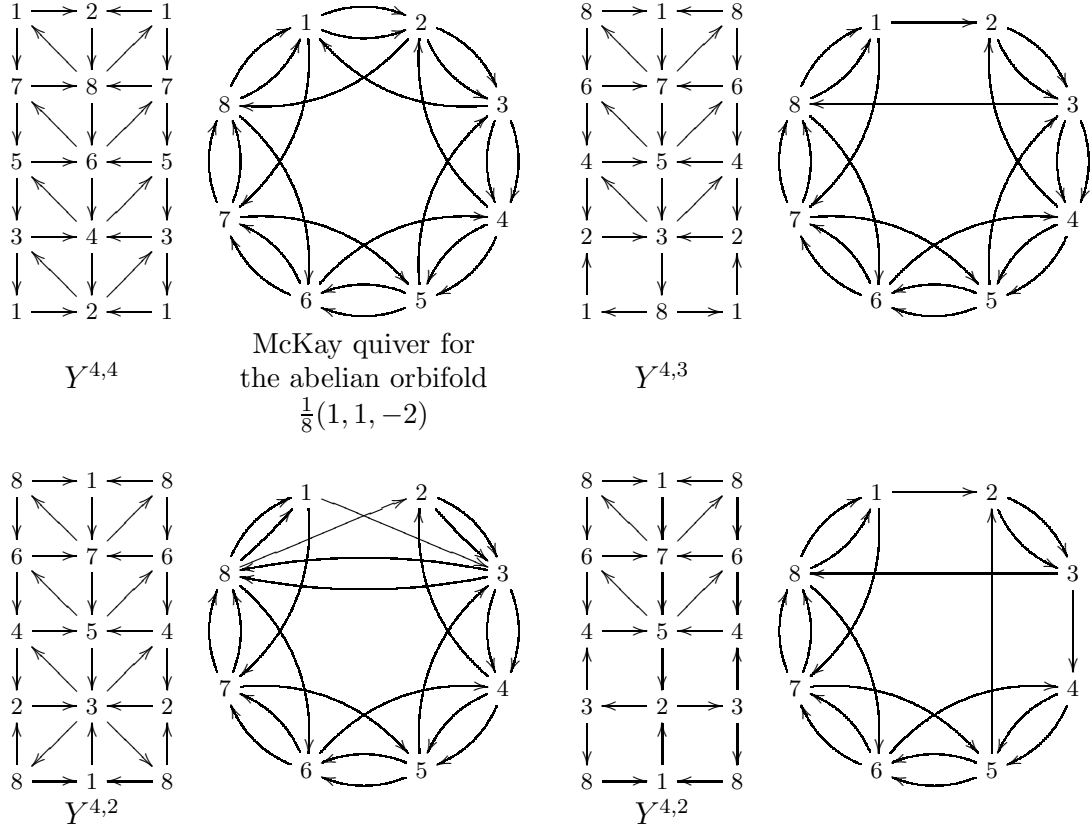

$$\begin{array}{ccccc}
\xy
(-10,-20)*+{\text{\scriptsize{$1$}}}="a1";(-10,-10)*+{\text{\scriptsize{$3$}}}="a2";(-10,0)*+{\text{\scriptsize{$5$}}}="a3";(-10,10)*+{\text{\scriptsize{$7$}}}="a4";(-10,20)*+{\text{\scriptsize{$1$}}}="a5";
(0,-20)*+{\text{\scriptsize{$2$}}}="b1";(0,-10)*+{\text{\scriptsize{$4$}}}="b2";(0,0)*+{\text{\scriptsize{$6$}}}="b3";(0,10)*+{\text{\scriptsize{$8$}}}="b4";(0,20)*+{\text{\scriptsize{$2$}}}="b5";
(10,-20)*+{\text{\scriptsize{$1$}}}="c1";(10,-10)*+{\text{\scriptsize{$3$}}}="c2";(10,0)*+{\text{\scriptsize{$5$}}}="c3";(10,10)*+{\text{\scriptsize{$7$}}}="c4";(10,20)*+{\text{\scriptsize{$1$}}}="c5";
{\ar@{->}"a1";"b1"};{\ar@{<-}"b1";"c1"};{\ar@{->}"a2";"b2"};{\ar@{<-}"b2";"c2"};{\ar@{->}"a3";"b3"};{\ar@{<-}"b3";"c3"};{\ar@{->}"a4";"b4"};{\ar@{<-}"b4";"c4"};{\ar@{->}"a5";"b5"};{\ar@{<-}"b5";"c5"};
{\ar@{<-}"a1";"a2"};{\ar@{<-}"b1";"b2"};{\ar@{<-}"c1";"c2"};
{\ar@{<-}"a2";"a3"};{\ar@{<-}"b2";"b3"};{\ar@{<-}"c2";"c3"};
{\ar@{<-}"a3";"a4"};{\ar@{<-}"b3";"b4"};{\ar@{<-}"c3";"c4"};
{\ar@{<-}"a4";"a5"};{\ar@{<-}"b4";"b5"};{\ar@{<-}"c4";"c5"};
{\ar@{->}"b1";"a2"};{\ar@{->}"b1";"c2"};
{\ar@{->}"b2";"a3"};{\ar@{->}"b2";"c3"};
{\ar@{->}"b3";"a4"};{\ar@{->}"b3";"c4"};
{\ar@{->}"b4";"a5"};{\ar@{->}"b4";"c5"};
\endxy
&
\xy (-18.478,7.654)*+{\text{\scriptsize{$8$}}}="8";(-18.478,-7.654)*+{\text{\scriptsize{$7$}}}="7";(-7.654,-18.478)*+{\text{\scriptsize{$6$}}}="6";(7.654,-18.478)*+{\text{\scriptsize{$5$}}}="5";(18.478,-7.654)*+{\text{\scriptsize{$4$}}}="4";(18.478,7.654)*+{\text{\scriptsize{$3$}}}="3";(7.654,18.478)*+{\text{\scriptsize{$2$}}}="2";(-7.654,18.478)*+{\text{\scriptsize{$1$}}}="1";
{\ar@/^/"1";"2"};{\ar@/_/"1";"2"};{\ar@/^1pc/"1";"7"};
{\ar@/^/"2";"3"};{\ar@/_/"2";"3"};{\ar@/^1pc/"2";"8"};
{\ar@/^/"3";"4"};{\ar@/_/"3";"4"};{\ar@/^1pc/"3";"1"};
{\ar@/^/"4";"5"};{\ar@/_/"4";"5"};{\ar@/^1pc/"4";"2"};
{\ar@/^/"5";"6"};{\ar@/_/"5";"6"};{\ar@/^1pc/"5";"3"};
{\ar@/^/"6";"7"};{\ar@/_/"6";"7"};{\ar@/^1pc/"6";"4"};
{\ar@/^/"7";"8"};{\ar@/_/"7";"8"};{\ar@/^1pc/"7";"5"};
{\ar@/^/"8";"1"};{\ar@/_/"8";"1"};{\ar@/^1pc/"8";"6"};
\endxy
&
&
\xy
(-10,-20)*+{\text{\scriptsize{$1$}}}="a1";(-10,-10)*+{\text{\scriptsize{$2$}}}="a2";(-10,0)*+{\text{\scriptsize{$4$}}}="a3";(-10,10)*+{\text{\scriptsize{$6$}}}="a4";(-10,20)*+{\text{\scriptsize{$8$}}}="a5";
(0,-20)*+{\text{\scriptsize{$8$}}}="b1";(0,-10)*+{\text{\scriptsize{$3$}}}="b2";(0,0)*+{\text{\scriptsize{$5$}}}="b3";(0,10)*+{\text{\scriptsize{$7$}}}="b4";(0,20)*+{\text{\scriptsize{$1$}}}="b5";
(10,-20)*+{\text{\scriptsize{$1$}}}="c1";(10,-10)*+{\text{\scriptsize{$2$}}}="c2";(10,0)*+{\text{\scriptsize{$4$}}}="c3";(10,10)*+{\text{\scriptsize{$6$}}}="c4";(10,20)*+{\text{\scriptsize{$8$}}}="c5";
{\ar@{<-}"a1";"b1"};{\ar@{->}"b1";"c1"};{\ar@{->}"a2";"b2"};{\ar@{<-}"b2";"c2"};{\ar@{->}"a3";"b3"};{\ar@{<-}"b3";"c3"};{\ar@{->}"a4";"b4"};{\ar@{<-}"b4";"c4"};{\ar@{->}"a5";"b5"};{\ar@{<-}"b5";"c5"};
{\ar@{->}"a1";"a2"};{\ar@{<-}"b1";"b2"};{\ar@{->}"c1";"c2"};
{\ar@{<-}"a2";"a3"};{\ar@{<-}"b2";"b3"};{\ar@{<-}"c2";"c3"};
{\ar@{<-}"a3";"a4"};{\ar@{<-}"b3";"b4"};{\ar@{<-}"c3";"c4"};
{\ar@{<-}"a4";"a5"};{\ar@{<-}"b4";"b5"};{\ar@{<-}"c4";"c5"};
{\ar@{->}"b2";"a3"};{\ar@{->}"b2";"c3"};
{\ar@{->}"b3";"a4"};{\ar@{->}"b3";"c4"};
{\ar@{->}"b4";"a5"};{\ar@{->}"b4";"c5"};
\endxy 
&
\xy (-18.478,7.654)*+{\text{\scriptsize{$8$}}}="8";(-18.478,-7.654)*+{\text{\scriptsize{$7$}}}="7";(-7.654,-18.478)*+{\text{\scriptsize{$6$}}}="6";(7.654,-18.478)*+{\text{\scriptsize{$5$}}}="5";(18.478,-7.654)*+{\text{\scriptsize{$4$}}}="4";(18.478,7.654)*+{\text{\scriptsize{$3$}}}="3";(7.654,18.478)*+{\text{\scriptsize{$2$}}}="2";(-7.654,18.478)*+{\text{\scriptsize{$1$}}}="1";
{\ar@{->}"1";"2"};{\ar@/^1pc/"1";"7"};
{\ar@/^/"2";"3"};{\ar@/_/"2";"3"};
{\ar@/^/"3";"4"};{\ar@/_/"3";"4"};{\ar@{->}"3";"8"};
{\ar@/^/"4";"5"};{\ar@/_/"4";"5"};{\ar@/^1pc/"4";"2"};
{\ar@/^/"5";"6"};{\ar@/_/"5";"6"};{\ar@/^1pc/"5";"3"};
{\ar@/^/"6";"7"};{\ar@/_/"6";"7"};{\ar@/^1pc/"6";"4"};
{\ar@/^/"7";"8"};{\ar@/_/"7";"8"};{\ar@/^1pc/"7";"5"};
{\ar@/^/"8";"1"};{\ar@/_/"8";"1"};{\ar@/^1pc/"8";"6"};
\endxy
\\
Y^{4,4} & \text{\small{\begin{tabular}{c} McKay quiver for \\ the abelian orbifold \\ $\frac18 (1,1,-2)$ \end{tabular}}} & & Y^{4,3} & 
\end{array}$$

$$\begin{array}{ccccc}
\xy
(-10,-20)*+{\text{\scriptsize{$8$}}}="a1";(-10,-10)*+{\text{\scriptsize{$2$}}}="a2";(-10,0)*+{\text{\scriptsize{$4$}}}="a3";(-10,10)*+{\text{\scriptsize{$6$}}}="a4";(-10,20)*+{\text{\scriptsize{$8$}}}="a5";
(0,-20)*+{\text{\scriptsize{$1$}}}="b1";(0,-10)*+{\text{\scriptsize{$3$}}}="b2";(0,0)*+{\text{\scriptsize{$5$}}}="b3";(0,10)*+{\text{\scriptsize{$7$}}}="b4";(0,20)*+{\text{\scriptsize{$1$}}}="b5";
(10,-20)*+{\text{\scriptsize{$8$}}}="c1";(10,-10)*+{\text{\scriptsize{$2$}}}="c2";(10,0)*+{\text{\scriptsize{$4$}}}="c3";(10,10)*+{\text{\scriptsize{$6$}}}="c4";(10,20)*+{\text{\scriptsize{$8$}}}="c5";
{\ar@{->}"a1";"b1"};{\ar@{<-}"b1";"c1"};{\ar@{->}"a2";"b2"};{\ar@{<-}"b2";"c2"};{\ar@{->}"a3";"b3"};{\ar@{<-}"b3";"c3"};{\ar@{->}"a4";"b4"};{\ar@{<-}"b4";"c4"};{\ar@{->}"a5";"b5"};{\ar@{<-}"b5";"c5"};
{\ar@{->}"a1";"a2"};{\ar@{->}"b1";"b2"};{\ar@{->}"c1";"c2"};
{\ar@{<-}"a2";"a3"};{\ar@{<-}"b2";"b3"};{\ar@{<-}"c2";"c3"};
{\ar@{<-}"a3";"a4"};{\ar@{<-}"b3";"b4"};{\ar@{<-}"c3";"c4"};
{\ar@{<-}"a4";"a5"};{\ar@{<-}"b4";"b5"};{\ar@{<-}"c4";"c5"};
{\ar@{->}"b2";"a1"};{\ar@{->}"b2";"c1"};
{\ar@{->}"b2";"a3"};{\ar@{->}"b2";"c3"};
{\ar@{->}"b3";"a4"};{\ar@{->}"b3";"c4"};
{\ar@{->}"b4";"a5"};{\ar@{->}"b4";"c5"};
\endxy
&
\xy (-18.478,7.654)*+{\text{\scriptsize{$8$}}}="8";(-18.478,-7.654)*+{\text{\scriptsize{$7$}}}="7";(-7.654,-18.478)*+{\text{\scriptsize{$6$}}}="6";(7.654,-18.478)*+{\text{\scriptsize{$5$}}}="5";(18.478,-7.654)*+{\text{\scriptsize{$4$}}}="4";(18.478,7.654)*+{\text{\scriptsize{$3$}}}="3";(7.654,18.478)*+{\text{\scriptsize{$2$}}}="2";(-7.654,18.478)*+{\text{\scriptsize{$1$}}}="1";
{\ar@/^1pc/"1";"7"};{\ar@{->}"1";"3"};
{\ar@/^/"2";"3"};{\ar@/_.1pc/"2";"3"};
{\ar@/^/"3";"4"};{\ar@/_/"3";"4"};{\ar@/_/"3";"8"};{\ar@/^/"3";"8"};
{\ar@/^/"4";"5"};{\ar@/_/"4";"5"};{\ar@/^1pc/"4";"2"};
{\ar@/^/"5";"6"};{\ar@/_/"5";"6"};{\ar@/^1pc/"5";"3"};
{\ar@/^/"6";"7"};{\ar@/_/"6";"7"};{\ar@/^1pc/"6";"4"};
{\ar@/^/"7";"8"};{\ar@/_/"7";"8"};{\ar@/^1pc/"7";"5"};
{\ar@/^/"8";"1"};{\ar@/_.1pc/"8";"1"};{\ar@/^1pc/"8";"6"};{\ar@{->}"8";"2"};
\endxy
&
&
\xy
(-10,-20)*+{\text{\scriptsize{$8$}}}="a1";(-10,-10)*+{\text{\scriptsize{$3$}}}="a2";(-10,0)*+{\text{\scriptsize{$4$}}}="a3";(-10,10)*+{\text{\scriptsize{$6$}}}="a4";(-10,20)*+{\text{\scriptsize{$8$}}}="a5";
(0,-20)*+{\text{\scriptsize{$1$}}}="b1";(0,-10)*+{\text{\scriptsize{$2$}}}="b2";(0,0)*+{\text{\scriptsize{$5$}}}="b3";(0,10)*+{\text{\scriptsize{$7$}}}="b4";(0,20)*+{\text{\scriptsize{$1$}}}="b5";
(10,-20)*+{\text{\scriptsize{$8$}}}="c1";(10,-10)*+{\text{\scriptsize{$3$}}}="c2";(10,0)*+{\text{\scriptsize{$4$}}}="c3";(10,10)*+{\text{\scriptsize{$6$}}}="c4";(10,20)*+{\text{\scriptsize{$8$}}}="c5";
{\ar@{->}"a1";"b1"};{\ar@{<-}"b1";"c1"};{\ar@{<-}"a2";"b2"};{\ar@{->}"b2";"c2"};{\ar@{->}"a3";"b3"};{\ar@{<-}"b3";"c3"};{\ar@{->}"a4";"b4"};{\ar@{<-}"b4";"c4"};{\ar@{->}"a5";"b5"};{\ar@{<-}"b5";"c5"};
{\ar@{<-}"a1";"a2"};{\ar@{->}"b1";"b2"};{\ar@{<-}"c1";"c2"};
{\ar@{->}"a2";"a3"};{\ar@{<-}"b2";"b3"};{\ar@{->}"c2";"c3"};
{\ar@{<-}"a3";"a4"};{\ar@{<-}"b3";"b4"};{\ar@{<-}"c3";"c4"};
{\ar@{<-}"a4";"a5"};{\ar@{<-}"b4";"b5"};{\ar@{<-}"c4";"c5"};
{\ar@{->}"b3";"a4"};{\ar@{->}"b3";"c4"};
{\ar@{->}"b4";"a5"};{\ar@{->}"b4";"c5"};
\endxy 
&
\xy (-18.478,7.654)*+{\text{\scriptsize{$8$}}}="8";(-18.478,-7.654)*+{\text{\scriptsize{$7$}}}="7";(-7.654,-18.478)*+{\text{\scriptsize{$6$}}}="6";(7.654,-18.478)*+{\text{\scriptsize{$5$}}}="5";(18.478,-7.654)*+{\text{\scriptsize{$4$}}}="4";(18.478,7.654)*+{\text{\scriptsize{$3$}}}="3";(7.654,18.478)*+{\text{\scriptsize{$2$}}}="2";(-7.654,18.478)*+{\text{\scriptsize{$1$}}}="1";
{\ar@{->}"1";"2"};{\ar@/^1pc/"1";"7"};
{\ar@/^/"2";"3"};{\ar@/_/"2";"3"};
{\ar@{->}"3";"4"};{\ar@{->}"3";"8"};
{\ar@/^/"4";"5"};{\ar@/_/"4";"5"};
{\ar@/^/"5";"6"};{\ar@/_/"5";"6"};{\ar@{->}"5";"2"};
{\ar@/^/"6";"7"};{\ar@/_/"6";"7"};{\ar@/^1pc/"6";"4"};
{\ar@/^/"7";"8"};{\ar@/_/"7";"8"};{\ar@/^1pc/"7";"5"};
{\ar@/^/"8";"1"};{\ar@/_/"8";"1"};{\ar@/^1pc/"8";"6"};
\endxy
\\
Y^{4,2} & & & Y^{4,2} & 
\end{array}$$
\caption{Some examples of $Y^{p,q}$ quivers.}
\label{Ypq figure}
\end{figure}
\section{Impressions} \label{Impressions}

\subsection{Definition and utility}

The definition we introduce in this section, called an impression, will serve as our main tool for analyzing square superpotential algebras.  Recall that an algebra $A$ is representable if there is an algebra monomorphism $A \hookrightarrow M_d(B)$ for some commutative noetherian algebra $B$.  An impression may be thought of as a way of placing (commutative) coordinates within a representable algebra.

\begin{Definition and Lemma} \label{impression}
Let $A$ be a representable algebra over an algebraically closed field $k$ and denote by $Z$ its center.  
Suppose there exists a commutative finitely-generated $k$-algebra $B$, an open dense subset $U \subseteq \operatorname{Max}B$, and an algebra monomorphism $\tau: A \rightarrow \operatorname{End}_B \left(B^d \right)$ with $d < \infty$, such that the composition with the evaluation map
$$\tau_{\mathfrak{q}}: A \stackrel{\tau}{\longrightarrow} \operatorname{End}_B\left( B^d \right) \longrightarrow \operatorname{End}_B\left( (B/\mathfrak{q})^d \right) \cong \operatorname{End}_k \left( k^d \right)$$
is a simple representation for each $\mathfrak{q} \in U$. Then
\begin{equation} \label{U}
Z \cong R := \left\{f \in B \ | \ f1_d \in \operatorname{im}\tau \right\} \subset B.
\end{equation}
If the map
\begin{equation} \label{U'} 
\operatorname{Max}B \stackrel{\phi}{\rightarrow} \operatorname{Max}R, \ \ \ \mathfrak{q} \mapsto \mathfrak{q} \cap R,
\end{equation}
is surjective then we call $(\tau,B)$ an \rm{impression} \textit{of $A$.}
\end{Definition and Lemma}

\begin{proof}
$k$ is algebraically closed, hence infinite, and $B$ is finitely-generated over $k$, and therefore $B/\mathfrak{q} \cong k$ for each $\mathfrak{q} \in \operatorname{Max}B$ by \cite[9.1.12]{MR}.  We first prove (\ref{U}).  Suppose $a \in Z$.  Identify $\operatorname{End}_B\left( (B/\mathfrak{q})^d \right) \cong M_d(B/\mathfrak{q})$, so $\tau_{\mathfrak{q}}(a) = \left( b_{ij}(\mathfrak{q}) \right)$ is a matrix with entries $b_{ij}(\mathfrak{q})$ in $k$.  For each $\mathfrak{q} \in U$, Shur's lemma implies $\tau_{\mathfrak{q}}(a) \in k1_d$.   Thus $b_{ij}(\mathfrak{q})= 0$ whenever $i \not = j$, and $b_{ii}(\mathfrak{q})=b_{jj}(\mathfrak{q})$ for each $i,j$.  Since $U$ is dense in $\operatorname{Max}B$ it follows that $b_{ij} \sim 0$ for $i \not = j$ and $b_{ii} \sim b_{jj}$ for each $i,j$, that is, $\tau(a) = b_{11} 1_d$.

Conversely, suppose $f1_d \in \operatorname{im}\tau$, say $\tau(a) = f1_d$ for some $a \in A$.  For any $b \in A$, $\tau(ab-ba) = \tau(a)\tau(b)-\tau(b)\tau(a)=0$, so $ab=ba$ since $\tau$ is a monomorphism, and thus $a \in Z$.

We now prove $\phi$ is well-defined.  Recall our standing assumption that $A$ and $B$ are unital.  Since $\tau_{\mathfrak{q}}$ is a simple representation for each $\mathfrak{q} \in U$, $\tau(1_A) = 1_d$.  Thus $1_B \in R$, so for any $\mathfrak{q} \in \operatorname{Max}B$ the composition $\psi: R \hookrightarrow B \rightarrow B/\mathfrak{q} \cong k$ is an epimorphism.  It follows that $R/ \operatorname{ker}\psi \cong k$, and so $\mathfrak{q} \cap R = \operatorname{ker}\psi \in \operatorname{Max}R$ since $R$ is a unital commutative ring.
\end{proof}

We call $(\tau,B)$ a \textit{pre-impression} of $A$ if $\phi$ is not assumed to be surjective.  As we will see, a (pre-)impression of an algebra $A$ may be useful because
\begin{itemize}
  \item it determines the center of $A$;
  \item it may enable symplectic geometric concepts to be related to the representation theory of $A$;
  \item if $B$ is prime and $A$ is noetherian and module-finite over its center, then its impression explicitly determines all simple $A$-modules of maximal $k$-dimension up to isomorphism.
\end{itemize}

\begin{Lemma} \label{morphism of var}
If $B$ is reduced and $Z$ is finitely-generated then $\phi: \operatorname{Max}B \rightarrow \operatorname{Max}R$ is a morphism of varieties.
\end{Lemma}

\begin{proof}
By Hilbert's Nullstellensatz \cite[Corollary 1.10]{E}, the algebra monomorphism $R \hookrightarrow B$ induces the morphism $\phi$ defined by $\mathfrak{q} \mapsto \mathfrak{q} \cap R$ since $B$ is reduced, whence $Z$ is reduced, and $k$ is algebraically closed.
\end{proof}

Recall that a ring $R$ is prime if its zero ideal is prime.

\begin{Lemma} \label{prime} 
Let $(\tau, B)$ be a pre-impression of an algebra $A$.  If $B$ is a prime ring then $A$ and its center $Z$ are both prime rings.
\end{Lemma}

\begin{proof}
Since $B$ is prime, $M_d(B) \cong \operatorname{End}_B(B^d)$ is prime \cite[Proposition 10.20]{Lam}, and thus $A$ is prime since $\tau$ is an algebra monomorphism.  $Z$ is an integral domain since by (\ref{U}) it is a subring of an integral domain.
\end{proof}

We call a simple module (resp.\ representation) of maximal $k$-dimension a \textit{large module} (resp.\ \textit{large representation}).  Under suitable hypotheses given in section \ref{Overview}, the large modules are parameterized by the smooth locus of the algebras center.

\begin{Lemma} \label{maximal}
For each $\mathfrak{q} \in U$, the composition $\tau_{\mathfrak{q}}$ is a large representation of $A$.
\end{Lemma}

\begin{proof}
Let $n$ be the maximal $k$-dimension of the simple $A$-modules; we claim that $d = n$.  Since $\tau_{\mathfrak{q}}$ is simple, $n \geq d$.  Conversely, let $V$ be a large $A$-module and set $\mathfrak{p}:= \operatorname{ann}_AV$.  $A$ is a PI ring of PI degree $d' \leq d$ since $A \cong \tau(A) \subset M_d(B)$ and $B$ is commutative \cite[13.3.3.ii]{MR}, so $A/\mathfrak{p}$ is a primitive PI ring of PI degree $r \leq d'$ \cite[13.7.2.i]{MR}, and thus $A/\mathfrak{p}$ is a central simple algebra isomorphic to $M_r(\operatorname{End}_{A/\mathfrak{p}}V)$ with center $\operatorname{End}_{A/\mathfrak{p}}V$ \cite[13.3.8]{MR}.  But $\operatorname{End}_{A/\mathfrak{p}}V \cong k$ since $k$ is algebraically closed.  It follows that there is a faithful simple representation $M_r(k) \cong A/\mathfrak{p} \hookrightarrow M_n(k)$, so $r = n$, whence $n = r \leq d$.
\end{proof}

\begin{Proposition} \label{(B/r)^d} 
Let $(\tau, B)$ be an impression of a finitely-generated algebra $A$ module-finite over its center, with $B$ prime.  If $V$ is a large $A$-module then there is some $\mathfrak{q} \in \operatorname{Max}B$ such that $V \cong (B/\mathfrak{q})^d$, where $av := \tau_{\mathfrak{q}}(a) v$.
\end{Proposition}

\begin{proof} 
First note that $A$ and its center $Z$ are noetherian since $A$ is finitely-generated over $k$ and module-finite over $Z$ \cite[13.9.7]{MR}, and $A$ is prime by Lemma \ref{prime}.

Let $V$ be a large $A$-module and $\mathfrak{p} = \operatorname{ann}_AV$ its annihilator.  
Since $A$ is module-finite over its center, $\mathfrak{m} := \mathfrak{p} \cap Z$ is a maximal ideal of $Z$ \cite[Theorem 4.2.2(2)]{Smith}.  
Since $\phi$ is surjective there is some $\mathfrak{q} \in \operatorname{Max}B$ such that $\phi(\mathfrak{q})= \mathfrak{m}$.  
But then $\mathfrak{m} = \phi(\mathfrak{q})=\operatorname{ann}_Z\left((B/\mathfrak{q})^d \right)$. 
 
Let $W$ be a nonzero simple submodule of $(B/\mathfrak{q})^d$ and $\mathfrak{p}'$ its annihilator.  
Then $\operatorname{ann}_ZW \supseteq \operatorname{ann}_Z(B/\mathfrak{q})^d =m$, and so since $\mathfrak{m}$ is a maximal ideal it must be that $\mathfrak{p}' \cap Z = \operatorname{ann}_ZW = \mathfrak{m}$.  
Since $\mathfrak{m}$ is in the Azumaya locus of $A$ (that is, $A_{\mathfrak{m}}$ is Azumaya over $Z_{\mathfrak{m}}$) and $A$ is prime, noetherian, and module-finite over its finitely-generated center, $W$ must be of maximal $k$-dimension \cite[Proposition 3.1]{BGood}, namely $d$, so $W \cong (B/\mathfrak{q})^d$.  
Moreover, since $A_{\mathfrak{m}}$ is Azumaya over $Z_{\mathfrak{m}}$ and $\mathfrak{m}$ is a maximal ideal of $Z_{\mathfrak{m}}$, $A_{\mathfrak{m}}/\mathfrak{m}A_{\mathfrak{m}}$ is a central simple algebra \cite[Proposition 13.7.11]{MR}.
But $A_{\mathfrak{m}}/\mathfrak{m}A_{\mathfrak{m}} \cong A/\mathfrak{m}A$ as algebras \cite[Lemma 13.7.12]{MR}, so $A/\mathfrak{m}A$ is central simple.
Therefore, since $V$ and $W$ are both simple modules over $A/\mathfrak{m}A$, we have $V \cong W$ as $A/\mathfrak{m}A$-modules.  
Since $V$ and $W$ are both annihilated by $\mathfrak{m}$, we also have $V \cong W$ as $A$-modules.
Therefore $V \cong W \cong (B/\mathfrak{q})^d$ as $A$-modules.
\end{proof}

\begin{Remark} \label{prime PI} \rm{ 
Suppose $B$ is a field and $(\tau,B)$ is an impression of $A$.  Then $\tau: A \stackrel{\cong}{\rightarrow} \operatorname{End}_B(B^d)$ is an isomorphism, $U$ consists of the zero ideal, and $Z \cong B$.  
} \end{Remark}

\begin{Remark} \label{order} \rm{
Supposing $R$ is a normal noetherian domain, an $R$-order is a particular type of prime PI ring: an $R$-order is an $R$-algebra $A \subset M_n(\operatorname{Frac}(R))$ that is a finitely-generated $R$-module satisfying $\operatorname{Frac}(R) \otimes_R A \cong M_n(\operatorname{Frac}(R))$, and thus $R$ and $A$ are in some sense birationally equivalent.  We will show in section \ref{Endomorphism rings} that any square superpotential algebra $A$ with center $Z$ is indeed a $Z$-order, but the impression ring $B$ will be quite different from $\operatorname{Frac}(Z)$. 
}\end{Remark}

\begin{Notation} \label{notation} \rm{
Let $A = kQ/I$ be a quiver algebra and let $\tau: A \rightarrow \operatorname{End}_B(B^d)$ be an algebra homomorphism.  For each $i,j \in Q_0$ set $d_i := \operatorname{rank}\tau(e_i)$ and define the $k$-linear map
$$\bar{\tau}: e_jAe_i \rightarrow \operatorname{Hom}_B\left(B^{d_i}, B^{d_j}\right),$$
where for each $a \in e_jAe_i$, $\bar{\tau}(a)$ is the restriction of $\tau(a)$ to 
$$B^{d_i} \cong \tau(e_i)B^{d} \rightarrow B^{d_j} \cong \tau(e_j)B^{d}.$$

Our interest here will be in the case $d_i = d_j =1$, so $\bar{\tau}\left(e_jAe_i\right) \subseteq \operatorname{Hom}_B(B,B) \cong B$.  If $a \in e_jkQe_i$ is a representative of an element $a+I \in A$, then set $\bar{\tau}(a):= \bar{\tau}(a+I)$.
}\end{Notation}

\subsection{The case when $\bar{\tau}(e_iAe_i) = \bar{\tau}(e_jAe_j) \subset B$} \label{Dimension vectors and noetherian centers}

Throughout, denote by $E_{ji}$ the matrix with a $1$ in the $(ji)$-th slot and zeros elsewhere.  
In this section we collect results for a general quiver algebra $A = kQ/I$ that admits a pre-impression $(\tau,B)$, where $d = |Q_0|$, $\tau(e_i) = E_{ii}$ for each $i \in Q_0$, and $\bar{\tau}(e_iAe_i) = \bar{\tau}(e_jAe_j) \subset B$ for each $i,j \in Q_0$.  In the next section we will show that square superpotential algebras admit such impressions.  
We identify $\operatorname{End}_B(B^d) \cong M_d(B)$.

\subsubsection{Noetherianity} \label{Noetherianity}

\begin{Lemma} \label{bartau}
Let $A=kQ/I$ be a quiver algebra, and suppose there exists an algebra monomorphism $\tau: A \rightarrow \operatorname{End}_B\left( B^{|Q_0|} \right)$ such that 
\begin{equation} \label{ei}
\tau(e_i) = E_{ii} \ \ \text{ and } \ \ \bar{\tau}(e_iAe_i) = \bar{\tau}(e_jAe_j) \subset B \ \ \text{ for each } i,j \in Q_0.
\end{equation}
Then for each $i \in Q_0$, there is an algebra isomorphism
$$e_iAe_i \cong \bar{\tau}(e_iAe_i) \subset B.$$
Therefore there is an isomorphism of corner rings $e_iAe_i \cong e_jAe_j$.
\end{Lemma}

\begin{proof}
For $c,d \in e_iAe_i$ we have
$$\bar{\tau}(dc)E_{ii} = \tau(dc) = \tau(d)\tau(c) = \bar{\tau}(d)E_{ii}\bar{\tau}(c)E_{ii}=\bar{\tau}(d)\bar{\tau}(c)E_{ii},$$
so $\bar{\tau}(dc)= \bar{\tau}(d)\bar{\tau}(c)$, and similarly $\bar{\tau}(c+d)= \bar{\tau}(c)+ \bar{\tau}(d)$.  If $c \in e_iAe_i$ satisfies $\bar{\tau}(c)=0$ then $\tau(c) = \bar{\tau}(c)E_{ii}=0$, whence $c = 0$, and so the restriction $\bar{\tau}: e_iAe_i \rightarrow B$ is also an algebra monomorphism.  It follows that $e_iAe_i \cong \bar{\tau}(e_iAe_i)$ as algebras.
\end{proof}

\begin{Lemma} \label{without cyclic proper subpaths}
Let $A=kQ/I$ be a quiver algebra, and suppose there exists an algebra monomorphism $\tau: A \rightarrow \operatorname{End}_B\left( B^{|Q_0|} \right)$ such that (\ref{ei}) holds.  Then each corner ring $e_iAe_i$ is a finitely-generated subalgebra.
\end{Lemma}

\begin{proof}
The corner ring $e_iAe_i$ is generated by representative cycles $c \in e_ikQe_i$ such that $c \not = b_2e_ib_1$ for any cycles $b_1, b_2$ of nonzero length.  We claim that any cycle $c = c_2dc_1 \in e_ikQe_i$, with a cyclic proper subpath $d \in e_jkQe_j$, is equal to some cycle $b_2e_ib_1 \in e_ikQe_i$ modulo $I$.  Since $\bar{\tau}(e_iAe_i) = \bar{\tau}(e_jAe_j)$, there exists a $d' \in e_ikQe_i$ such that $\bar{\tau}(d')=\bar{\tau}(d)$.  Since $\operatorname{rank}\tau(e_{\ell}) =1$ for each $\ell \in Q_0$, we have
$$\bar{\tau}(d'c_2c_1)E_{ii} = \tau(d'c_2c_1) = \tau(d')\tau(c_2)\tau(c_1)$$
$$= \bar{\tau}(d') E_{ii} \bar{\tau}(c_2)E_{ij} \bar{\tau}(c_1)E_{ji} = \bar{\tau}(d') \bar{\tau}(c_2)\bar{\tau}(c_1)E_{ii},$$
so $\bar{\tau}(d'c_2c_1) = \bar{\tau}(d') \bar{\tau}(c_2) \bar{\tau}(c_1)$.  Similarly $\bar{\tau}(c_2dc_1) = \bar{\tau}(c_2) \bar{\tau}(d) \bar{\tau}(c_1)$.  But then 
$$\bar{\tau}(d'c_2c_1) = \bar{\tau}(c_2dc_1),$$ 
and the claim follows since $\bar{\tau}: e_iAe_i \rightarrow B$ is an algebra monomorphism by Lemma \ref{bartau}.  $e_iAe_i$ is therefore generated by cycles without cyclic proper subpaths, and the lemma follows since $|Q_0|< \infty$.
\end{proof}

\begin{Theorem} \label{center}
Let $A=kQ/I$ be a quiver algebra that admits a pre-impression $(\tau,B)$ such that (\ref{ei}) holds.  
Then $A$ and its center $Z$ are noetherian rings, $A$ is a finitely-generated $Z$-module, and
\begin{equation} \label{Z}
Z = k \left[ \sum_{i \in Q_0} \gamma_i \in \bigoplus_{i \in Q_0}e_iAe_i \ | \ \bar{\tau}(\gamma_i)= \bar{\tau}(\gamma_j) \text{ for each } i,j \in Q_0 \right].
\end{equation}
In particular, $Z \cong Ze_i = e_iAe_i$ for each $i \in Q_0$.
\end{Theorem}

\begin{proof}
(\ref{Z}) follows from Lemma \ref{impression} (\ref{U}).  In particular, for each $i \in Q_0$ there is an algebra epimorphism $Z \rightarrow Ze_i$ given by $z \mapsto ze_i$.  This map is injective: first note that $ze_j = ze_j^2 = e_jze_j$ for each $j \in Q_0$ since $z \in Z$.  Suppose $ze_i = 0$; then $\bar{\tau}(e_ize_i) = 0$, so $\bar{\tau}(e_jze_j) = 0$ for each $j \in Q_0$, so $\tau(z) = 0$ since $z = z\sum_{j \in Q_0}e_j = \sum_{j \in Q_0}e_jze_j$, whence $z = 0$ since $\tau$ is injective.  Therefore $Z \cong Ze_i$.  Furthermore, (\ref{Z}) implies $e_iAe_i \subseteq Ze_i$ since by assumption $\bar{\tau}(e_iAe_i) = \bar{\tau}(e_jAe_j)$ for each $i,j \in Q_0$, so $e_iAe_i = Ze_i$.  By Lemma \ref{without cyclic proper subpaths}, $Z \cong Ze_i = e_iAe_i$ is finitely-generated, and so $Z$ is noetherian.

We claim that $_ZA$ is generated by all paths in $Q$ of length $\leq m:= |Q_0|$, modulo $I$.  Let $\ell(p)$ denote the length of a path $p \in kQ$.  If $a \in Q_{>m}$ is a path, then $a$ must have a (non-vertex) cyclic subpath, say $a = b'cb$ where $c$ is a cycle in $Q$ and $b,b'$ are paths.  Since $c$ is not a vertex and $\ell(a) < \infty$, we have that $\ell(b'b)< \ell(a)$.  Since $e_{\operatorname{t}(c)}Ae_{\operatorname{t}(c)} = Ze_{\operatorname{t}(c)}$, there exists a $\tilde{c} \in Z$ such that $\tilde{c}e_{\operatorname{t}(c)}=c + I$, so 
$$a +I= b'cb +I= b'\tilde{c}e_{\operatorname{t}(c)}b +I =b'\tilde{c}b+I= \tilde{c}b'b+I.$$
But $b'b \in kQ$ is a path representative of $b'b +I \in A$, so we may repeat this process with $b'b$ in place of $a$, and then do so a finite number of times until $a +I \in Z(b'' +I)$ with $\ell(b'') \leq m$.  We may then extend this argument $k$-linearly to $kQ$, proving our claim.  $A$ is therefore module-finite over its center $Z$.  Since $A$ is also finitely-generated, $A$ is noetherian by the Artin-Tate lemma \cite[4.2.1]{Smith}.
\end{proof}

The following corollary will be used in later sections.

\begin{Corollary} \label{ba not 0} 
Suppose a quiver algebra $A=kQ/I$ admits a pre-impression $(\tau,B)$, where $B$ is an integral domain and (\ref{ei}) holds.  If $a \in e_iA$ and $b \in Ae_i$ are nonzero, then $ba$ is nonzero as well.
\end{Corollary}

\begin{proof} 
Suppose $ba =0$.  By Lemma \ref{prime} $A$ is prime and by Theorem \ref{center} $e_iAe_i = Ze_i$ for each $i \in Q_0$, so $bra \not = 0$ for some $r \in e_iAe_i = Ze_i$.  But then there is a $z \in Z$ such that $ze_i = r$, so $bra = bza = zba=0$, a contradiction.
\end{proof}

\subsubsection{Large modules} \label{Large modules}
\ \\
Suppose $A = kQ/I$ is a quiver algebra that admits a pre-impression $(\tau,B)$ with $d = |Q_0|$ and $\bar{\tau}(e_iAe_i) = \bar{\tau}(e_jAe_j) \subset B$ for each $i,j \in Q_0$.  By Lemma \ref{maximal}, the large $A$-modules (that is, the simple $A$-modules of maximal $k$-dimension) have $k$-dimension $|Q_0|$.  In this subsection we give algebraic and homological characterizations of these modules.

Recall that if a $k$-algebra $A$ is finitely-generated and module-finite over its center $Z$, and $V$ is a simple $A$-module, then $\operatorname{ann}_ZV$ will be a maximal ideal of $Z$ \cite[Theorem 4.2.2]{Smith}.

\begin{Lemma} \label{Azumaya locus of square algebras}  
Suppose a quiver algebra $A=kQ/I$ admits a pre-impression $(\tau,B)$, where $B$ is an integral domain and (\ref{ei}) holds.  Then a simple $A$-module $V$ is large if and only if $\operatorname{ann}_ZV \in \operatorname{Max}Z$ is contained in the Azumaya locus of $A$.
\end{Lemma}

\begin{proof} 
$A$ and $Z$ are prime noetherian algebras and $A$ is a finitely-generated $Z$-module by Theorem \ref{center} and Lemma \ref{prime}.  The lemma then follows from \cite[Proposition 3.1]{BGood} since $k$ is assumed to be algebraically closed.
\end{proof}

We will consider the Ore localizations $A_{\mathfrak{m}}:= Z_{\mathfrak{m}} \otimes_Z A$ with $\mathfrak{m} \in \operatorname{Max}Z$.  When $\mathfrak{m}$ is in the Azumaya locus, $A_{\mathfrak{m}}$ is local with unique maximal ideal $\mathfrak{m}_{\mathfrak{m}}A_{\mathfrak{m}}$ \cite[13.7.9]{MR}.  A simple $A$-module $V$ can be localized to an $A_{\mathfrak{m}}$-module by setting $V_{\mathfrak{m}}:= Z_{\mathfrak{m}}/\mathfrak{m}_{\mathfrak{m}} \otimes_k V$, and is only nonzero if $\mathfrak{m} = \operatorname{ann}_ZV$.  There is an obvious $A$-module isomorphism $\phi: V \rightarrow V_{\mathfrak{m}}$ defined by $\phi(v) = 1 \otimes v$,\footnote{$\phi$ is injective since $\frac 1t \cdot Z_{\mathfrak{m}}/\mathfrak{m}_{\mathfrak{m}} \otimes tv = 1 \otimes v =\phi(v)=0$ implies $tv= 0$, which implies $v = 0$ since $t \not \in \mathfrak{m} = \operatorname{ann}_ZV$.} so $V$ may be viewed as an $A_{\mathfrak{m}}$-module by setting $b v := \phi^{-1}(b \phi(v))$, and the inverse map $\phi^{-1}: V_{\mathfrak{m}} \rightarrow V$ is then an $A_{\mathfrak{m}}$-module isomorphism since $b \phi^{-1}(w) = bv = \phi^{-1}(b \phi(v)) = \phi^{-1}(bw)$.  $V$ and $V_{\mathfrak{m}}$ are thus isomorphic modules over both $A$ and $A_{\mathfrak{m}}$.

\begin{Lemma} \label{A/p cong V} 
Let $A$ be a prime noetherian algebra, module-finite over its center, let $V$ be a large $A$-module with annihilator $\mathfrak{p}$, and set $\mathfrak{m} := \mathfrak{p} \cap Z$.  Then there are $A$-module and $A_{\mathfrak{m}}$-module isomorphisms
\begin{equation} \label{V cong V_m}
A_{\mathfrak{m}}/\mathfrak{p}_{\mathfrak{m}} \cong A/\mathfrak{p} \cong V^{\oplus d} \cong \left( V_{\mathfrak{m}} \right)^{\oplus d},
\end{equation}
where $d := \operatorname{dim}_k(V)$.
This holds in particular if $A= kQ/I$ is a quiver algebra that admits a pre-impression $(\tau,B)$, where $B$ is an integral domain and (\ref{ei}) holds; in this case $d = |Q_0|$.
\end{Lemma}

\begin{proof} First note that the PI-degree of $A$ is $d$ \cite[3.1.a]{BGood}.  The factor $A/\mathfrak{p}$ is then a primitive PI ring, so by Kaplansky's theorem \cite[13.3.8]{MR} it is a central simple algebra whose only simple is $V$, so the PI-degree of $A/\mathfrak{p}$ is also $d$, and thus it has dimension $d^2$ over its center.  By the Artin-Wedderburn theorem there is an isomorphism of $A$-modules, $A/\mathfrak{p} \cong V^{\oplus d}$.\footnote{Of course there is a maximal \textit{left} ideal $\mathfrak{r}$ such that $V \cong A/\mathfrak{r}$ as $A$-modules, namely $\mathfrak{r} = \operatorname{ker}\phi_v$ where $\phi_v: A \rightarrow V$ is given by $a \mapsto av$.}

For the special case, $A$ is noetherian and module-finite over its center by Theorem \ref{center}, prime by Lemma \ref{prime}, and $d=|Q_0|$ by Lemma \ref{maximal}.
\end{proof}

Recall that a ring is semiperfect if every finitely-generated left (right) module $V$ admits a projective cover $P$, that is, there is an epimorphism $P \rightarrow V$ such that, for any submodule $L \subset P$, $\operatorname{ker}\phi + L=P$ implies $L=P$.  If a projective resolution is constructed from projective covers then its length will give the precise projective dimension rather than just an upper bound, and in (\ref{proj cover}) we determine the (unique) projective covers of the large $A_{\mathfrak{m}}$-modules.  Also, recall that a set of idempotents in a ring $S$ is basic if it is a complete set of orthogonal idempotents $e_1, \ldots, e_n$ such that $Se_1, \ldots, Se_n$ is a complete irredundant set of representatives of the $S$-modules of the form $Se$ for some primitive idempotent $e$ \cite[section 27]{AF}.

\begin{Proposition} \label{strange} 
Suppose a quiver algebra $A=kQ/I$ admits a pre-impression $(\tau,B)$, where $B$ is an integral domain and (\ref{ei}) holds.  
Further suppose $I \subset kQ_{\geq 2}$.
Let $V$ be a large $A$-module, and set $\mathfrak{p} := \operatorname{ann}_AV$, $\mathfrak{m} := \mathfrak{p} \cap Z \in \operatorname{Max}Z$.  
Then the localization $A_{\mathfrak{m}}$ is semiperfect and
$$\begin{array}{ccc}
A_{\mathfrak{m}}e_i \cong A_{\mathfrak{m}} e_j, & \ \ \ \ \ & \mathfrak{p}_{\mathfrak{m}} e_i \cong \mathfrak{p}_{\mathfrak{m}} e_j,\\
Ae_i \not \cong Ae_j, & \ \ \ \ \ & \mathfrak{p}e_i \not \cong \mathfrak{p} e_j, 
\end{array}$$
while as $A$-modules,
\begin{equation} \label{bigoplus}
Ae_i/\mathfrak{p}e_i \cong Ae_j/\mathfrak{p}e_j \cong V.
\end{equation}
Consequently, any single vertex forms a basic set of idempotents for $A_{\mathfrak{m}}$, while the set of all vertices forms a basic set for $A$.
\end{Proposition}

\begin{proof}
By Lemma \ref{Azumaya locus of square algebras}, $\mathfrak{m}$ is in the Azumaya locus of $A$, so $A_{\mathfrak{m}}$ contains only one primitive ideal, namely $\mathfrak{p}_{\mathfrak{m}}$, and thus the Jacobson radical of $A_{\mathfrak{m}}$ is $J=\mathfrak{p}_{\mathfrak{m}}$ \cite[13.7.5,9]{MR}.  Moreover, $A_{\mathfrak{m}}$ has a complete set of orthogonal idempotents $e_1, \ldots, e_n$, and for each $i \in Q_0$, the corner ring $e_i A_{\mathfrak{m}} e_i$ is local: 
\begin{equation} \label{local} e_i A_{\mathfrak{m}} e_i = Z_{\mathfrak{m}} \otimes_Z e_i A e_i \cong Z_{\mathfrak{m}} \otimes_Z Z \cong Z_{\mathfrak{m}}. \end{equation}
It follows \cite[Theorem 27.6]{AF} that $A_{\mathfrak{m}}$ is semiperfect and the set 
$$A_{\mathfrak{m}} e_1/\mathfrak{p}_{\mathfrak{m}}e_1 , \ldots, A_{\mathfrak{m}}e_n/\mathfrak{p}_{\mathfrak{m}}e_n$$
is the set of all simple $A_{\mathfrak{m}}$-modules, with
$$A_{\mathfrak{m}}/\mathfrak{p}_{\mathfrak{m}} = A_{\mathfrak{m}} e_1/\mathfrak{p}_{\mathfrak{m}}e_1 \oplus \cdots \oplus A_{\mathfrak{m}}e_n/\mathfrak{p}_{\mathfrak{m}}e_n.$$
Since $A_{\mathfrak{m}}$ is Azumaya there is only one simple $A_{\mathfrak{m}}$-module,\footnote{Indeed, since $A_{\mathfrak{m}}$ is Azumaya, any simple $A_{\mathfrak{m}}$ has annihilator $\mathfrak{p}_{\mathfrak{m}}$.  
Thus any simple $A_{\mathfrak{m}}$-module is also a simple module over $A_{\mathfrak{m}}/\mathfrak{p}_{\mathfrak{m}} = A/\mathfrak{p}$.
But $A$ admits an embedding into a matrix ring, so $A/\mathfrak{p}$ is primitive PI, and thus a central simple algebra by Kaplansky's theorem.
Therefore $A/\mathfrak{p}$, hence $A_{\mathfrak{m}}$, has only one simple module up to isomorphism.} so
\begin{equation} \label{iso V_m}
A_{\mathfrak{m}}e_i/\mathfrak{p}_{\mathfrak{m}} e_i \cong A_{\mathfrak{m}} e_j/\mathfrak{p}_{\mathfrak{m}}e_j \cong V_{\mathfrak{m}}.
\end{equation}

The following characterizes projective covers \cite[27.13]{AF}: Suppose $S$ is a semiperfect ring with a basic set of idempotents $e_1, \ldots, e_n$ and Jacobson radical $J$, and let $M$ be a finitely-generated $S$-module.  Then if
$$M/JM \cong \left(Se_1/Je_1\right)^{(k_1)} \oplus \cdots \oplus \left(Se_n/Je_n \right)^{(k_n)},$$
there is a unique projective cover $Se_1^{(k_1)} \oplus \cdots \oplus Se_n^{(k_n)} \rightarrow M/JM \rightarrow 0$.  Consider the case $S = A_{\mathfrak{m}}$ and $M=A_{\mathfrak{m}}e$.  As mentioned above, $J(A_{\mathfrak{m}})=\mathfrak{p}_{\mathfrak{m}}$, so
\begin{equation} \label{proj cover}
A_{\mathfrak{m}} e_i \stackrel{\phi= \ \cdot 1}{\twoheadrightarrow} V_{\mathfrak{m}} \cong A_{\mathfrak{m}}e_i/\mathfrak{p}_{\mathfrak{m}}e_i
\end{equation}
is the unique projective cover of $V_{\mathfrak{m}}$.  Therefore by (\ref{iso V_m}), $\phi$ must factor through $A_{\mathfrak{m}} e_j$, so by symmetry
\begin{equation} \label{isom}
A_{\mathfrak{m}} e_i \cong A_{\mathfrak{m}} e_j.
\end{equation}
Of course, $Ae_i \not \cong Ae_j$ when $i \not = j$ (argument in \cite[p.\ 4]{CB}: otherwise there would be some $f \in e_iAe_j$ and $g \in e_jAe_i$ with $fg=e_i$, $gf=e_j$, so $e_i=fg \in Ae_jA$.
But by assumption $I \subset kQ_{\geq 2}$, and therefore $e_i \not \in Ae_jA$, a contradiction).  
We remark that $A_{\mathfrak{m}} e_i$ is indecomposable since its endomorphism ring $\operatorname{End}_{A_{\mathfrak{m}}}\left(A_{\mathfrak{m}} e_i \right) \cong e_i A_{\mathfrak{m}} e_i$ is local by (\ref{local}).  By \cite[Corollary 17.20]{AF} $J(A_{\mathfrak{m}})e_i$ is the unique maximal submodule of $A_{\mathfrak{m}} e_i$, and so by (\ref{isom}), 
$$\mathfrak{p}_{\mathfrak{m}} e_i = J(A_{\mathfrak{m}}) e_i \cong J(A_{\mathfrak{m}}) e_j = \mathfrak{p}_{\mathfrak{m}} e_j.$$

Now since $\mathfrak{m}= \operatorname{ann}_ZV$, it follows by (\ref{iso V_m}) that the following are isomorphic both as $A$-modules and $A_{\mathfrak{m}}$-modules:
$$Ae_i/\mathfrak{p}e_i \cong A_{\mathfrak{m}}e_i/\mathfrak{p}_{\mathfrak{m}} e_i \cong A_{\mathfrak{m}} e_j/\mathfrak{p}_{\mathfrak{m}}e_j \cong Ae_j/\mathfrak{p}e_j,$$
and these are also isomorphic to $V$ and $V_{\mathfrak{m}}$.
\end{proof}

Note that an alternative proof of (\ref{V cong V_m}), namely $A/\mathfrak{p} \cong V^{\oplus |Q_0|}$, is immediate from Proposition \ref{strange}.

\begin{Theorem} \label{projective dimensions} 
Suppose a quiver algebra $A=kQ/I$ admits an impression $(\tau,B)$, where $B$ is an integral domain and (\ref{ei}) holds.  Let $V$ be a large $A$-module, and set $\mathfrak{p} := \operatorname{ann}_AV$, $\mathfrak{m} := \mathfrak{p} \cap Z \in \operatorname{Max}Z$.  Then
\begin{equation} \label{equalities}
\operatorname{pd}_A(V) = \operatorname{pd}_{A_{\mathfrak{m}}}(V_{\mathfrak{m}}) = \operatorname{pd}_{A}(A/\mathfrak{p})=\operatorname{pd}_{A_{\mathfrak{m}}}(A_{\mathfrak{m}}/\mathfrak{p}_{\mathfrak{m}}) = \operatorname{pd}_{Z_{\mathfrak{m}}}\left(Z_{\mathfrak{m}}/\mathfrak{m}_{\mathfrak{m}} \right).
\end{equation}
\end{Theorem}

\begin{proof} 
(i) We first claim $\operatorname{pd}_{A_{\mathfrak{m}}}(A_{\mathfrak{m}}/\mathfrak{p}_{\mathfrak{m}}) \leq \operatorname{pd}_{Z_{\mathfrak{m}}}\left(Z_{\mathfrak{m}}/\mathfrak{m}_{\mathfrak{m}}\right)$.  Consider a projective resolution of the residue field $Z_{\mathfrak{m}}/\mathfrak{m}_{\mathfrak{m}}$ over the local ring $Z_{\mathfrak{m}}$,
\begin{equation} \label{sequence}
\cdots \longrightarrow \left( Z_{\mathfrak{m}} \right)^{\oplus n} \longrightarrow Z_{\mathfrak{m}} \stackrel{\cdot 1}{\longrightarrow} Z_{\mathfrak{m}}/\mathfrak{m}_{\mathfrak{m}} \longrightarrow 0.
\end{equation}
By Lemma \ref{Azumaya locus of square algebras}, $\mathfrak{m}$ is in the Azumaya locus of $A$, so the localization $A_{\mathfrak{m}}$ is an Azumaya algebra, and thus (by definition) $A_{\mathfrak{m}}$, and hence the direct summand $A_{\mathfrak{m}}e$ for any $e$ in a basic set of idempotents for $A_{\mathfrak{m}}$, is a free $Z_{\mathfrak{m}}$-module \cite[13.7.6]{MR}.  But then $A_{\mathfrak{m}}e$ is a flat $Z_{\mathfrak{m}}$-module as well, so the functor $ - \otimes_{Z_{\mathfrak{m}}}A_{\mathfrak{m}}e$ is exact.  Applying this functor to the resolution (\ref{sequence}) we obtain the exact sequence
\begin{equation} \label{sequence2}
\begin{array}{l}
\cdots \longrightarrow \left( Z_{\mathfrak{m}} \right)^{\oplus n} \otimes_{Z_{\mathfrak{m}}} A_{\mathfrak{m}}e \cong \left( A_{\mathfrak{m}}e \right)^{\oplus n} \longrightarrow Z_{\mathfrak{m}} \otimes_{Z_{\mathfrak{m}}} A_{\mathfrak{m}}e \cong A_{\mathfrak{m}}e \\
\stackrel{\cdot 1}{\longrightarrow} Z_{\mathfrak{m}}/\mathfrak{m}_{\mathfrak{m}} \otimes_{Z_{\mathfrak{m}}} A_{\mathfrak{m}}e \longrightarrow 0.
\end{array}
\end{equation}
The modules in this sequence are now over $Z_{\mathfrak{m}} \otimes_{Z_{\mathfrak{m}}} A_{\mathfrak{m}} \cong A_{\mathfrak{m}}$.  By \cite[13.7.9]{MR}, the ideal $\mathfrak{p}_{\mathfrak{m}} \subset A_{\mathfrak{m}}$ is generated by $\mathfrak{m}$, that is, $\mathfrak{p}_{\mathfrak{m}} = \mathfrak{m}A_{\mathfrak{m}}$, and so 
$$Z_{\mathfrak{m}}/\mathfrak{m}_{\mathfrak{m}} \otimes_{Z_{\mathfrak{m}}} A_{\mathfrak{m}}e \cong Z_{\mathfrak{m}} \otimes_{Z_{\mathfrak{m}}} A_{\mathfrak{m}}e/(\mathfrak{m}_{\mathfrak{m}}(A_{\mathfrak{m}}e)) \cong A_{\mathfrak{m}}e/\mathfrak{p}_{\mathfrak{m}}e.$$
Claim (i) then follows by the exactness of (\ref{sequence2}).

(ii) We now claim $\operatorname{pd}_{Z_{\mathfrak{m}}}\left(Z_{\mathfrak{m}}/\mathfrak{m}_{\mathfrak{m}} \right) \leq \operatorname{pd}_{A_{\mathfrak{m}}}\left( V_{\mathfrak{m}} \right)$.  Consider a projective resolution of $V_{\mathfrak{m}}$ over $A_{\mathfrak{m}}$,
\begin{equation} \label{sequence2.5}
\cdots \longrightarrow P_1 \stackrel{\delta_1}{\longrightarrow} P_0 \stackrel{\delta_0}{\longrightarrow} V_{\mathfrak{m}} \longrightarrow 0.
\end{equation}
Fix a vertex $e$ and consider the sequence of $eA_{\mathfrak{m}}e$-modules, 
\begin{equation} \label{sequence3}
\cdots \longrightarrow eP_1 \stackrel{\delta_1|_{eP_1}}{\longrightarrow} eP_0 \stackrel{\delta_0|_{eP_0}}{\longrightarrow} eV_{\mathfrak{m}} \longrightarrow 0.
\end{equation}
By Theorem \ref{center}, $eAe \cong Z$ as algebras, and by Proposition \ref{(B/r)^d} and (\ref{U}), $eV \cong Z/\mathfrak{m}$ as $eAe$-modules.  Thus $eA_{\mathfrak{m}}e \cong Z_{\mathfrak{m}}$ as algebras and $eV_{\mathfrak{m}} \cong Z_{\mathfrak{m}}/\mathfrak{m}_{\mathfrak{m}}$ as $eA_{\mathfrak{m}}e$-modules.  For each $i$ we have the following inclusions:

$\bullet$ $\operatorname{ker}\left(\delta_i|_{eP_i}\right) \subseteq e \left( \operatorname{ker} \delta_i \right)$:\\
If $v \in \operatorname{ker}\left( \delta_i|_{eP_i} \right)$ then $v \in eP_i$ and $\delta_i(v) = 0$, so $v \in \operatorname{ker}\delta_i \cap eP_i = e\left( \operatorname{ker}\delta_i \right)$.

$\bullet$ $\operatorname{ker}\left(\delta_i|_{eP_i}\right) \supseteq e \left( \operatorname{ker} \delta_i \right)$:\\
If $v \in e\left( \operatorname{ker}\delta_i \right)$ then $v \in eP_i$ and $\delta_i(v+w) = 0$ for some $w \in P_i$ satisfying $ew = 0$.  But $v \in eP_i$ implies $\delta_i(v) = \delta_i(ev) = e \delta_i(v) \in e P_{i-1}$, and similarly $\delta_i(w) \not \in eP_{i-1}$, so $\delta_i(v) + \delta_i(w) = \delta_i(v+w) = 0$ implies $\delta_i(v) = 0$, and thus $v \in \operatorname{ker}\left( \delta_i|_{eP_i} \right)$.

$\bullet$ $\operatorname{im}\left(\delta_i|_{eP_i}\right) \subseteq e \left( \operatorname{im} \delta_i \right)$:\\
If $v \in \operatorname{im}\left( \delta_i|_{eP_i} \right)$ then there is some $u \in eP_i$ such that $v = \delta_i(u) = \delta_i(eu) = e \delta_i(u) \in eP_{i-1}$, so $v \in \operatorname{im}\left( \delta_i \right) \cap eP_{i-1} = e \left( \operatorname{im} \delta_i \right)$.

$\bullet$ $\operatorname{im}\left(\delta_i|_{eP_i}\right) \supseteq e \left( \operatorname{im} \delta_i \right)$:\\
If $v \in e \left( \operatorname{im} \delta_i \right)$ then $v \in eP_{i-1}$ and $v+w = \delta_i(u)$ for some $w \in P_{i-1}$ satisfying $ew = 0$ and $u \in P_i$.  But then $v = e(v+w) = e\left( \delta_i(u) \right) = \delta_i(eu)$, so $v \in \operatorname{im}\left( \delta_i|_{eP_i} \right)$.

Since (\ref{sequence2.5}) is an exact sequence, it follows that for each $i$,
$$\operatorname{ker}\left( \delta_i|_{eP_i}\right) = e \left( \operatorname{ker} \delta_i \right) = e \left( \operatorname{im} \delta_{i+1} \right) = \operatorname{im}\left( \delta_{i+1}|_{eP_{i+1}} \right),$$
so (\ref{sequence3}) is also an exact sequence, and thus (\ref{sequence3}) is a projective resolution of $Z_{\mathfrak{m}}/\mathfrak{m}_{\mathfrak{m}} \cong eV_{\mathfrak{m}}$ over $Z_{\mathfrak{m}} \cong eA_{\mathfrak{m}}e$.

(iii) For any ring $S$ and family of $S$-modules $M_i$, $\operatorname{pd}_S\left(\bigoplus_i M_i \right) = \operatorname{sup}\{ \operatorname{pd}_S(M_i) \}$ \cite[Proposition 5.1.20]{R}.  Thus by Lemma \ref{A/p cong V},  
$$\operatorname{pd}_A(A/\mathfrak{p}) = \operatorname{pd}_A(V) \ \ \text{ and } \ \ \operatorname{pd}_{A_{\mathfrak{m}}}(A_{\mathfrak{m}}/\mathfrak{p}_{\mathfrak{m}}) = \operatorname{pd}_{A_{\mathfrak{m}}}(V_{\mathfrak{m}}).$$
Consequently, by claims (i) and (ii),
$$\operatorname{pd}_{Z_{\mathfrak{m}}}\left(Z_{\mathfrak{m}}/\mathfrak{m}_{\mathfrak{m}} \right) \leq \operatorname{pd}_{A_{\mathfrak{m}}}\left( V_{\mathfrak{m}} \right) = \operatorname{pd}_{A_{\mathfrak{m}}}( A_{\mathfrak{m}}/\mathfrak{p}_{\mathfrak{m}} ) \leq \operatorname{pd}_{Z_{\mathfrak{m}}}\left(Z_{\mathfrak{m}}/\mathfrak{m}_{\mathfrak{m}}\right),$$
so $\operatorname{pd}_{Z_{\mathfrak{m}}}\left( Z_{\mathfrak{m}}/\mathfrak{m}_{\mathfrak{m}} \right) = \operatorname{pd}_{A_{\mathfrak{m}}}(A_{\mathfrak{m}}/\mathfrak{p}_{\mathfrak{m}})$.

(iv) Finally, $\operatorname{pd}_A(V)= \operatorname{pd}_{A_{\mathfrak{m}}}(V_{\mathfrak{m}})$ since exactness is preserved under localization.\footnote{This follows since $Z_{\mathfrak{m}}$ is a projective $Z_{\mathfrak{m}}$-module, and by \cite[7.4.2.iii]{MR} $\operatorname{fd}_{Z}(Z_{\mathfrak{m}})= \operatorname{fd}_{Z_{\mathfrak{m}}}(Z_{\mathfrak{m}})=0$, so we may apply the exact functor $Z_{\mathfrak{m}} \otimes_{Z} -$ to a projective resolution of $V$ over $A$, giving a projective resolution of $Z_{\mathfrak{m}} \otimes_Z V \cong Z_{\mathfrak{m}}/\mathfrak{m}_{\mathfrak{m}} \otimes_Z V = V_{\mathfrak{m}}$ over $Z_{\mathfrak{m}} \otimes_Z A = A_{\mathfrak{m}}$.}
\end{proof}

\section{Impressions of square superpotential algebras} \label{Impressions of square superpotential algebras}

\subsection{An impression} \label{An impression}

In this section we give an explicit impression for all square superpotential algebras.  To do this, we first determine an algebra monomorphism $\tau: A \rightarrow \operatorname{End}_B\left( B^{|Q_0|} \right)$, and then we show that $\bar{\tau}(e_iAe_i) = \bar{\tau}(e_jAe_j) \subset B$ for each $i,j \in Q_0$ and apply the results of section \ref{Dimension vectors and noetherian centers}.  

\begin{Notation} \rm{
Let $A = kQ/\partial W$ be a square superpotential algebra with covering quiver $\widetilde{Q}$ and projection $\pi: \widetilde{Q} \rightarrow Q$.  For brevity we will write $p \sim p'$ in place of $p = p'$ modulo $\partial W$; similarly for $p,q \in k\widetilde{Q}$, we will write $p \sim q$ whenever $\pi(p) \sim \pi(q)$.  If $p$ is a path in $kQ$ then we will refer to $p + \partial W$ as a \textit{path} in $A$ since if $p' \sim p$ then clearly $p'$ must be a path as well.
} \end{Notation}

Throughout, set $B := k\left[x_1,x_2,y_1,y_2\right]$.  Recall that the underlying graph $\widetilde{Q}^{\circ}$ of the covering quiver $\pi^{-1}(Q) = \widetilde{Q}$ of $Q$ embeds into $\mathbb{R}^2$ as a square grid with vertex set $\mathbb{Z} \times \mathbb{Z}$, and with at most one diagonal arrow in each unit square.  For each $a \in \widetilde{Q}_1$, define $\bar{\tau}(a)$ to be the monomial corresponding to the orientation of $a$ given in figure \ref{labels}.
\begin{figure}
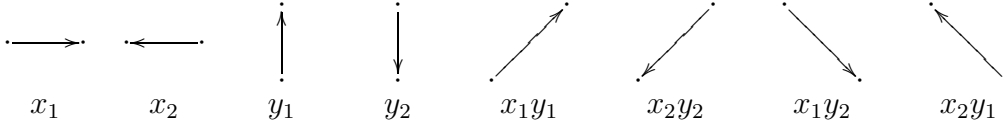

$$\begin{array}{cccccccc}
\xy 
(-5,0)*{\cdot}="1";(5,0)*{\cdot}="2";{\ar@{->}"1";"2"};
\endxy
&
\xy 
(-5,0)*{\cdot}="1";(5,0)*{\cdot}="2";{\ar@{<-}"1";"2"};
\endxy
&
\xy 
(0,-5)*{\cdot}="1";(0,5)*{\cdot}="2";{\ar@{->}"1";"2"};
\endxy
&
\xy 
(0,5)*{\cdot}="1";(0,-5)*{\cdot}="2";{\ar@{->}"1";"2"};
\endxy
&
\xy 
(-5,-5)*{\cdot}="1";(5,5)*{\cdot}="2";{\ar@{->}"1";"2"};
\endxy
&
\xy 
(-5,-5)*{\cdot}="1";(5,5)*{\cdot}="2";{\ar@{<-}"1";"2"};
\endxy
&
\xy 
(-5,5)*{\cdot}="1";(5,-5)*{\cdot}="2";{\ar@{->}"1";"2"};
\endxy
&
\xy 
(-5,5)*{\cdot}="1";(5,-5)*{\cdot}="2";{\ar@{<-}"1";"2"};
\endxy
\\
\ \ \ x_1 \ \ \ & \ \ \ x_2 \ \ \ & \ \ \ y_1 \ \ \ & \ \ \ y_2 \ \ \ & \ \ \ x_1y_1 \ \ \ & \ \ \ x_2y_2 \ \ \ & \ \ \ x_1y_2 \ \ \ & \ \ \ x_2y_1 \ \ \ 
\end{array}$$
\caption{A labeling of arrows in the quiver of a square superpotential algebra that specifies an impression.}
\label{labels}
\end{figure}

For each $a \in \widetilde{Q}_1$, set $\bar{\tau}(\pi(a)) := \bar{\tau}(a)$.  Let $E_{ji}$ denote the matrix with a $1$ in the $(ji)$-th slot and zeros elsewhere.  Define the $k$-algebra homomorphism
\begin{equation} \label{tau}
\tau: A \rightarrow M_{|Q_0|}(B) \cong \operatorname{End}_B\left(B^{|Q_0|} \right)
\end{equation}
on the generating set $Q_0 \cup Q_1$, by 
$$e_i \mapsto E_{ii} \ \text{ for each } \ i \in Q_0 \ \text{ and } \ a \mapsto \bar{\tau}(a) E_{\operatorname{h}(a),\operatorname{t}(a)} \ \text{ for each } \ a \in Q_1.$$
We will show that $(\tau,B)$ is an impression of $A$.  Note that $\tau$ is well-defined since the paths satisfy the same multiplication as the matrices $E_{ij}$, that is, $A$ is isomorphic to the matrix ring
$$A \cong \left[ \begin{array}{cccc} e_1Ae_1 & e_1Ae_2 & \cdots & e_1Ae_{|Q_0|} \\ e_2Ae_1 & e_2Ae_2 & & \\ \vdots & & \ddots & \\ e_{|Q_0|}Ae_1 & & & e_{|Q_0|}Ae_{|Q_0|} \end{array} \right],$$
and the labeling of arrows given in figure \ref{labels} is preserved under $\partial W$.  Also, note that the definition of $\bar{\tau}$ given above extends to the definition of $\bar{\tau}$ as a $k$-linear map given in notation \ref{notation}.

For the following lemma, denote a path $p \in kQe_i$ by its ordered monomial labeling in the non-commuting variables $\mathsf{x}_1,\mathsf{x}_2,\mathsf{y}_1,\mathsf{y}_2$.  If a subword $\mathsf{x}_{\alpha}\mathsf{y}_{\beta}$ corresponds to a diagonal arrow, then set $\mathsf{x}_{\alpha}\mathsf{y}_{\beta} = \mathsf{y}_{\beta}\mathsf{x}_{\alpha}$.  The proof is given in Appendix A.

\begin{Lemma} \label{appendix}
Consider a path $p=\mathsf{t}_n\cdots \mathsf{t}_2 \mathsf{t}_1 \mathsf{u} \mathsf{s}_{m} \cdots \mathsf{s}_1$
with $\mathsf{s}_{\ell}, \mathsf{t}_{\ell}, \mathsf{u} \in \{\mathsf{x}_1,\mathsf{x}_2,\mathsf{y}_1,\mathsf{y}_2 \}$.
Suppose there exists an arrow $a \not = \mathsf{u}$ whose head (resp.\ tail) is a vertex subpath of $\mathsf{t}_n \cdots \mathsf{t}_1$ (resp.\ $\mathsf{s}_{m} \cdots \mathsf{s}_1$) such that $u \mid \bar{\tau}(a)$ and $\bar{\tau}(a) \mid \bar{\tau}(p)$.
Then
\begin{equation} \label{t1t2}
p \sim \mathsf{t}_n \cdots \mathsf{t}_2(\mathsf{u} \mathsf{t}_1)\mathsf{s}_m \cdots \mathsf{s}_1 \ \ \text{ or } \ \ p \sim \mathsf{t}_n \cdots \mathsf{t}_3(\mathsf{u} \mathsf{t}_i \mathsf{t}_j)\mathsf{s}_m \cdots \mathsf{s}_1
\end{equation}
$$\left(\text{resp.\ } \ \ p \sim \mathsf{t}_n \cdots \mathsf{t}_1(\mathsf{s}_m \mathsf{u})\mathsf{s}_{m-1} \cdots \mathsf{s}_1 \ \ \text{ or } \ \ p \sim \mathsf{t}_n \cdots \mathsf{t}_1(\mathsf{s}_i \mathsf{s}_j \mathsf{u}) \mathsf{s}_{m-2} \cdots \mathsf{s}_1 \right),$$
where $i,j \in \{1,2\}$ (resp.\ $i,j \in \{m,m-1\}$) are distinct.
\end{Lemma}

Referring to the proof of Lemma \ref{appendix}, we remark that when $\mathsf{u} = \mathsf{y}_1$ and $\mathsf{t}_2 \mathsf{t}_1 = \mathsf{x}_{\alpha} \mathsf{y}_2$ is a diagonal arrow, we have $\mathsf{t}_2 \mathsf{t}_1 \mathsf{u} \sim \mathsf{u} \mathsf{t}_2 \mathsf{t}_1$, whereas when $\mathsf{u} = \mathsf{y}_1$ is a vertical arrow and $\mathsf{t}_2 \mathsf{t}_1 = \mathsf{y}_2 \mathsf{x}_{\alpha}$, we have $\mathsf{t}_2 \mathsf{t}_1 \mathsf{u} \sim \mathsf{u} \mathsf{t}_1 \mathsf{t}_2$.  
These two cases illustrate how the order of $\mathsf{t}_1$ and $\mathsf{t}_2$ in (\ref{t1t2}) may depend on the path $p$.
We also remark that the lemma will fail in general without the assumption on the existence of the arrow $a$.

Any square superpotential algebra $A$ admits a $\mathbb{Z}$-grading determined by $\tau$: the horizontal and vertical arrows (the first four arrows in figure \ref{labels}) have degree 1 while the diagonal arrows (the latter four arrows in figure \ref{labels}) have degree 2.  Clearly if $p$ and $p'$ are two paths and $p \sim p'$ then $p$ and $p'$ have the same degree.  The following two lemmas will be proved by induction on degree.

\begin{Lemma} \label{injective}
If $p$ and $p'$ are two paths in $Q$ with the same tail such that $\bar{\tau}(p) = \bar{\tau}(p')$, then $p \sim p'$.  Consequently the map $\tau$ in (\ref{tau}) is an algebra monomorphism.
\end{Lemma}

\begin{proof}
If $p$ has degree 1 or 2 then clearly $p \sim p'$.  
Suppose the assertion holds for paths of degree $< n$, and that $p$ has degree $n$.  
Further suppose $u \in \{x_1,x_2,y_1,y_2\}$ divides the $\bar{\tau}$-image of the leftmost arrow subpath $a$ of $p$.  
Since $\bar{\tau}(p) = \bar{\tau}(p')$ we have $\bar{\tau}(a) \mid \bar{\tau}(p')$.
Therefore by Lemma \ref{appendix} we can `commute' the leftmost arrow subpath of $p'$ whose $\bar{\tau}$-image is divisible by $u$ to the left, to form a path $p'' \sim p'$ whose leftmost arrow coincides with the leftmost arrow of $p$, and satisfies $\bar{\tau}(p'')= \bar{\tau}(p') = \bar{\tau}(p)$.  
The proof then follows by induction.

$\tau$ is injective: Suppose $p,p' \in A$ satisfy $\tau(p) = \tau(p')$.  Then the corresponding matrix entries must be equal, so we may assume $p,p' \in e_jAe_i$ for some $i,j \in Q_0$.  In this case, $\tau(p) = \tau(p')$ is equivalent to $\bar{\tau}(p)=\bar{\tau}(p')$.  
\end{proof}

The following lemma will be essential throughout.

\begin{Lemma} \label{star}
If $p$ and $p'$ are two paths in $Q$ with the same tail such that $\bar{\tau}(p) = m \bar{\tau}(p')$ for some monomial $m \in B$, then there exists a path $q \in e_{\operatorname{h}(p)}kQe_{\operatorname{h}(p')}$ such that $\bar{\tau}(q) = m$ and $p \sim q p'$.
\end{Lemma}

\begin{proof}
We proceed by induction.

First suppose $p'$ is an arrow, and so has degree 1 or 2, and suppose $u \in \{ x_1,x_2,y_1,y_2 \}$ divides $\bar{\tau}(p')$.  
Since $\bar{\tau}(p') \mid \bar{\tau}(p)$, Lemma \ref{appendix} (with $p' = a$) implies that we can `commute' the rightmost arrow subpath of $p$ whose $\bar{\tau}$-image is divisible by $u$ to the right, to form a path $qp' \sim p$.
Then $\bar{\tau}(q) = \bar{\tau}(p)/\bar{\tau}(p') = m$.

Now suppose the assertion holds for paths of degree $< n$, and that $p'$ has degree $n$.  
Let $p''$ be the subpath of $p'$ obtained by removing the leftmost arrow $b$ from $p'$.  
Since the degree of $p''$ is $< n$, by induction there is a path $q' \in e_{\operatorname{h}(p)}kQe_{\operatorname{h}(p'')}$ such that $\bar{\tau}(q') = \bar{\tau}(p)/\bar{\tau}(p'') = m \bar{\tau}(b)$.
Since $b$ is an arrow, its degree is also $< n$, so again by induction there is a path $q \in e_{\operatorname{h}(p)}kQe_{\operatorname{h}(b)}$ such that $\bar{\tau}(q) = \bar{\tau}(q')/\bar{\tau}(b) = m \bar{\tau}(b)/\bar{\tau}(b) = m$, proving our claim.

Finally, $p$ and $qp'$ have coincident tails and $\bar{\tau}(p) = \bar{\tau}(q) \bar{\tau}(p') = \bar{\tau}(qp')$, whence $p \sim qp'$ by Lemma \ref{injective}.
\end{proof}

\begin{Notation} \rm{
For each $i,j \in \widetilde{Q}_0$, denote by $\bar{\tau}: e_jk\widetilde{Q}e_i \rightarrow B$ the $k$-linear map defined by $\bar{\tau}(a) := \bar{\tau}(\pi(a))$.  Also, set $\sigma := x_1x_2y_1y_2$ (though later $\sigma$ will denote a cycle whose $\bar{\tau}$-image is $x_1x_2y_1y_2$).
} \end{Notation}

\begin{Lemma} \label{cycle in cover}
If $c$ is a cycle in $\widetilde{Q}$ then $\bar{\tau}(c) = \sigma^m$ for some $m \geq 0$.
\end{Lemma}

\begin{proof}
Suppose that $\sigma^m | \bar{\tau}(c)$ but $\sigma^{m+1} \nmid \bar{\tau}(c)$.  
By Lemma \ref{star} there is a cycle $d$ in $\widetilde{Q}$ at $\operatorname{t}(c)$ such that $\bar{\tau}(d) = \bar{\tau}(c)\sigma^{-m}$.  
But then $\sigma \nmid \bar{\tau}(d)$, and so $d$ must be the vertex $e_{\operatorname{t}(c)}$.
\end{proof}

We now prove that the labeling of arrows given in figure \ref{labels} determines an impression of any square superpotential algebra, and has the property that $\bar{\tau}(e_iAe_i) = \bar{\tau}(e_jAe_j) \subset B$ for each $i,j \in Q_0$.

\begin{Theorem} \label{square impression} 
Let $A$ be square superpotential algebra.  
Then $A$ admits an impression $(\tau,B = k\left[x_1,x_2,y_1,y_2 \right])$, where $\tau$ is given by the labeling of arrows in figure \ref{labels} and $\tau(e_i) = E_{ii}$ for each $i \in Q_0$.  
Furthermore, $\bar{\tau}(e_iAe_i) = \bar{\tau}(e_jAe_j) \subset B$ for each $i,j \in Q_0$.
\end{Theorem}

\begin{proof} 
We first show that $(\tau,B)$ is an impression of $A$.  By Lemma \ref{injective}, $\tau: A \rightarrow \operatorname{End}_A \left( B^{|Q_0|} \right)$ is an algebra monomorphism.       

Let $\mathfrak{q}$ be any point in the dense open subset
$$U := \left\{ x_1x_2y_1y_2 \not = 0 \right\} \subset \operatorname{Max}B.$$
Then for each $\mathfrak{q} \in U$, $\tau_{\mathfrak{q}}$ is a simple representation of $A$: each arrow $a$ is represented by a nonzero scalar multiple of $E_{\operatorname{h}(a), \operatorname{t}(a)}$, and there is a path from $i$ to $j$ for each $i,j \in Q_0$.

Finally, the map $\phi: \operatorname{Max}B \rightarrow \operatorname{Max}R$, $\mathfrak{q} \mapsto \mathfrak{q} \cap R$, is surjective: for any $\mathfrak{m} \in \operatorname{Max}R$, $B\mathfrak{m}$ is a (nonzero) proper ideal of $B$ since the only units of $B$ are the scalars.  Since $B$ is noetherian there is a maximal ideal $\mathfrak{q} \in \operatorname{Max}B$ containing $B\mathfrak{m}$.  But then $\mathfrak{q} \cap R \supseteq B\mathfrak{m} \cap R = \mathfrak{m}$, and since $\mathfrak{m}$ is a maximal ideal of $R$, $\mathfrak{q} \cap R \subseteq \mathfrak{m}$, so $\mathfrak{q} \cap R = \mathfrak{m}$.

We now show that $\bar{\tau}(e_iAe_i) = \bar{\tau}(e_jAe_j) \subset B$ for each $i,j \in Q_0$.  Since the cycles in $e_jAe_j$ generate $e_jAe_j$, it suffices to consider a cycle $p \in e_jkQe_j$.  Consider paths $r$ and $s$ in $\widetilde{Q}$ from $i' \in \pi^{-1}(i)$ to $j' \in \pi^{-1}(j)$ and $j'$ to $i'$, respectively.  By Lemma \ref{cycle in cover} $\bar{\tau}(rs) = \sigma^m$ for some $m \geq 0$.  Consequently $\sigma^m | \bar{\tau}(\pi(s)p\pi(r))$, so by Lemma \ref{star} there exists a cycle $p' \in e_ikQe_i$ such that $\bar{\tau}(p') = \bar{\tau}(p)$.
\end{proof}

By algebraic variety, we mean an irreducible affine variety.

\begin{Corollary} \label{square prime} 
Both a square superpotential algebra $A$ and its center $Z$ are prime noetherian rings, $\operatorname{Max}Z$ is a toric algebraic variety, and $A$ is a finitely-generated $Z$-module.  
\end{Corollary}

\begin{proof} 
By Theorems \ref{center} and \ref{square impression}, $A$ and $Z$ are noetherian and $A$ is module-finite over $Z$.  By Lemma \ref{prime}, $A$ and $Z$ are prime since $B =k[x_1,x_2,y_1,y_2]$ is prime.

We now show $Z$ is the coordinate ring for a toric algebraic variety: for each $i \in Q_0$, $e_iAe_i$ is generated by cycles, and the $\bar{\tau}$-image of a cycle is a monomial in $B$, so $\bar{\tau}(e_iAe_i) \subset B$ is generated by monomials in the polynomial ring $B$.  By Lemma \ref{bartau} and Theorem \ref{center}, $Z \cong e_iAe_i \cong \bar{\tau}(e_iAe_i)$.  $Z$ is therefore prime, noetherian, and isomorphic to a subalgebra of a polynomial ring generated by monomials.
\end{proof}

\section{3-dimensional normal Gorenstein centers} \label{Gorenstein centers}

Throughout this section $A = kQ/\partial W$ denotes a square superpotential algebra, $Z$ denotes its center, and $\widetilde{Q}$ denotes the covering quiver with projection $\pi: \widetilde{Q} \rightarrow Q$.  Recall that $Z$ is noetherian by Corollary \ref{square prime}.  

Recall that a vertex simple is a simple module in which every path, with the exception of a single vertex, is represented by zero.  
The $Z$-annihilators of the vertex simple $A$-modules are all equal, and we call this maximal ideal $\mathfrak{m}$ the \textit{origin} of $\operatorname{Max}Z$.  
We will show that $Z$ is a 3-dimensional normal domain, and that the localization $Z_{\mathfrak{m}}$ is Gorenstein.

\subsection{Transcendence basis}

In this section we show that the Krull dimension of the center of a square superpotential algebra is 3.  

\begin{Lemma} \label{h}
If $p$ and $p'$ are two paths in $\widetilde{Q}$ with the same tail such that $p \sim p'$, then $\operatorname{h}(p) = \operatorname{h}(p')$.
\end{Lemma}

\begin{proof}
Set $(v_1,v_2) = \operatorname{h}(p) - \operatorname{t}(p) \in \widetilde{Q}_0 = \mathbb{Z}^2$.  There is some $s,t \geq 0$ such that $$\bar{\tau}(p) = x_{n(v_1)}^{|v_1|}y_{n(v_2)}^{|v_2|} (x_1x_2)^s(y_1y_2)^t,$$ 
where $n(v_i) = 1$ or $2$ if $\operatorname{sign}(v_i) > 0$ or $\operatorname{sign}(v_i)<0$, respectively.
\end{proof}

\begin{Lemma} \label{star star} 
Modulo $\partial W$, there is a unique path $p$ without cyclic proper subpaths between any two vertices in the covering quiver $\widetilde{Q}$.
\end{Lemma}

\begin{proof}
Recall that $\sigma := x_1x_2y_1y_2$.  Suppose $p$ is a path in $\widetilde{Q}$ without cyclic subpaths and $\bar{\tau}(p) = x_1^ay_1^b(x_1x_2)^c(y_1y_2)^d$.  Then $c = 0$ or $d= 0$, since otherwise $\sigma | \bar{\tau}(p)$ whence $p$ has a cyclic subpath by Lemma \ref{star}.  
Let $p'$ be another path in $\widetilde{Q}$ from $\operatorname{t}(p)$ to $\operatorname{h}(p)$ without cyclic subpaths; then similarly $\bar{\tau}(p') = x_1^ay_1^b(x_1x_2)^{c'}(y_1y_2)^{d'}$ with $c' = 0$ or $d' = 0$.  
We claim $c=c'$ and $d=d'$; the lemma will then follow from Lemma \ref{injective}.  It suffices to consider the following two cases.

(i) Suppose $d = d' = 0$ and $c \leq c'$.  
Then $\bar{\tau}(p) | \bar{\tau}(p')$.
Thus by Lemma \ref{star} there is a cycle in $\widetilde{Q}$ with $\bar{\tau}$-image $(x_1x_2)^{c'-c}$.
Since the underlying graph of $\widetilde{Q}$ embeds into the plane, we must have $c'-c = 0$.
Therefore $\bar{\tau}(p') = \bar{\tau}(p)$.

(ii) Now suppose $c = d' = 0$ and $d \leq c'$.  
Then $\bar{\tau}(p) | \sigma^d \bar{\tau}(p') = \bar{\tau}(u^d p')$, where $u$ is a unit cycle at $\operatorname{h}(p)$.
Thus by Lemma \ref{star} there is a cycle in $\widetilde{Q}$ with $\bar{\tau}$-image $(x_1x_2)^{c'+d}$.
Again since the underlying graph of $\widetilde{Q}$ embeds into the plane, we must have $c' + d = 0$.
Since $c', d \geq 0$, this implies $c' = d = 0$, and so $\bar{\tau}(p') = \bar{\tau}(p)$.
\end{proof}

Fix $i \in Q_0$ and consider the sublattice $\pi^{-1}(i) \subset \widetilde{Q}_0 = \mathbb{Z}^2$.  
Let $u,v \in \pi^{-1}(i)$ be $\mathbb{Z}$-generators of $\pi^{-1}(i)$ with respect to a fixed origin $(0,0) \in \pi^{-1}(i)$,
\begin{equation} \label{uv}
\pi^{-1}(i) = \mathbb{Z}u \oplus \mathbb{Z}v.
\end{equation}
Let $\alpha \in e_iAe_i$ (resp.\ $\beta \in e_iAe_i$) be the $\pi$-image modulo $\partial W$ of a path of minimal length from $(0,0)$ to $u$ (resp.\ $v$) in $\widetilde{Q}$; these paths are unique by Lemma \ref{star star}.   
Further, by abuse of notation let $\sigma \in e_iAe_i$ be the (unique) cycle satisfying $\bar{\tau}(\sigma)= x_1x_2y_1y_2$.

\begin{Proposition} \label{Krull dim} 
Let $Z$ be the center of a square superpotential algebra $A$.  Then $Ze_i$ has transcendence basis $\left\{ \alpha, \beta, \sigma \right\}$ over $k$, and so the Krull dimension of $Z \cong Ze_i$ is 3.
\end{Proposition}

\begin{proof} 
We claim that the set $\left\{ \alpha, \beta, \sigma \right\}$ is algebraically independent over $k$.  
Consider a nonzero polynomial
$$f(\alpha, \beta, \sigma) = \sum_{\ell = 1}^n b_{\ell} \alpha^{r_{\ell}} \beta^{s_{\ell}} \sigma^{t_{\ell}}$$
with coefficients $b_{\ell} \in k$.  
Fix $1 \leq \ell \leq n$, and let $d_{\ell}$ be the lift of $\alpha^{r_{\ell}} \beta^{s_{\ell}} \sigma^{t_{\ell}}$ in $\widetilde{Q}$ with tail at $(0,0) \in \pi^{-1}(i) \subset \widetilde{Q}_0 = \mathbb{Z}^2$.
We may assume, without loss of generality, that $\operatorname{h}(d_{\ell}) = \operatorname{h}(d_1)$ in $\widetilde{Q}$ since otherwise there is no relation between $d_{\ell}$ and $d_1$ by Lemma \ref{h}. 
This yields $r_{\ell}u + s_{\ell}v = \operatorname{h}(d_{\ell}) = \operatorname{h}(d_1) = r_1 u + s_1 v$.
Thus, since $u$ and $v$ are linearly independent, $r_{\ell} = r_1$ and $s_{\ell} = s_1$. 
Therefore $d_{\ell} \sim \alpha^{r_1} \beta^{s_1} \sigma^{t_{\ell}}$.
Since this holds for each $\ell$, we have
$$f(\alpha, \beta, \sigma) = \alpha^{r_1} \beta^{s_1} \sum_{\ell} b_{\ell} \sigma^{t_{\ell}}.$$
But $\bar{\tau}(\sigma)$ is algebraically independent over $k$, whence $f(\alpha, \beta, \sigma) \not = 0$.

We now claim that if $g \in Ze_i = e_iAe_i$ is a cycle, then $\left\{ \alpha, \beta, \sigma, g \right\}$ is algebraically dependent over $k$.  
Let $g^+$ be the lift of $g$ with tail at $(0,0) \in \pi^{-1}(i) \subset \widetilde{Q}_0$.
Then $\operatorname{h}(g^+) = mu + nv$ for some $m,n \in \mathbb{Z}$.  
It suffices to suppose $m,n \geq 0$.  
Denote by $h$ the path in $\widetilde{Q}$ obtained by removing all cyclic proper subpaths of $g^+$, and denote by $h'$ the lift of $\alpha^m \beta^n$ with tail at $(0,0)$.
Since the lifts of $\alpha$ and $\beta$ have no cyclic proper subpaths by definition, $h'$ also has no cyclic proper subpaths.
Therefore, since $h$ and $h'$ have coincident heads and tails in $\widetilde{Q}$, Lemma \ref{star star} implies $h \sim h'$.
$h$ and $g^+$ also have coincident heads and tails in $\widetilde{Q}$, and $\bar{\tau}(h) | \bar{\tau}(g^+)$, and so Lemma \ref{star} implies that there is a cycle $c$ in $\widetilde{Q}$ at $(0,0)$ such that $hc \sim g^+$.
Moreover, Lemma \ref{cycle in cover} implies that $\pi(c) \sim \sigma^r$ for some $r \geq 0$.
Therefore 
$$g = \pi(g^+) \sim \pi(hc) = \pi(h) \pi(c) \sim \pi(h') \pi(c) \sim \alpha^m \beta^n \sigma^r,$$
proving our claim.
  
If $g \in Ze_i$ is a linear combination of cycles, then we may apply this argument to each monomial summand of $g$ to conclude that the set $\left\{ \alpha, \beta, \sigma, g \right\}$ is algebraically dependent.
\end{proof}

Although $\left\{ \alpha, \beta, \sigma \right\}$ is algebraically independent over $k$, it does not in general form a $Ze_i$-regular sequence, so that is what we now determine.

\subsection{$Z$-regular sequence and socle}

Throughout, $\mathfrak{m}$ denotes the maximal ideal at the origin of $\operatorname{Max}Z$ (or $\operatorname{Max}(Ze_i)$, depending on context).  
If a prime, finitely-generated $k$-algebra is homologically homogeneous and module-finite over its center, then its center is Cohen-Macaulay \cite[Theorem 2.2]{SV} and normal \cite[Theorem 6.1]{BH}.  
It will then follow from Proposition \ref{q} below that the localization $Z_{\mathfrak{m}}$ is Cohen-Macaulay and normal.  
In this section we will show that $Z$ is normal and $Z_{\mathfrak{m}}$ is Gorenstein.  
To show Gorenstein, we will first determine an explicit $Z$-regular sequence $s$ in $\mathfrak{m}$ (thus providing a direct proof that $Z_{\mathfrak{m}}$ is Cohen-Macaulay), and then we will show that the zero-dimensional local ring $Z_{\mathfrak{m}}/(s)$ has a simple socle.

\begin{Lemma} \label{normal}
The center of a square superpotential algebra is normal.
\end{Lemma}

\begin{proof}
By Corollary \ref{square prime}, $\operatorname{Max}Z$ is a toric algebraic variety.
Thus we want to show that the semigroup $S \subseteq \mathbb{N}^4$ of $\operatorname{Max}Z$ is saturated in the lattice it generates $\mathbb{Z}S \subseteq \mathbb{Z}^4$, that is, $\mathbb{Z}S \cap \mathbb{N}^4 = S$ \cite[Theorem 1.3.5]{CLS}.
Let $w \in \mathbb{Z}S \cap \mathbb{N}^4$.  
Then 
$$w = \sum_{i = 1}^{n} r_iu_i - \sum_{j = 1}^{n'} s_j v_j \ \text{ for some } \ u_i, v_j \in S \ \text{ and } \ r_i,s_j \geq 0.$$
Let $m_w$, $m_{u_i}$, $m_{v_j} \in B$ be the monomials corresponding respectively to $w$, $u_i$, and $v_j$. 
Then 
$m_w m_{v_1}^{s_1} \cdots m_{v_{n'}}^{s_{n'}} = m_{u_1}^{r_1} \cdots m_{u_n}^{r_n}$.
Since the $u_i, v_j$ are in $S$, the monomials $m_{u_i}, m_{v_j}$ are $\bar{\tau}$-images of cycles.
Therefore by Lemma \ref{star}, $m_w$ is the $\bar{\tau}$-image of a cycle, so $w \in S$ as well, proving the lemma.
\end{proof}

\begin{Proposition} \label{triangular} 
Let $A$ be a square superpotential algebra.  
Suppose there is a cycle $c$ in $Q$ such that $\bar{\tau}(c) = x_{\alpha}^n$ or $y_{\beta}^n$ for some $n \geq 1$.
Then $A$ is the McKay quiver algebra of a representation in $\operatorname{SL}_3(k)$ of an abelian group, and we say that $Q$ is \textit{McKay}.  
In particular, $Z_{\mathfrak{m}}$ is Gorenstein.
Moreover, $A$ admits an impression $(\tau,B)$ where $B$ is the polynomial ring $k[x,y,z]$ in three variables.
\end{Proposition}

\begin{proof}
If there is a cycle whose $\bar{\tau}$-image is only divisible by $x_{\alpha}$ (resp.\ $y_{\beta}$), then for each $i \in Q_0$ there is a cycle in $e_iAe_i$ whose $\bar{\tau}$-image is only divisible by $x_{\alpha}$ (resp.\ $y_{\beta}$) by Theorem \ref{square impression}.  Recalling figure \ref{squares}, it follows that that each row (resp.\ column) of building blocks in the covering quiver $\widetilde{Q}$ must consist of identical building blocks, and each building block must contain a diagonal arrow.  In this case $\widetilde{Q}$ can be redrawn so that there are only three orientations of arrows, namely one horizontal, one vertical, and one diagonal, as shown in figure \ref{redrawn}.
\begin{figure}
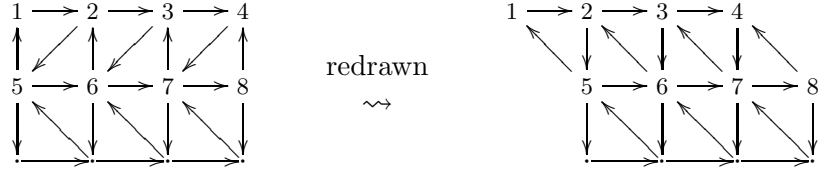

$$\xy (-15,10)*+{\text{\scriptsize{$1$}}}="1";(-5,10)*+{\text{\scriptsize{$2$}}}="2";(5,10)*+{\text{\scriptsize{$3$}}}="3";(15,10)*+{\text{\scriptsize{$4$}}}="4";
(-15,0)*+{\text{\scriptsize{$5$}}}="5";(-5,0)*+{\text{\scriptsize{$6$}}}="6";(5,0)*+{\text{\scriptsize{$7$}}}="7";(15,0)*+{\text{\scriptsize{$8$}}}="8";
(-15,-10)*{\cdot}="9";(-5,-10)*{\cdot}="10";(5,-10)*{\cdot}="11";(15,-10)*{\cdot}="12";
{\ar"1";"2"};{\ar"2";"3"};{\ar"3";"4"};
{\ar"5";"6"};{\ar"6";"7"};{\ar"7";"8"};
{\ar"9";"10"};{\ar"10";"11"};{\ar"11";"12"};
{\ar"5";"1"};{\ar"6";"2"};{\ar"7";"3"};{\ar"8";"4"};
{\ar"5";"9"};{\ar"6";"10"};{\ar"7";"11"};{\ar"8";"12"};
{\ar"2";"5"};{\ar"3";"6"};{\ar"4";"7"};
{\ar"10";"5"};{\ar"11";"6"};{\ar"12";"7"};
\endxy
\ \ \ \  \begin{array}{c} \text{\small{\begin{tabular}{c} redrawn \end{tabular}}} \\ \leadsto \end{array}
\ \ \ \
\xy
(-20,10)*+{\text{\scriptsize{$1$}}}="1";(-10,10)*+{\text{\scriptsize{$2$}}}="2";(0,10)*+{\text{\scriptsize{$3$}}}="3";(10,10)*+{\text{\scriptsize{$4$}}}="4";
(-10,0)*+{\text{\scriptsize{$5$}}}="5";(0,0)*+{\text{\scriptsize{$6$}}}="6";(10,0)*+{\text{\scriptsize{$7$}}}="7";(20,0)*+{\text{\scriptsize{$8$}}}="8";
(-10,-10)*{\cdot}="9";(0,-10)*{\cdot}="10";(10,-10)*{\cdot}="11";(20,-10)*{\cdot}="12";
{\ar"1";"2"};{\ar"2";"3"};{\ar"3";"4"};
{\ar"5";"6"};{\ar"6";"7"};{\ar"7";"8"};
{\ar"9";"10"};{\ar"10";"11"};{\ar"11";"12"};
{\ar"5";"1"};{\ar"6";"2"};{\ar"7";"3"};{\ar"8";"4"};
{\ar"5";"9"};{\ar"6";"10"};{\ar"7";"11"};{\ar"8";"12"};
{\ar"2";"5"};{\ar"3";"6"};{\ar"4";"7"};
{\ar"10";"5"};{\ar"11";"6"};{\ar"12";"7"};
\endxy$$
\caption{
Redrawing the quiver of a McKay square superpotential algebra.}
\label{redrawn}
\end{figure}

Let $\bar{\tau}$ now be defined by the orientation of the arrows in the redrawn covering quiver, and denote by $x$, $y$, and $z$ the respective $\bar{\tau}$-images of the horizontal, vertical, and diagonal arrows.
It follows from Theorem \ref{square impression} that $\bar{\tau}$ defines an impression of $A$, with $B = k[x,y,z]$.

Fix $i \in Q_0$ and recall the sublattice $\pi^{-1}(i) \subset \widetilde{Q}_0 = \mathbb{Z}^2$ in (\ref{uv}).
Suppose the horizontal and vertical arrows point left and up, respectively.
Let $\epsilon_x$ and $\epsilon_y$ be the standard basis vectors in $\mathbb{R}^2$.
Then
$$\mathbb{Z}^2 = \mathbb{N}\epsilon_x \oplus \mathbb{N}\epsilon_y \oplus \mathbb{N}(-\epsilon_x- \epsilon_y),$$
where
$$(a,b) \mapsto \left\{ \begin{array}{ll} (a,b,0) & \text{ if } \ a \geq 0, \ b \geq 0, \\ (a +|b|,0,|b|) & \text{ if } \ b<0, \ b \leq a, \\  (0,b + |a|,|a|) & \text{ if } \ a < 0, \ a \leq b. \end{array} \right.$$
Furthermore, elements of $\mathbb{N}\epsilon_x \oplus \mathbb{N}\epsilon_y \oplus \mathbb{N}(-\epsilon_x- \epsilon_y)$ may be viewed as monomials in $B$, where $(a,b,c) \mapsto x^ay^bz^c$.
Under this identification,
\begin{equation} \label{pi-1}
\bar{\tau}(e_iAe_i) = k[\pi^{-1}(i)][xyz],
\end{equation}
where $xyz$ has been adjoined since there is precisely one cycle in $\widetilde{Q}$ at $(0,0) \in \pi^{-1}(i)$ without cyclic proper subpaths by Lemma \ref{cycle in cover}, namely the unit cycle with $\bar{\tau}$-image $xyz$.

By the general theory of abelian groups, the quotient $G := \mathbb{Z}^2 /\pi^{-1}(i)$ is a finite abelian group.
Moreover, there is a faithful representation $\rho: G \rightarrow \operatorname{SL}_3(k)$, $g \mapsto \operatorname{diag}\left( \begin{array}{ccc} \omega_{g,x} & \omega_{g,y} & \omega_{g,x}^{-1}\omega_{g,y}^{-1} \end{array} \right)$, with $\omega_{g,x}, \omega_{g,y} \in k$ roots of unity, such that the ring of invariants under the diagonal action of $\rho(G)$ on $k[x,y,z]$ is
$$k[x,y,z]^{\rho(G)} = k[\pi^{-1}(i)][xyz].$$
Therefore by (\ref{pi-1}), 
$$Z \cong \bar{\tau}(e_iAe_i) = k[\pi^{-1}(i)][xyz] = k[x,y,z]^{\rho(G)}.$$

We remark that $A = kQ/\partial W$ is the standard McKay quiver algebra
$$A = kQ/\left\langle [ x,y], \ [y,z], \ [z, x] \right\rangle,$$
where $x$, $y$, and $z$ are the sums of all the horizontal, vertical, and diagonal arrows, respectively (in the projection of the redrawn covering quiver), of the representation $\rho$ of $G$.
\end{proof}

\begin{Proposition} \label{astronaut}
Suppose $Q$ is not McKay and let $i \in Q_0$.  Then for each $\alpha,\beta \in \{1,2\}$ there exists a unique cycle in $e_iAe_i$, without cyclic proper subpaths, whose $\bar{\tau}$-image is of the form $x_{\alpha}^sy_{\beta}^t$ for some $s,t \geq 1$.
\end{Proposition}

\begin{proof}
We first show existence.  We claim that for any choice of $\alpha, \beta \in \{1,2\}$, there is a cycle whose $\bar{\tau}$-image is \textit{not} divisible by $x_{\alpha}$ and $y_{\beta}$.
Indeed, each vertex is the tail of an arrow whose $\bar{\tau}$-image is not divisible by $x_{\alpha}$ or $y_{\beta}$; see figure \ref{squares}.  
We can therefore construct a path in $Q$ of arbitrary length whose $\bar{\tau}$-image is not divisible by $x_{\alpha}$ and $y_{\beta}$.  
Since $|Q_0| < \infty$ we can suppose this path intersects itself, say at vertex $j$, thereby forming a cycle at $j$ whose $\bar{\tau}$-image is not divisible by $x_{\alpha}$ and $y_{\beta}$.  
By Theorem \ref{square impression}, $\bar{\tau}(e_iAe_i) = \bar{\tau}(e_jAe_j)$, so there is a cycle at $i$ whose $\bar{\tau}$-image is not divisible by $x_{\alpha}$ and $y_{\beta}$, and therefore is of the form $x_{\alpha+1}^sy_{\beta+1}^t$ (with indices mod 2).
Since $Q$ is not McKay, Proposition \ref{triangular} implies $s, t \geq 1$.

We now show uniqueness.
Fix $i \in Q_0$.
Without loss of generality, consider two cycles $p_1, p_2 \in e_iAe_i$ without cyclic proper subpaths whose respective $\bar{\tau}$-images are $x_1^{s_1} y_2^{t_1}$ and $x_1^{s_2}y_2^{t_2}$, with $s_1 < s_2$.
Consider the lifts $p_1^+$ and $p_2^+$ of $p_1$ and $p_2$ with tails at $(0,0) \in \pi^{-1}(i) \subset \widetilde{Q}_0 = \mathbb{Z}^2$.
Since the $\bar{\tau}$-images of $p_1$ and $p_2$ are only divisible by $x_1$ and $y_1$, the head of $p_1^+$ is at $(s_1,t_1)$ and the head of $p_2^+$ is at $(s_2,t_2)$.
Clearly there exists a path $d$ in $\widetilde{Q}$ from $(s_1,t_1)$ to $(s_2,t_2)$ with $\bar{\tau}$-image $x_1^{s_2-s_1}y_1^{u_1}y_2^{u_2}$ for some $u_1, u_2 \geq 0$.
Since $B$ is a polynomial ring and $x_2$ does not divide $\bar{\tau}(d)$ or $\bar{\tau}(p_1^+)$, $x_2$ does not divide their product $\bar{\tau}(d)\bar{\tau}(p_1^+) = \bar{\tau}(dp_1^+)$. 
Thus $\sigma \nmid \bar{\tau}(dp_1^+)$, so $dp_1^+$ has no cyclic proper subpaths by Lemma \ref{cycle in cover}.
Therefore $dp_1^+$ have $p_2^+$ have coincident heads and tails in $\widetilde{Q}$ and no cyclic proper subpaths, whence $p_2 \sim \pi(d)p_1$ by Lemma \ref{star star}.
Since $p_2$ has no cyclic proper subpaths, $\pi(d)$ must be the vertex $\operatorname{h}(p_2)$.
Therefore $(s_1,t_1) = (s_2,t_2)$, contradicting our assumption that $s_1 < s_2$.
\end{proof}

Suppose $Q$ is not McKay and fix $i \in Q_0$.  Denote by $a,b,c,d \in e_iAe_i = Ze_i$ the unique cycles, without cyclic proper subpaths, of the form 
\begin{equation} \label{abcd}
\bar{\tau}(a) = x_1^{s_a}y_1^{t_a}, \ \ \ \bar{\tau}(b)= x_2^{s_b}y_2^{t_b}, \ \ \ \bar{\tau}(c) = x_1^{s_c} y_2^{t_c}, \ \ \ \bar{\tau}(d) = x_2^{s_d}y_1^{t_d},
\end{equation}
with $s_*, t_* \geq 1$.  Denote by $\sigma \in Ze_i$ the unique cycle satisfying $\bar{\tau}(\sigma) = x_1x_2y_1y_2$.

\begin{Lemma} \label{Z-reg} 
The sequence $(c-d,a,b)$ is a $Ze_i$-regular sequence.
\end{Lemma}

\begin{proof}
If $\alpha$ and $\beta$ are elements of $R \subset B$ and $\beta$ is a zerodivisor on $R/(\alpha)R$, then $\beta$ will also be a zerodivisor on $B/(\alpha)B$.  
It follows by the contrapositive that if $\alpha_1, \ldots, \alpha_n \in R$ is a $B$-regular sequence and $R/(\alpha_1, \ldots, \alpha_n)R$ is nonzero, then $\alpha_1, \ldots, \alpha_n$ will be an $R$-regular sequence. 

Clearly the sequence $s := \left( \bar{\tau}(c)-\bar{\tau}(d), \bar{\tau}(a), \bar{\tau}(b) \right)$ is a $B$-regular sequence.
Furthermore, let $\gamma \in Ze_i$ be the cycle satisfying $\bar{\tau}(\gamma) = \bar{\tau}(abc)\bar{\tau}(\sigma)^{-1}$, which exists by Lemma \ref{star} and is unique by Lemma \ref{star star}.
Then clearly $\bar{\tau}(\gamma) \in R$ is not in the ideal $(s)R$, so $R/(s)R$ is nonzero.
Therefore, by the previous paragraph, $s$ is also an $R$-regular sequence. 
By Theorem \ref{square impression}, $R = \bar{\tau}(e_iZe_i) = \bar{\tau}(Ze_i)$, and so $s$ is a $\bar{\tau}(Ze_i)$-regular sequence.
Moreover, by Lemma \ref{injective}, $\bar{\tau}: Ze_i \rightarrow \bar{\tau}(Ze_i)$ is an algebra isomorphism.
Therefore $(c-d,a,b)$ is a $Ze_i$-regular sequence.
\end{proof} 

\begin{Remark} \label{Cohen-Macaulay} \rm{
A ring $R$ is Cohen-Macaulay if $\operatorname{depth}\mathfrak{m} = \operatorname{codim}\mathfrak{m}$ for every maximal ideal $\mathfrak{m}$ of $R$ \cite[section 18.2]{E}, and as noted above, $Z_{\mathfrak{m}}$ is Cohen-Macaulay since $A_{\mathfrak{m}}$ is homologically homogeneous.  However, this also follows directly from Lemma \ref{Z-reg} and Proposition \ref{Krull dim} since
$$3 \leq \operatorname{depth}\mathfrak{m} \leq \operatorname{codim} \mathfrak{m} = 3.$$
} \end{Remark}

Let $(R,\mathfrak{m})$ be a commutative noetherian local ring.  Recall that the socle of an $R$-module $M$ is the annihilator in $M$ of the unique maximal ideal $\mathfrak{m}$ of $R$.   Also recall that if $R$ is zero-dimensional, then $R$ is Gorenstein if $R \cong \omega$, where the canonical module $\omega$ is the injective hull of the residue field $R/\mathfrak{m}$ \cite[definition in section 21.2]{E}, and this is equivalent to the socle of $R$ being simple \cite[Proposition 21.5]{E}.  More generally, if $R$ is Cohen-Macaulay then $R$ is Gorenstein if there is a nonzerodivisor $x \in R$ such that $R/(x)$ is Gorenstein \cite[definition in section 21.3]{E}.

\begin{Proposition} \label{simple socle}
Denote by $I \subset Ze_i$ the ideal generated by the $Ze_i$-regular sequence $(c-d,a,b)$, and consider the zero-dimensional local ring $( R := (Ze_i)_{\mathfrak{m}} /I_{\mathfrak{m}}, \mathfrak{m}_{\mathfrak{m}}/I_{\mathfrak{m}} )$.  
The socle $\operatorname{ann}_R(\mathfrak{m}_{\mathfrak{m}}/I_{\mathfrak{m}})$ of $R$ is a simple $R$-module, and so $R$ is Gorenstein.
\end{Proposition}

\begin{proof}
By abuse of notation, in the following denote by $a,b,c,d \in B$ the respective $\bar{\tau}$-images of the cycles $a,b,c,d \in Ze_i$.
Set $S := R_{\mathfrak{m}}/(c - d)R_{\mathfrak{m}}$, which is nonzero since $c-d \in \mathfrak{m}$.
Clearly for $\eta \in R_{\mathfrak{m}}$,
$\eta \in \operatorname{ann}_{R_{\mathfrak{m}}}(\mathfrak{m}_{\mathfrak{m}}/(a,b,c-d))$ if and only if $\eta \mathfrak{m}_{\mathfrak{m}} \subseteq (a,b,c-d) R_{\mathfrak{m}}$, if and only if $\eta(\mathfrak{m}_{\mathfrak{m}}/(c-d)) \subseteq (a,b)S$.

(I) Denote by $\sigma := x_1x_2y_1y_2$ the $\bar{\tau}$-image of a unit cycle at vertex $i$.
Since $a$, $b$, $c$, and $\sigma$ are $\bar{\tau}$-images of cycles at $i$, Lemma \ref{star} implies that there is a cycle $\gamma$ at $i$ whose $\bar{\tau}$-image is $\frac{abc}{\sigma}$.
Thus, since $\frac{abc}{\sigma}$ is the $\bar{\tau}$-image of a cycle, it is in $\mathfrak{m}$.
Moreover, it is clear that $\frac{abc}{\sigma}$ is not in the ideal $(a,b,c-d)R$.
Therefore $\frac{abc}{\sigma}$ is nonzero in $\mathfrak{m}_{\mathfrak{m}}/(a,b,c-d)$.

(II) We claim that $\frac{abc}{\sigma} (\mathfrak{m}_{\mathfrak{m}}/(c-d)) \subseteq (a,b)S$.

First let $m \in \mathfrak{m}S$ be a monomial.
If $x_1y_1$ or $x_2y_2$ divides $m$ then $a$ or $b$ respectively divides $\frac{abc}{\sigma} m$, yielding $\frac{abc}{\sigma} m \in (a,b)S$.
Since $Q$ is not McKay, there is no monomial in $R_{\mathfrak{m}}$ of the form $x_{\alpha}^n$ or $y_{\beta}^n$ by Proposition \ref{triangular}.
Thus if both $x_1y_1$ and $x_2y_2$ do not divide $m$, then it must be that $m = c^n = d^n$ for some $n \geq 1$.
But $\sigma | cd$, and so we have 
$$\frac{abc}{\sigma} m = \frac{abc}{\sigma} c^n = c^{n-1} \frac{cd}{\sigma} ab \in (a,b)S.$$

Now consider a polynomial $p = \sum_j m_j \in \mathfrak{m}S$, with each $m_j$ a (nonconstant) monomial.
Then $\frac{abc}{\sigma} p = \sum_j \frac{abc}{\sigma} m_j$ is in $(a,b)S$ since each term $\frac{abc}{\sigma}m_j$ is in $(a,b)S$ by the previous paragraph. 
This proves (II).

(III.1) Suppose $\eta \in R$ is a monomial which is not in $(a,b,c-d)R$ and satisfies 
\begin{equation} \label{eta}
\eta (\mathfrak{m}_{\mathfrak{m}}/(c-d)) \subseteq (a,b)S.
\end{equation}
We claim that $\eta = \frac{abc}{\sigma}$.

Since $\sigma$ is in $\mathfrak{m}$, by (\ref{eta}) we may view $\eta \sigma$ as an element of $(a,b)S$.  
Since $\eta$ is a monomial, this implies $\frac{a}{x_1y_1}$ or $\frac{b}{x_2y_2}$ divides $\eta$.
Without loss of generality, we may therefore suppose there are integers $\alpha, \beta_1, \beta_2 \geq 0$ such that
$$\eta = \frac{a}{x_1y_1} x_2^{\alpha} y_1^{\beta_1} y_2^{\beta_2}.$$
Since $\eta$ is not in $(a,b,c-d)R$, we must have
\begin{equation} \label{b nmid}
b \nmid x_2^{\alpha}y_2^{\beta_2},
\end{equation}
\begin{equation} \label{d nmid}
d \nmid x_2^{\alpha}y_1^{\beta_1 -1},
\end{equation}
where (\ref{d nmid}) holds because otherwise $cy_1 = dy_1$ divides $x_2^{\alpha}y_1^{\beta_1}$ in $S$, whence $x_1y_1$ divides $x_2^{\alpha}y_1^{\beta_1}$ in $S$, yielding $a | \eta$ in $S$.

Set $\alpha_1 := \operatorname{min} \left\{ \alpha, s_b-1 \right\}$ and $\alpha_2 := \alpha - \alpha_1 \geq 0$.

(III.1.a) We first claim $x_2^{\alpha_1}y_2^{\beta_2} | \frac{b}{x_2y_2}$.  By condition (\ref{b nmid}) there are two possibilities:
$$\begin{array}{clcl}
(i) & x_2^{\alpha_1} y_2^{\beta_2} = \frac{b}{x_2}y_2^r = \frac{b}{x_2y_2}y_2^{r+1}, & \ \ \ \ & r \geq -t_b,\\
(ii) & x_2^{\alpha_1} y_2^{\beta_2} = \frac{b}{y_2} x_2^{r'}, & & r'  \geq -s_b.  
\end{array}$$

First consider (i): Suppose to the contrary that $r \geq 0$.
By assumption $x_2^{\alpha_1} y_2^{\beta_2} = \frac{b}{x_2}y_2^r$, so $\beta_2 = t_b + r$, and therefore $\beta_2 \geq t_b$. 

If $\alpha_2 \geq 1$ then $\alpha \geq 1 + \alpha_1 > \alpha_1$, so $\alpha_1 = s_b -1$, whence 
$$\alpha \geq 1 + \alpha_1 = 1 + (s_b -1) = s_b.$$
But then $\beta_2 \geq t_b$ and $\alpha \geq s_b$, which implies $b = x_2^{s_b}y_2^{t_b}| x_2^{\alpha}y_2^{\beta_2}$, contrary to (\ref{b nmid}).
Therefore the assumption $r \geq 0$ implies that $\alpha_2 = 0$, and so $\alpha = \alpha_1$.
In this case,
$$\eta = \frac{a}{x_1y_1} x_2^{\alpha} y_1^{\beta_1}y_2^{\beta_2} = \frac{a}{x_1y_1} x_2^{\alpha_1} y_1^{\beta_1}y_2^{\beta_2} = \frac{a}{x_1y_1}\frac{b}{x_2y_2} y_2^{r+1}y_1^{\beta_1} = \frac{ab}{\sigma} y_2^{r+1}y_1^{\beta_1}.$$ 
But $\eta$, $ab$, and $\sigma$ are $\bar{\tau}$-images of cycles, so Lemma \ref{star} implies there must exist a cycle with $\bar{\tau}$-image $y_2^{r+1}y_1^{\beta_1}$, a contradiction.
Therefore 
\begin{equation} \label{tb}
-t_b \leq r \leq -1.
\end{equation}

For (ii): By assumption $x_2^{\alpha_1}y_2^{\beta_2} = \frac{b}{y_2} x_2^{r'}$, so $\alpha_1 = s_b + r'$.
But $\alpha_1 < s_b$, so $r' \leq -1$, yielding
\begin{equation} \label{sb}
-s_b \leq r' \leq -1.
\end{equation}
By (\ref{tb}) and (\ref{sb}) we therefore have $x_2^{\alpha_1}y_2^{\beta_2} | \frac{b}{x_2y_2}$, proving (III.1.a). 

(III.1.b) We now claim $x_2^{\alpha_2} y_1^{\beta_1} | d$.
Since $ab$ and $\sigma$ are $\bar{\tau}$-images of cycles at $i$ and $\frac{a}{x_1y_1}x_2^{\alpha_1}y_2^{\beta_2}$ divides $\frac{ab}{x_1x_2y_1y_2} = \frac{ab}{\sigma}$ by (a), Lemma \ref{star} implies that there must be a cycle at $i$ whose $\bar{\tau}$-image is $\frac{a}{x_1y_1}x_2^{\alpha_1}y_2^{\beta_2}$.
But $\eta = \frac{a}{x_1y_1} x_2^{\alpha} y_1^{\beta_1} y_2^{\beta_2}$ is also the $\bar{\tau}$-image of a cycle, and so $x_2^{\alpha_2}y_1^{\beta_1}$ must be the $\bar{\tau}$-image of a cycle as well, again by Lemma \ref{star}.  
Therefore $x_2^{\alpha_2}y_1^{\beta_1} = d^n$ for some $n \geq 1$ by the uniqueness in Proposition \ref{astronaut}.
Furthermore, $n$ must equal 1 for otherwise (\ref{d nmid}) would not hold, proving (III.1.b).

By (III.1.a) and (III.1.b), we have 
$$\eta = \frac{a}{x_1y_1} x_2^{\alpha} y_1^{\beta_1}y_2^{\beta_2} = \frac{a}{x_1y_1} \left( x_2^{\alpha_1} y_2^{\beta_2} \right) \left( x_2^{\alpha_2} y_1^{\beta_1} \right) \ \mid  \ \frac{abd}{\sigma}.$$
Therefore, since $\eta$ is the $\bar{\tau}$-image of a cycle at $i$, Lemma \ref{star} implies that there is a cycle $h$ at $i$ whose $\bar{\tau}$-image is $\frac{abd}{\sigma}\eta^{-1}$.
If $h$ is a cycle of positive length, then $\bar{\tau}(h) \in \mathfrak{m}$, yielding
$$\eta = \frac{abc}{\sigma} \bar{\tau}(h) \in (a,b,c-d)R,$$ 
a contradiction to our choice of $\eta$.
Therefore $h$ must be a vertex, hence $\eta = \frac{abc}{\sigma}$, which is not in $(a,b,c-d)R$ by (I).
This proves our claim (III.1).

(III.2)
Now suppose $\eta = \sum_{\ell} \eta_{\ell} \in S$ is a polynomial which is not in $(a,b)S$ and satisfies (\ref{eta}). 
Then for any polynomial $p = \sum_j m_j$ in $\mathfrak{m}S$ with nonconstant monomials summands $m_j$, there are polynomials $\mu_a, \mu_b \in S$ such that 
\begin{equation} \label{pair}
\eta p = \mu_a a + \mu_b b.
\end{equation}
But any representative of a monomial in $S = R/(c-d)$ is a monomial in $R$, and $R$ is a subalgebra of the polynomial ring $B$.
Thus each monomial summand $\eta_{\ell}m_j$ on the left hand side of (\ref{pair}) equals some monomial summand on the right hand side, so $\eta_{\ell}m_j$ is in $(a)S$ or $(b)S$, and so is in $(a,b)S$.
Furthermore, since $m_j$ is a nonconstant monomial, $m_j \in \mathfrak{m}S$.
Therefore we may apply (III.1) to conclude that $\eta_{\ell}$ is a scalar multiple of $\frac{abc}{\sigma}$, so $\eta$ itself is a scalar multiple of $\frac{abc}{\sigma}$.

(IV)
By (III.1) and (III.2), the only nonzero elements of $\operatorname{ann}_R(\mathfrak{m}/(a,b,c-d))$ are scalar multiples of $\frac{abc}{\sigma}$.
Thus upon localizing at $\mathfrak{m}$ we have
$$\operatorname{ann}_{R_{\mathfrak{m}}}(\mathfrak{m}_{\mathfrak{m}}/(a,b,c-d)) = \left\{ \frac{f}{g} \frac{abc}{\sigma} \ | \ f, g \in R \setminus \mathfrak{m} \right\} \cup \{0 \}.$$
Therefore $\operatorname{ann}_{R_{\mathfrak{m}}}(\mathfrak{m}_{\mathfrak{m}}/(a,b,c-d))$ is a simple $R_{\mathfrak{m}}$-module generated by $\frac{abc}{\sigma}$.
\end{proof}

\subsection{Main result}

\begin{Theorem} \label{Gorenstein} 
Let $Z$ be the center of a square superpotential algebra.  Then $Z$ is a 3-dimensional normal domain and the localization $Z_{\mathfrak{m}}$ at the origin $\mathfrak{m}$ of $\operatorname{Max}Z$ is Gorenstein.
\end{Theorem}

\begin{proof}
$Z$ is 3-dimensional by Proposition \ref{Krull dim}; normal by Lemma \ref{normal}; and prime by Corollary \ref{prime}.  Furthermore, $Z_{\mathfrak{m}} \cong Z_{\mathfrak{m}}e_i$ is Gorenstein since $Z_{\mathfrak{m}}$ is Cohen-Macaulay by remark \ref{Cohen-Macaulay}; $(c-d,a,b)$ is a $Z_{\mathfrak{m}}e_i$-regular sequence by Lemma \ref{Z-reg}; and the zero-dimensional local ring $R = Z_{\mathfrak{m}}e_i/(c-d,a,b)_{\mathfrak{m}}$ has a simple socle by Proposition \ref{simple socle}.
\end{proof}

\section{Classification of simple modules} \label{Classification of simples}

Let $A$ be a square superpotential algebra.  
In this section we classify all simple $A$-modules up to isomorphism, and describe their `noncommutative residue fields' $A/\operatorname{ann}_AV$ in terms of $V$ when $V$ is a vertex simple or large $A$-module.  
We note that the only simple modules over the localization $A_{\mathfrak{m}}$ at the origin $\mathfrak{m}$ of $\operatorname{Max}Z$ are the vertex simples (see Lemma \ref{localized vertex} and Remark \ref{nonnoetherian} below), whereas the non-localized algebra $A$ has at least an affine varieties worth of simples.

\begin{Lemma} \label{leq 1}
Let $A=kQ/I$ be a quiver algebra and suppose there exists an algebra monomorphism $\tau: A \rightarrow \operatorname{End}_B\left( B^{|Q_0|} \right)$ such that $\tau(e_i) = E_{ii}$ and $\bar{\tau}(e_iAe_i) = \bar{\tau}(e_jAe_j) \subset B$ for each $i,j \in Q_0$.  If $V$ is a simple $A$-module then $\operatorname{dim}_k e_iV \leq 1$ for each $i \in Q_0$.  Therefore $V$ may be identified with a vector space diagram on $Q$ where each arrow is represented by a scalar (possibly zero).
\end{Lemma}

\begin{proof}
Suppose $S$ is a finitely-generated $k$-algebra with a complete set of orthogonal idempotents $L$, $V$ is a simple $S$-module, and $e_i \in L$.  We first claim that $e_iV$ is a simple $e_iSe_i$-module or zero.  Suppose to the contrary that $e_iV \not = 0$ is not a simple $e_iSe_i$-module. Then there exists a nonzero proper $e_iSe_i$-submodule $W \subsetneq e_iV$.  Let $u \in e_iV \setminus W$ and $w \in W$, with $w$ nonzero.  Since $V$ is a simple $S$-module there is an $a \in S$ such that $aw = u$.  Since $L$ is a complete set of idempotents, $S = \sum_{e_j,e_k \in L} e_j Se_k$, so we may write $a = \sum_{j,k} a_{jk}$ with $a_{jk} \in e_jSe_k$.  But then by orthogonality,
$$u = e_iu = e_i(aw) = e_i \sum_{j,k} a_{jk} w = e_i \sum_{j,k} e_ja_{jk} e_k \left(e_iw \right) = a_{ii} w,$$
so $a_{ii}w = u$ and $a_{ii} \in e_iSe_i$, contradicting our choice of $w$.  
Consequently if $e_iSe_i$ is a commutative finitely-generated $k$-algebra and $k$ is algebraically closed, then $\operatorname{dim}_k e_iV \leq 1$.  
For the case $S = A$, $\{e_i\}_{i \in Q_0}$ is a complete set of orthogonal idempotents, and for each $i \in Q_0$, $e_iAe_i$ is commutative by Lemma \ref{bartau} and finitely-generated by Lemma \ref{without cyclic proper subpaths}. 
\end{proof}

\begin{Theorem} \label{simples} 
Let $A$ be a square superpotential algebra with impression $(\tau,B)$, and let $V$ be a simple $A$-module.  Set $\mathfrak{p} := \operatorname{ann}_AV$ and $\mathfrak{m} := \mathfrak{p} \cap Z \in \operatorname{Max}Z$.  Then $\operatorname{dim}_ke_iV \leq 1$ for each $i \in Q_0$.  
Furthermore, one of the following holds.
\begin{enumerate}
 \item $V$ is a vertex simple $A$-module, in which case $A/\mathfrak{p} \cong V$ as $A$-modules.
 \item $V$ is supported on a single cycle $c$ in $A$ up to cyclic permutation.  
   \begin{enumerate}
     \item If $Q$ is not McKay then $\bar{\tau}(c)$ is divisible by precisely two of $x_1,x_2,y_1,y_2$.
     \item If $Q$ is McKay with $\tau$ defined in Proposition \ref{triangular}, then $\bar{\tau}(c)$ divisible by precisely one of $x,y,z$.
   \end{enumerate}
 \item $V$ is a large $A$-module, in which case
  \begin{enumerate} 
   \item $A/\mathfrak{p} \cong V^{|Q_0|}$ as $A$-modules;
   \item there is a point $\mathfrak{q} \in \operatorname{Max}B$ such that $V \cong (B/\mathfrak{q})^{|Q_0|}$, where the module structure of $(B/\mathfrak{q})^{|Q_0|}$ is given by $av := \tau_{\mathfrak{q}}(a)v$; and
   \item the projective dimension of $V$ is determined by $\mathfrak{m}$: $$\operatorname{pd}_A(V) = \operatorname{pd}_{A_{\mathfrak{m}}}(A_{\mathfrak{m}}/\mathfrak{p}_{\mathfrak{m}}) = \operatorname{pd}_{Z_{\mathfrak{m}}}(Z_{\mathfrak{m}}/\mathfrak{m}_{\mathfrak{m}}).$$
  \end{enumerate}
\end{enumerate}
\end{Theorem}

\begin{proof}
Lemma \ref{leq 1} applies since $(\tau,B)$ is an impression of $A$ by Theorem \ref{square impression}.

If $V$ is the vertex simple $A$-module at $i \in Q_0$, then there is an obvious $A$-module isomorphism $A/\mathfrak{p} \rightarrow V = kv$ given by $e_i \mapsto v$, where $0 \not = v \in V$.
  
So suppose $V$ is a non-vertex simple $A$-module.  
Then there is an arrow $g$ that does not annihilate $V$; set $i = \operatorname{t}(g)$.  
Since $V$ is simple, $g$ must be a subpath of a cycle $c \in e_iAe_i$ that does not annihilate $V$.  
But then $c^n \not \in \operatorname{ann}_AV$ for any $n \geq 1$ by Lemma \ref{leq 1}.
Therefore we may suppose $c$ has no cyclic proper subpaths.

We first consider the case where $Q$ is not McKay.
Suppose at least three of $x_1,x_2,y_1,y_2$ divide $\bar{\tau}(c)$, say $x_1$, $x_2$, and $y_1$.  
Pick $j \in Q_0$.  
Clearly there is a path $p$ from $i$ to $j$ whose $\bar{\tau}$-image is only divisible by $x_1$, $x_2$, and $y_1$.  
For $n \geq 1$ sufficiently large, $\bar{\tau}(c^n) = m\bar{\tau}(p)$ for some monomial $m \in B$.
Thus by Lemma \ref{star} there exists a path $q \in e_iAe_j$ such that $c^n = qp$ (and $\bar{\tau}(q) = m$).  
But then $qe_jp = c^n \not \in \operatorname{ann}_AV$, so $e_j \not \in \operatorname{ann}_AV$.  
Thus $\operatorname{dim}_kV \geq |Q_0|$.  
But $\operatorname{dim}_kV \leq |Q_0|$ by Lemma \ref{leq 1}, whence $\operatorname{dim}_kV = |Q_0|$, so $V$ is a large module.  
By Corollary \ref{square prime}, $A$ is prime, noetherian, and module-finite over $Z$.
Therefore (a) follows from Lemma \ref{A/p cong V}; (b) follows from Proposition \ref{(B/r)^d}; and (c) follows from Theorem \ref{projective dimensions}. 

Otherwise suppose that the $\bar{\tau}$-image of each cycle that does not annihilate $V$, including $c$, is divisible by at most two of $x_1,x_2,y_1,y_2$.   
Since we are assuming that $Q$ is not McKay, Proposition \ref{triangular} implies that $c$ is divisible by precisely two of $x_1,x_2,y_1,y_2$.
Since the underlying graph of $\widetilde{Q}$ embeds into the plane, it is not possible for $\bar{\tau}(c) = x_1^mx_2^n$ or $\bar{\tau}(c) = y_1^my_2^n$ for any $m,n \geq 1$.
Therefore $\bar{\tau}(c) = x_{\alpha}^m y_{\beta}^n$ for some $m,n \geq 1$ and $\alpha, \beta \in \{1,2\}$. 

If there is a cycle $d$ at $i$ whose $\bar{\tau}$-image is divisible by $x_{\alpha+1}$ or $y_{\beta+1}$ (indices modulo 2) and does not annihilate $V$, then the cycle $cd$ has $\bar{\tau}$-image divisible by three of $x_1,x_2,y_1,y_2$ and does not annihilate $V$, contrary to our assumption.
Moreover, by Proposition \ref{astronaut} $c$ is the only cycle at $i$ without cyclic proper subpaths whose $\bar{\tau}$-image is divisible by $x_{\alpha}$ and $y_{\beta}$ (modulo $\partial W$).
Therefore the only cycles at $i$ which do not annihilate $V$ are $c^n$ for $n \geq 0$, with $c^0 = e_i$.

If $e_j$ does not annihilate $V$, then $e_j$ and $e_i$ must be contained in a cycle that does not annihilate $V$ since $V$ is simple.
Thus $e_j$ must be a vertex subpath of $c^n$ for some $n \geq 0$.
$V$ will then be large if and only if each vertex is a subpath of $c^n$ (modulo $\partial W$) for sufficiently large $n$.

We now consider the case where $Q$ is Mckay.
Suppose at least two of $x,y,z$ divide $\bar{\tau}(c)$, say $x$ and $y$.
Pick $j \in Q_0$.
There is a path $p$ from $i$ to $j$ whose $\bar{\tau}$-image is only divisible by $x$ and $y$, so we may apply the argument in the non-McKay case to conclude that $V$ is a large module.

Otherwise suppose that the $\bar{\tau}$-image of each cycle that does not annihilate $V$, including $c$, is divisible by precisely one of $x,y,z$; say $\bar{\tau}(c) = x^n$ for some $n \geq 1$.   
If there is a cycle $d$ at $i$ whose $\bar{\tau}$-image is divisible by $y$ or $z$ and does not annihilate $V$, then the cycle $cd$ has $\bar{\tau}$-image divisible by two of $x,y,z$ and does not annihilate $V$, contrary to our assumption.
Moreover, $c$ is the only cycle at $i$ without cyclic proper subpaths whose $\bar{\tau}$-image is divisible by $x$.
Therefore the only cycles at $i$ which do not annihilate $V$ are $c^n$ for $n \geq 0$, again with $c^0 = e_i$.
As in the non-McKay case, $V$ will then be large if and only if each vertex is a subpath of $c^n$ for sufficiently large $n$.
\end{proof}

\section{Noncommutative crepant resolutions} \label{Noncommutative crepant resolutions}

We recall two definitions.  A \textit{noncommutative crepant resolution} of a normal Gorenstein domain $R$ is a homologically homogeneous $R$-algebra of the form $A = \operatorname{End}_R(M)$, where $M$ is a reflexive $R$-module \cite[Definition 4.1]{VdB}.  Furthermore, a ring $A$ which is a finitely-generated module over a central normal Gorenstein subdomain $R$ is \textit{Calabi-Yau of dimension} $n$ if (i) $\operatorname{gl.dim}A = \operatorname{K.dim}R = n$; (ii) $A$ is a maximal Cohen-Macaulay module over $R$; and (iii) $\operatorname{Hom}_R(A,R) \cong A$, as $A$-bimodules \cite[Introduction]{Braun}.  

Throughout $\mathfrak{m}$ will denote the origin of $\operatorname{Max}Z$, which is defined to be the $Z$-annihilator of the vertex simple $A$-modules.
The main result of this section is that the localization $A_{\mathfrak{m}}$ of a square superpotential algebra $A$ is a noncommutative crepant resolution of $Z_{\mathfrak{m}}$, and consequently a local Calabi-Yau algebra.  Section \ref{Endomorphism rings} is based on joint work with Alex Dugas.

\subsection{Homological homogeneity} \label{Homological homogeneity}

In this section we show that the localization $A_{\mathfrak{m}} = Z_{\mathfrak{m}} \otimes_Z A$ of a square superpotential algebra $A$ at the origin $\mathfrak{m}$ of $\operatorname{Max}Z$ is homologically homogenous with global dimension 3.  
Recall that if $S$ is a commutative noetherian equidimensional $k$-algebra and $A$ is a module-finite $S$-algebra, then $A$ is homologically homogeneous if all simple $A$-modules have the same projective dimension (see \cite{BH}, \cite[section 3]{VdB}).
We denote by $V^i$ the vertex simple $A$-module in which every path, with the exception of vertex $i$, is represented by zero.  
In physics terms, the vertex simples are (often) the fractional branes that probe the apex of a tangent cone on a singular Calabi-Yau variety.

\begin{Lemma} \label{localized vertex}
Let $A=kQ/I$ be a quiver algebra that admits a pre-impression $(\tau,B)$ such that (\ref{ei}) holds.
Then the only simple modules over the localization $A_{\mathfrak{m}}$ of $A$ at the origin $\mathfrak{m}$ of $\operatorname{Max}Z$ are the vertex simples.
\end{Lemma}

\begin{proof}
Suppose $V_{\mathfrak{m}} = Z_{\mathfrak{m}} \otimes_Z V$ is a simple $A_{\mathfrak{m}}$-module which is not annihilated by $1 \otimes a$, where $a$ is an arrow.
Then clearly $V$ is a simple $A$-module not annihilated by $a$. 

By Theorem \ref{simples}, $\operatorname{dim}_k e_j V \leq 1$ for each $j \in Q_0$.
Thus, viewing $V$ as a vector space diagram on $Q$, $a$ is represented by a nonzero scalar.
Since $V$ is simple, $a$ must be contained in some cycle $c_i \in e_iAe_i$ that is also represented by a nonzero scalar $\rho(c_i)$.
By Theorem \ref{center} there is a central element $c = \sum_{j \in Q_0} c_j \in Z$, where each $c_j \in e_jAe_j$ is a cycle.
If $\rho(c_j)$ is nonzero, then since $V$ is simple there must be a path $q$ from $i$ to $j$ such that $\rho(q)$ is nonzero.
But then $\rho(q) \rho(c_i) = \rho(qc_i) = \rho(c_jq) = \rho(c_j) \rho(q)$.
Dividing both sides by $\rho(q)$ we find
\begin{equation} \label{cj}
\rho(c_j) = \left\{ \begin{array}{ll} \rho(c_i) & \text{ if } \ e_jV \not = 0, \\ 0 & \text{ if } \ e_jV = 0. \end{array} \right.
\end{equation}
Set $\gamma = \rho(c_i)1_A$.
Then (\ref{cj}) implies that $c- \gamma$ annihilates $V$.

Again by Theorem \ref{center}, $\bar{\tau}(c_j) = \bar{\tau}(c_i)$ for each $j \in Q_0$.
Thus, since $c_i$ is a cycle of nonzero length, each cycle summand $c_i$ of $c$ will have nonzero length.
Therefore $c$ annihilates each vertex simple, and so $c - \gamma$ annihilates no vertex simple.
In particular, $c - \gamma \in Z \setminus \mathfrak{m}$. 

Therefore, for any $1 \otimes v \in V_{\mathfrak{m}}$,
$$1 \otimes v = \frac{c- \gamma}{c-\gamma} \otimes v = \frac{1}{c-\gamma} \otimes (c-\gamma)v = 0,$$ 
whence $V_{\mathfrak{m}} = 0$.
\end{proof}

\begin{Remark} \label{nonnoetherian} \rm{
Lemma \ref{localized vertex} does not hold in general.
For example, if $A$ is a non-cancellative superpotential algebra obtained from a brane tiling then its center will be nonnoetherian, and the simple module isoclasses over the localization $A_{\mathfrak{m}}$ of $A$ at the origin $\mathfrak{m}$ of $\operatorname{Max}Z$ (i.e., the $Z$-annihilator of the vertex simples) will be parameterized by a positive dimensional affine variety; see \cite{B2}.
} \end{Remark}

We establish notation.  If $g,h \in Q_1$, set $\delta_{h,g}= e_{\operatorname{t}(g)} + e_{\operatorname{h}(g)}$ if $g =h$ and 0 otherwise.  For $p =g_n \cdots g_1 \in Q_{\geq 1}$, $g_i,h \in Q_1$, set
$$\begin{array}{ccl}
\stackrel{\rightarrow}{\delta}_h p & := & \delta_{h,g_n}g_{n-1} \cdots g_1,\\
p \stackrel{\leftarrow}{\delta}_h & := & g_n \cdots g_2 \delta_{h,g_1},
\end{array}$$
and for any $i \in Q_0$,
$$\stackrel{\rightarrow}{\delta}_h e_i = e_i \stackrel{\leftarrow}{\delta}_h := 0.$$
Extend $k$-linearly to $kQ$.

\begin{Lemma}  \label{nice} 
Let $Q$ be a quiver and $W \in \operatorname{tr}\left(kQ_{\geq 2}\right)$ a superpotential.  Then for each $i \in Q_0$, $g \in Q_1e_i$, and $h \in e_iQ_1$,
$$\stackrel{\rightarrow}{\delta}_h \left( \partial_g W \right) = \left( \partial_h W \right) \stackrel{\leftarrow}{\delta}_g =: W_{hg}.$$
Consequently
\begin{equation} \label{sum} \partial_h W = \sum_{g \in Q_1e_i} W_{hg}g  \ \ \text{ and } \ \ \partial_g W = \sum_{h \in e_iQ_1} h W_{hg}.
\end{equation}
\end{Lemma}

\begin{proof} Let $i \in Q_0$ and $p =  \left(d_n gh\right) \cdots \left(d_2gh\right) \left( d_1gh \right) \in e_i Q_{\geq 1}e_i$, with $g,h \in Q_1$ and $gh$ not a subpath of $d_j$ for each $1 \leq j \leq n$ (though $g$ or $h$ separately may be).  Then 
$$\partial_g \sum_{\text{cyc}} p = \left( hd_n gh \cdots d_1 \right) + \left( hd_{n-1}gh \cdots d_n \right) + \cdots + \left( hd_1gh \cdots d_2 \right) + A,$$
$$\partial_h \sum_{\text{cyc}} p = \left( d_ngh \cdots d_1g \right) + \left( d_{n-1}gh \cdots d_ng \right) + \cdots + \left( d_1gh \cdots d_2g \right) +B,$$
where $\stackrel{\rightarrow}{\delta}_h A = B \stackrel{\leftarrow}{\delta}_g =0$.  Thus 
$$\begin{array}{ccl}
\stackrel{\rightarrow}{\delta}_h \left( \partial_g \sum_{\text{cyc}} p \right) &=& \left( d_n gh \cdots d_1 \right) + \left( d_{n-1}gh \cdots d_n \right) + \cdots + \left( d_1gh \cdots d_2 \right)\\
& = & \left( \partial_h \sum_{\text{cyc}} p \right) \stackrel{\leftarrow}{\delta}_g.
\end{array}$$
\end{proof}

\begin{Lemma} \label{B-D prop} 
Let $A = kQ/\partial W$ be a square superpotential algebra with center $Z$, let $\mathfrak{m}$ be the maximal ideal at the origin of $\operatorname{Max}Z$, and let $V^i_{\mathfrak{m}}$ be the vertex simple $A_{\mathfrak{m}}$-module at $i \in Q_0$.  
Write $Q_1e_i = \{g_1, \ldots, g_n \}$, $e_iQ_1 = \{h_1, \ldots, h_n \}$, and set $\mathfrak{p}_{\mathfrak{m}}:= \operatorname{ann}_{A_{\mathfrak{m}}} V_{\mathfrak{m}}^i$.  
Then the sequence
\begin{equation} \label{B-D complex}
\begin{array}{c} 0 \rightarrow A_{\mathfrak{m}}e_i
 \stackrel{ \delta_2 := \cdot \left[ \begin{array}{ccc} h_1 & \ldots & h_n \end{array} \right]}{\longrightarrow}
\displaystyle \bigoplus_{1 \leq k \leq n} A_{\mathfrak{m}}e_{\operatorname{t}(h_k)} \\
\\
\stackrel{ \delta_1 := \cdot \left[ \begin{array}{c} W_{h_k g_j} \end{array} \right]_{k,j}}{\longrightarrow}
\displaystyle \bigoplus_{1 \leq j \leq n} A_{\mathfrak{m}} e_{\operatorname{h}(g_j)}
 \stackrel{ \delta_0 := \cdot \left[ \begin{array}{c} g_1 \\ \vdots \\ g_n \end{array} \right]}{\longrightarrow} A_{\mathfrak{m}} e_i
 \stackrel{\phi = \ \cdot 1}{\longrightarrow} A_{\mathfrak{m}}/\mathfrak{p}_{\mathfrak{m}} \cong V_{\mathfrak{m}}^i \rightarrow 0,
 \end{array}
\end{equation}
is a projective complex.
\end{Lemma}
 
\begin{proof} 
Note that the modules $\bigoplus_{1 \leq k \leq n} A_{\mathfrak{m}}e_{\operatorname{t}(h_k)}$ and $\bigoplus_{1 \leq j \leq n} A_{\mathfrak{m}} e_{\operatorname{h}(g_j)}$ are considered as row spaces.
Each term of the sequence is a direct summand of a free $A$-module and so is projective.  
The sequence is a complex by Lemma \ref{nice}.
\end{proof}

We call (\ref{B-D complex}) the \textit{Berenstein-Douglas complex}.  In \cite[section 5.5]{BD}, Berenstein and Douglas constructed this complex for a general superpotential algebra $A$ and raised the question: under what conditions is this complex a projective resolution of a vertex simple $A$-module?  We will show by example that in general the complex may fail to be exact in both the second and third terms.  However, we will also show that when $A$ is a square superpotential algebra the complex is indeed a projective resolution of any vertex simple module.

\begin{Lemma} \label{imdelta2} 
Let $A = kQ/\partial W$ be a square superpotential algebra and $V$ a vertex simple $A$-module.  
Then $\operatorname{im}\delta_2 = \operatorname{ker}\delta_1$ and $\operatorname{im}\delta_1 = \operatorname{ker}\delta_0$ in the Berenstein-Douglas complex.
\end{Lemma}

\begin{proof} 
(i) We first show that $\operatorname{im}\delta_2 = \operatorname{ker}\delta_1$.
By Lemma \ref{B-D prop} it suffices to show that $\operatorname{im}\delta_2 \supseteq \operatorname{ker}\delta_1$.  
Order the sets $Q_1e_i = \{g_1, \ldots, g_n\}$ and $e_iQ_1 = \{h_1, \ldots, h_n\}$ both clockwise, such that $g_1h_1$ a subpath of a term of $W$.  Then
$$\delta_1 = \cdot \left[ \begin{array}{c} W_{h_k g_j} \end{array} \right]_{k,j}
=
\cdot \left[ \begin{array}{cccc} 
a_1    &0    &    \cdots & -b_m \\
-b_1    & a_2 &            &0            \\
0      & -b_2 &     &        0    \\
\vdots &     &     \ddots  & \\
0      & 0   &           & a_m
\end{array} \right],$$
where each $a_{\ell},b_{\ell}$ is nonzero.  
Suppose $\left[ \begin{array}{ccc} d_1 & \cdots & d_n \end{array} \right] \in \operatorname{ker}\delta_1$.  
Then $d_{\ell}a_{\ell}-d_{\ell+1}b_{\ell}=0$ for each $1 \leq \ell \leq n$ (indices modulo $n$).  
By Corollary \ref{ba not 0},  $d_{\ell}a_{\ell}$ and $d_{\ell+1}b_{\ell}$ are each nonzero.
Thus they must be in the same corner ring, hence $d_{\ell}, d_{\ell+1} \in e_jA$ for some $j \in Q_0$.
Since this holds for each $\ell$, we have $\{d_1, \ldots, d_n \} \subset e_jA$.  
Furthermore,
$$\bar{\tau}(d_{\ell})\bar{\tau}(a_{\ell}) E_{j, \operatorname{t}(a_{\ell})} = \tau\left(d_{\ell}a_{\ell} \right)=\tau\left(d_{\ell+1}b_{\ell}\right) = \bar{\tau}(d_{\ell+1})\bar{\tau}(b_{\ell})E_{j, \operatorname{t}(b_{\ell})},$$
so $\bar{\tau}(d_{\ell})\bar{\tau}(a_{\ell}) =  \bar{\tau}(d_{\ell+1})\bar{\tau}(b_{\ell})$ since $\operatorname{t}(a_{\ell})= \operatorname{h}(g_{\ell})=\operatorname{t}(b_{\ell})$.  
In addition, $h_{\ell} a_{\ell}-h_{\ell+1}b_{\ell} = 0$.
Therefore for each $\ell$,
$$\frac{\bar{\tau}(d_{\ell})}{\bar{\tau}(d_{\ell+1})} = \frac{\bar{\tau}(b_{\ell})}{\bar{\tau}(a_{\ell})} = \frac{\bar{\tau}(h_{\ell})}{\bar{\tau}(h_{\ell+1})}.$$
Thus
\begin{equation} \label{coffee}
\frac{\bar{\tau}(d_{\ell+1})}{\bar{\tau}(h_{\ell+1})} = \frac{\bar{\tau}(d_{\ell})}{\bar{\tau}(h_{\ell})} = \frac{r_1}{r_2} \in \operatorname{Frac}B,
\end{equation}
with $r_1,r_2 \in B$ coprime.
  
Since $\bar{\tau}(d_{\ell})$ and $\bar{\tau}(h_{\ell})$ are elements of $B$ and (\ref{coffee}) holds for each $\ell$, it must be that $r_2 |\bar{\tau}(h_{\ell})$ for each $\ell$.  
But there are (two sets of) two arrows in $e_iQ_1$ whose $\bar{\tau}$-images are coprime in $B$, so $r_2 = 1$.  
This, together with (\ref{coffee}) and $\operatorname{t}(d_{\ell}) = \operatorname{t}(h_{\ell})$, implies that there is a path $p \in e_{\operatorname{h}(d_{\ell})}Ae_{\operatorname{h}(h_{\ell})} = e_jAe_i$ such that $d_{\ell} = p h_{\ell}$ and $\bar{\tau}(p) = r_1$ by Lemma \ref{star}.  Since this holds for each $\ell$, we have
$$\left[ \begin{array}{ccc} d_1 & \cdots & d_n \end{array} \right] = p \left[ \begin{array}{ccc} h_1 & \cdots & h_n \end{array} \right] \in \operatorname{im}\delta_2.$$

(ii) We now show that $\operatorname{im}\delta_1 = \operatorname{ker}\delta_0$, and again by Lemma \ref{B-D prop} it suffices to show that $\operatorname{im}\delta_1 \supseteq \operatorname{ker}\delta_0$.  
Suppose $\left[ \begin{array}{ccc} d_1 & \cdots & d_n \end{array} \right] \in \operatorname{ker}\delta_0$, so $d_1g_1 + \cdots + d_ng_n = 0$. 
By Corollary \ref{ba not 0}, each $d_{\ell}g_{\ell}$ is nonzero, so we may assume $\{d_1, \ldots, d_n \} \subset e_jA$ for some $j \in Q_0$. 
Furthermore, since the relations $\partial W$ are generated by binomials, it suffices to suppose $d_{\ell}g_{\ell} + d_{\ell + 1}g_{\ell +1} = 0$.
In addition, $-b_{\ell}g_{\ell} + a_{\ell+1}g_{\ell+1} = 0$.
Thus, similar to (i) we have
$$\frac{\bar{\tau}(d_{\ell})}{\bar{\tau}(d_{\ell+1})} = \frac{-\bar{\tau}(g_{\ell+1})}{\bar{\tau}(g_{\ell})} = \frac{-\bar{\tau}(b_{\ell})}{\bar{\tau}(a_{\ell+1})}.$$
Therefore 
\begin{equation}\label{bnn}
\frac{\bar{\tau}(d_{\ell})}{\bar{\tau}(b_{\ell})} = \frac{-\bar{\tau}(d_{\ell+1})}{\bar{\tau}(a_{\ell+1})} = \frac{r_1}{r_2} \in \operatorname{Frac}B,
\end{equation}
with $r_1,r_2 \in B$ coprime.
Moreover,
$$\frac{-\bar{\tau}(d_{\ell+1})\bar{\tau}(b_{\ell})}{\bar{\tau}(a_{\ell+1})} = \bar{\tau}(d_{\ell}) \in B,$$
and since $\bar{\tau}(a_{\ell+1})$ and $\bar{\tau}(b_{\ell})$ are coprime in $B$, it must be that $\bar{\tau}(a_{\ell+1}) \mid \bar{\tau}(d_{\ell+1})$.
Thus $r_2 = 1$.
Therefore, similar to (i), (\ref{bnn}) implies that there is path $p$ such that $\bar{\tau}(p) = r_1$ and
$$\left[ \begin{array}{cccccc} 0 & \cdots & d_{\ell} & d_{\ell+1} & \cdots & 0 \end{array} \right] = p \left[ \begin{array}{cccccc} 0 & \cdots & -b_{\ell} & a_{\ell+1} & \cdots & 0 \end{array} \right] \in \operatorname{im}(\delta_1).$$
It follows that $\operatorname{ker} \delta_0 \subseteq \operatorname{im} \delta_1$, whence $\operatorname{ker} \delta_0 = \operatorname{im} \delta_1$. 
\end{proof}

Recall that a submodule $K$ of a module $M$ is superfluous if given any submodule $L \subseteq M$ satisfying $L + K = M$, we have $L = M$.
Furthermore, a module epimorphism $\delta: M \rightarrow N$ is a projective cover if $M$ is projective and $\operatorname{ker} \delta \subseteq M$ is a superfluous submodule.  

\begin{Lemma} \label{superfluous}
Let $A=kQ/I$ be a quiver algebra that admits a pre-impression $(\tau,B)$ such that (\ref{ei}) holds, let $\mathfrak{m} \in \operatorname{Max}Z$, and let $N$ be an $A_{\mathfrak{m}}$-module.
Suppose
$$\delta: \bigoplus_{j = 1}^n A_{\mathfrak{m}}e_{i(j)} \rightarrow N, \ \ \ i(j) \in Q_0,$$
is an $A_{\mathfrak{m}}$-module epimorphism.
If for each $j$, the intersection of $Z_{\mathfrak{m}}e_{i(j)}$ with the $j$-th summand of $\operatorname{ker}\delta$ is contained in $\mathfrak{m}_{\mathfrak{m}}e_{i(j)}$, then $\delta$ is a projective cover.
\end{Lemma}

\begin{proof}
Let $L$ be a submodule of $\bigoplus_{j} A_{\mathfrak{m}}e_{i(j)}$ such that $\operatorname{ker}\delta + L = \bigoplus_{j}A_{\mathfrak{m}}e_{i(j)}$.
We claim that $L = \bigoplus_{j} A_{\mathfrak{m}}e_{i(j)}$.
Fix $j$ and set $i:= i(j)$. 
Since $e_i \in Z_{\mathfrak{m}}e_i$ but $e_i \not \in \mathfrak{m}_{\mathfrak{m}}e_i$, our intersection assumption implies that $e_i \not \in \operatorname{ker} \delta$.
Thus there must be some $b \in \operatorname{ker}\delta$ such that $e_{i} = (-b) + (e_{i}+b) \in \operatorname{ker}\delta + L$ with $e_{i} + b \in L$.
Since $e_{i}$ lives in the $j$-th summand of $\bigoplus_{\ell}A_{\mathfrak{m}}e_{i(\ell)}$, we may assume $b$ also lives in the $j$-th summand.
Therefore by Theorem \ref{center} and our intersection assumption,
$$e_ib \in \left( \operatorname{ker}\delta \right)_j \cap e_iA_{\mathfrak{m}}e_i = \left( \operatorname{ker} \delta \right)_j \cap Z_{\mathfrak{m}}e_i \subseteq \mathfrak{m}_{\mathfrak{m}}e_i,$$
so $e_ib = ze_i$ for some $z \in \mathfrak{m}_{\mathfrak{m}}$.
Thus the element $(1+z)$ has an inverse in $A_{\mathfrak{m}}$, so 
$$e_{i} = (1+z)^{-1}(1+z)e_{i} = (1+z)^{-1}e_{i}(e_{i} +b) \in L.$$
Since this holds for each $j$, we have $L \supseteq \bigoplus_{j} A_{\mathfrak{m}}e_{i(j)}$, yielding $L = \bigoplus_{j} A_{\mathfrak{m}}e_{i(j)}$. 
Therefore $\operatorname{ker} \delta$ is superfluous. 
Finally, since $\bigoplus_{j}A_{\mathfrak{m}}e_{i(j)}$ is a projective $A_{\mathfrak{m}}$-module, $\delta$ is a projective cover.
\end{proof}

\begin{Proposition} \label{q} 
Let $A$ be a square superpotential algebra.  If $V$ is a vertex simple $A$-module with annihilator $\mathfrak{p}$ then
$$\operatorname{pd}_{A_{\mathfrak{m}}}(V_{\mathfrak{m}})= \operatorname{pd}_{A_{\mathfrak{m}}}(A_{\mathfrak{m}}/\mathfrak{p}_{\mathfrak{m}})=3,$$
and (\ref{B-D complex}) is a minimal projective resolution of $V_{\mathfrak{m}} \cong A_{\mathfrak{m}}/\mathfrak{p}_{\mathfrak{m}}$.
\end{Proposition}

\begin{proof}
The Berenstein-Douglas sequence (\ref{B-D complex}) is a projective resolution since it is a complex by Lemma \ref{B-D prop}; $\operatorname{im}\delta_2 = \operatorname{ker}\delta_1$ and $\operatorname{im}\delta_1 = \operatorname{ker}\delta_0$ by Lemma \ref{imdelta2}; and $\operatorname{ker}\delta_2 = 0$ by Corollary \ref{ba not 0}.

Furthermore, $\operatorname{ker}\phi = \mathfrak{p}_{\mathfrak{m}}e_i$ is generated by paths of nonzero length that start at $i$, and the kernels of $\delta_0$, $\delta_1$, and $\delta_2$ are generated by certain sums of paths of nonzero length.
It follows that the hypotheses of Lemma \ref{superfluous} are satisfied in each case, and so the boundary homomorphisms $\phi$, $\delta_0$, $\delta_1$, $\delta_2$ are projective covers.
Therefore $\operatorname{pd}_{A_{\mathfrak{m}}}(V_{\mathfrak{m}}) \geq 3$, whence $\operatorname{pd}_{A_{\mathfrak{m}}}(V_{\mathfrak{m}}) = 3$, and so the Berenstein-Douglas resolution is a minimal projective resolution of $V_{\mathfrak{m}}$.
\end{proof}

We note that unlike the localized case we are considering, the $A$-module homomorphism $Ae_i \stackrel{\cdot 1}{\rightarrow} A/\mathfrak{p}$ may not be a projective cover.

The following three superpotential algebras are examples where exactness fails in the second term, $\operatorname{im}\delta_2 \subsetneq \operatorname{ker}\delta_1$, of the Berenstein-Douglas complex.  Each algebra has a nontrivial noetherian center--specifically 1 dimensional--and infinite global dimension.  In each case we show how the second connecting map may be realized in two different ways, which is why the exactness fails.  $V^i$ denotes the vertex simple at $i \in Q_0$, and the map $Ae_i \stackrel{\cdot 1}{\longrightarrow} V^i$ is defined via the $A$-module isomorphism $V^i \cong A/\operatorname{ann}_AV^i$.

\begin{Example} \label{delta2examples} \rm{\ 
\begin{itemize}
 \item $Q$: $\xymatrix{\cdot \ar@(ul,dl)[]|a \ar@(ur,dr)[]|b}$ \ \ $W = a^2b$, \ \ so $Z = k[b^2]$.\\ \\
$A$ has infinite global dimension since the vertex simple $V$ has projective resolution
   $$\cdots
\stackrel{ \cdot \scriptsize{ \left[ \begin{array}{cc} 0 &a\\ a&b \end{array} \right]}}{\longrightarrow}
A^2
\stackrel{ \cdot \scriptsize{ \left[ \begin{array}{cc} a & b\\ 0&a \end{array} \right] }}{\longrightarrow}
A^2
\stackrel{ \cdot \scriptsize{ \left[ \begin{array}{cc} b &a\\ a&0 \end{array} \right]}}{\longrightarrow}
A^2
\stackrel{ \cdot \scriptsize{ \left[ \begin{array}{c} a \\ b \end{array} \right]}}{\longrightarrow}
A \stackrel{ \cdot 1}{\longrightarrow} V \rightarrow 0.$$

The second map, $\delta_1$, satisfies
$$\left[ \begin{array}{cc} W_{aa} & W_{ab} \\ W_{ba} & W_{bb} \end{array} \right]= \left[ \begin{array}{cc} * & * \\ W_{ab} & 0 \end{array} \right]$$

 \item $Q$: $\xymatrix{
   \text{\scriptsize{1}} \ar@(ul,dl)[]|{c_1} \ar@/^/[r]|{a} & \text{\scriptsize{2}} \ar@(ur,dr)[]|{c_2} \ar@/^/[l]|{b}}$ \ \ $W = c_1ba -c_2ab$, \ \ so $Z = k[c_1 +c_2]$.\\ \\
$A$ has infinite global dimension:
$$\cdots
\stackrel{ \cdot \scriptsize{ \left[ \begin{array}{cc} a & -c_2\\ 0 & b \end{array} \right]}}{\longrightarrow}
Ae_1 \oplus Ae_2
\stackrel{\cdot \scriptsize{ \left[ \begin{array}{cc} b & c_1 \\ 0&a \end{array} \right]}}{\longrightarrow}
Ae_2 \oplus Ae_1 
\stackrel{ \cdot \scriptsize{ \left[ \begin{array}{cc} a & -c_2\\ 0 & b \end{array} \right]}}{\longrightarrow}
Ae_1 \oplus Ae_2 
\stackrel{\cdot \scriptsize{ \left[ \begin{array}{c} c_1 \\ a \end{array} \right] }}{\longrightarrow}
 Ae_1 \stackrel{ \cdot 1}{\longrightarrow} V^1 \rightarrow 0.$$

$\delta_1$ satisfies
 $$\left[ \begin{array}{cc} W_{bc_1} & W_{ba} \\ W_{c_1c_1} & W_{c_1a} \end{array} \right]= \left[ \begin{array}{cc} * & * \\ 0 & W_{ac_2} \end{array} \right]$$

 \item \ \\
 $
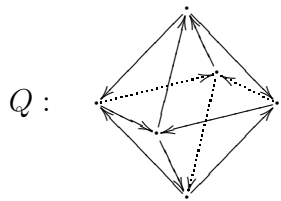
\xy
(0,0)*{}="1";(0,25)*{}="2";(0,12.5)*{Q:}="3";
\endxy
\ \ \ \
\xy
(0,25)*{\cdot}="1";(0,0)*{\cdot}="2";(-12,12.5)*{\cdot}="6";(-4,8.5)*{\cdot}="3";(4,16.5)*{\cdot}="5";(12,12.5)*{\cdot}="4";
{\ar@{->}"3";"1"};{\ar@{->}"3";"2"};{\ar@{->}"2";"4"};{\ar@{->}"1";"4"};{\ar@{->}"5";"1"};{\ar@{->}"1";"6"};{\ar@{->}"2";"6"};{\ar@{..>}"5";"2"};
{\ar@{->}"6";"3"};{\ar@{->}"4";"3"};{\ar@{..>}"4";"5"};{\ar@{..>}"6";"5"};
\endxy$ \ \ \ 
$\xy
(0,0)*{}="1";(0,25)*{}="2";(0,12.5)*{W \text{ given by (\ref{square superpotential}), with}}="3";
\endxy \ \ \ \
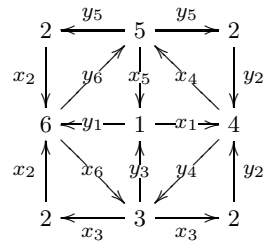
\xy
(0,0)*+{\text{\scriptsize{$2$}}}="1";(12.5,0)*+{\text{\scriptsize{$3$}}}="2";(25,0)*+{\text{\scriptsize{$2$}}}="3";
(0,12.5)*+{\text{\scriptsize{$6$}}}="4";(12.5,12.5)*+{\text{\scriptsize{$1$}}}="5";(25,12.5)*+{\text{\scriptsize{$4$}}}="6";
(0,25)*+{\text{\scriptsize{$2$}}}="7";(12.5,25)*+{\text{\scriptsize{$5$}}}="8";(25,25)*+{\text{\scriptsize{$2$}}}="9";
{\ar@{->}^{x_3}"2";"1"};{\ar@{->}_{x_3}"2";"3"};{\ar@{->}^{x_2}"1";"4"};{\ar@{->}|-{x_6}"4";"2"};{\ar@{->}|-{y_1}"5";"4"};{\ar@{->}|-{y_3}"2";"5"};{\ar@{->}|-{x_1}"5";"6"};{\ar@{->}_{y_2}"3";"6"};{\ar@{->}|-{y_4}"6";"2"};{\ar@{->}_{x_2}"7";"4"};{\ar@{->}_{y_5}"8";"7"};{\ar@{->}|-{y_6}"4";"8"};{\ar@{->}|-{x_5}"8";"5"};{\ar@{->}^{y_5}"8";"9"};{\ar@{->}^{y_2}"9";"6"};{\ar@{->}|-{x_4}"6";"8"};
\endxy$ \\
$A$ has infinite global dimension:
$$\cdots \rightarrow Ae_5 \oplus Ae_3 \stackrel{\cdot \left[ \begin{array}{cc} y_6 & x_4 \\ x_6 & y_4 \end{array} \right]}{\longrightarrow} Ae_6 \oplus Ae_4 \stackrel{\cdot \left[ \begin{array}{cc} y_1 & -x_2 \\ -x_1 & y_2 \end{array} \right]}{\longrightarrow} Ae_1 \oplus Ae_2 \stackrel{\cdot \left[ \begin{array}{cc} x_5 & y_3 \\ y_5 & x_3 \end{array} \right]}{\longrightarrow}$$ 
$$Ae_5 \oplus Ae_3 \stackrel{\cdot \left[ \begin{array}{cc} x_4 & -y_6 \\ -y_4 & x_6 \end{array} \right]}{\longrightarrow} Ae_4 \oplus Ae_6 \stackrel{\cdot \left[ \begin{array}{c} x_1 \\ y_1 \end{array} \right]}{\longrightarrow} Ae_1 \stackrel{\cdot 1}{\longrightarrow} V^1 \rightarrow 0.$$

$\delta_1$ satisfies
$$\left[ \begin{array}{cc} W_{x_5x_1} & W_{x_5y_1} \\ W_{y_3x_1} & W_{y_3y_1} \end{array} \right]= \left[ \begin{array}{cc} W_{y_5y_2} & W_{y_5x_2} \\ W_{x_3y_2} & W_{x_3x_2} \end{array} \right]$$
\end{itemize}
} \end{Example}

Next we consider a family of superpotential algebras where exactness fails in the third term, $\operatorname{ker}\delta_2 \not = 0$.  Each algebra has a nontrivial noetherian center (again 1 dimensional) and infinite global dimension.

\begin{Example} \label{examples 2} \rm{\
Let $Q$ be the cycle quiver, consisting of a single oriented cycle $c=a_n\cdots a_2a_1$, $a_i \in Q_1$, up to cyclic equivalence, and let $W \in k[c].$  Then $Z \cong k[c]/\left(\frac{dW}{dc} \right)$.\\
If not both $n =1$ and $W = c^2$ then the global dimension of $A$ is infinite:
$$\cdots \rightarrow Ae_{\operatorname{h}(a_{n-1})}
\stackrel{a_{n-1}}{\longrightarrow}
Ae_{\operatorname{t}(a_{n-1})}
\stackrel{\cdot W_{a_{n-1},a_n}}{\longrightarrow}
Ae_{\operatorname{h}(a_n)}
\stackrel{\cdot a_n}{\longrightarrow}
Ae_{\operatorname{t}(a_n)}$$
$$\stackrel{\cdot W_{a_n,a_1}}{\longrightarrow}
Ae_{\operatorname{h}(a_1)}
\stackrel{\cdot a_1}{\longrightarrow}
Ae_{\operatorname{t}(a_1)} \stackrel{ \cdot 1}{\longrightarrow} V^{\operatorname{t}(a_1)} \rightarrow 0.$$
} \end{Example}

\subsection{Endomorphism rings} \label{Endomorphism rings}

This section is based on joint work with Alex Dugas.  We show that a square superpotential algebra is an endomorphism ring of a reflexive module over its center.  For motivation, see \cite[section 4]{VdB}.\footnote{We give a partial account: A generalization of birationality is needed in order to view a homologically smooth noncommutative algebra as a resolution of its center.  Two varieties are birational precisely when they have isomorphic function fields; we may take the `function field' of a noncommutative algebra $A$ with prime center $Z$ to be  
$\operatorname{Frac}(Z) \otimes_Z A$.  If $X$ is an algebraic variety then $A$ and $k[X]$ are said to be \textit{birational} if their respective function fields are Morita equivalent, that is, $\operatorname{Frac}(Z) \otimes_Z A \cong \operatorname{End}_{k(X)}(k(X)^n)$ for some $n< \infty$, since requiring they be isomorphic is clearly too strong.  Morita equivalence therefore holds if (and only if) (i) $\operatorname{Frac}(Z) \cong k(X)$ (by comparing centers), and (ii) there exists a finitely-generated $Z$-module $M$ such that 
$$k(X) \otimes_Z A \cong \operatorname{End}_{k(X)} \left(k(X) \otimes_Z M\right),$$
since $k(X) \otimes_Z M$ is a finite dimensional $k(X)$-vector space, and this holds if $A \cong \operatorname{End}_Z \left( M\right)$.}

Note that $e_iAe_k$ is a $Z$-module for each $i \in Q_0$: if $z \in Z$, $a \in e_iAe_k$, then $za = ze_ia=e_iza \in e_iAe_k$.

\begin{Lemma} \label{isomorphism} 
Let $A$ be a square superpotential algebra.  Then for each $i,j,k \in Q_0$, there is an isomorphism
\begin{equation} \label{asteroid}
\begin{array}{ccc}
e_jAe_i & \stackrel{\cong}{\longrightarrow} & \operatorname{Hom}_Z\left(e_iAe_k,e_jAe_k\right)\\
d & \mapsto & f_d
\end{array}
\end{equation}
where $f_d(a)=da$. 
\end{Lemma}

\begin{proof}
\textit{Surjectivity:} Suppose $f \in \operatorname{Hom}_Z\left(e_iAe_k,e_jAe_k\right)$.  We want to show that there is some $d \in e_jAe_i$ such that $f = f_d$ is left multiplication by $d$, which then implies (\ref{asteroid}) is surjective.  

Fix an element $a \in e_iAe_k$ and a path $h \in e_kAe_i$.
By Theorem \ref{center}, $e_iAe_i = Ze_i$, so there is some $z \in Z$ such that $ze_i = ah$.
Similarly, for any $a' \in e_iAe_k$ there is some $z' \in Z$ such that $z'e_k = ha'$, whence
$$z'f(a) - zf(a') = f(z'a - za') = f(az' - za') = f(a(ha') - (ah)a') = f(0) = 0.$$
Thus, since $B$ is a domain,
\begin{equation} \label{space ship}
\frac{\bar{\tau}(a')\bar{\tau}(f(a))}{\bar{\tau}(a)} = \frac{\bar{\tau}(z')\bar{\tau}(f(a))}{\bar{\tau}(z)} = \bar{\tau}(f(a')) \in B.
\end{equation} 
It is clear that for each $w \in \left\{x_1,x_2,y_1,y_2\right\}$ there is a path from $k$ to $i$ whose $\bar{\tau}$-image is not divisible by $w$, since $Q$ embeds into a torus. 
Therefore, since (\ref{space ship}) holds for all $a' \in e_iAe_k$, it must be that $\bar{\tau}(a) | \bar{\tau}(f(a))$.
Set $m := \bar{\tau}(f(a))/ \bar{\tau}(a)$.

Write $a = \sum_{\ell = 1}^s \alpha_{\ell}$ and $f(a) = \sum_{\ell = 1}^t \beta_{\ell}$, where $\alpha_{\ell}, \beta_{\ell}$ are (scalar multiples of) paths.
Since $\bar{\tau}$ is $k$-linear,
$$m \bar{\tau}(\alpha_1) + \cdots + m \bar{\tau}(\alpha_s) = m \bar{\tau}(a) = \bar{\tau}(f(a)) = \bar{\tau}(\beta_1) + \cdots + \bar{\tau}(\beta_t).$$
Since $B$ is a polynomial ring and the $\bar{\tau}$-image of a path is a monomial, we have $s = t$, and by possibly re-indexing, $m \bar{\tau}(\alpha_{\ell}) = \bar{\tau}(\beta_{\ell})$. 
Since $\operatorname{t}(\alpha_{\ell}) = k = \operatorname{t}(\beta_{\ell})$, by Lemma \ref{star} there is some $d_{\ell} \in e_jAe_i$ such that $d_{\ell} \alpha_{\ell} = \beta_{\ell}$ and $\bar{\tau}(d_{\ell}) = m$.
By the injectivity of $\tau$, there is a unique path in $e_jAe_i$ with $\bar{\tau}$-image $m$, so it must be that $d_1 = \cdots = d_s =: d$.
Therefore $f(a) = da$.
Since $a$ was arbitrary, for any $b \in e_iAe_k$ we similarly have $f(a-b) = d(a-b)$.
This yields $f(b) = f(a-(a-b)) = f(a) - f(a-b) = db = f_d(b)$, proving our claim.

\textit{Injectivity:} Let $d \in e_jAe_i$ be nonzero.  Since $B$ is an integral domain, $da \not = 0$ for any nonzero $d \in e_iA$ by Corollary \ref{ba not 0}, so $f_d$ is injective, and in particular $f_d \not = 0$.
\end{proof}

\begin{Proposition} \label{endomorphism ring} 
Let $A$ be a square superpotential algebra.  Then for any $i \in Q_0$, $Ae_i$ is a reflexive $Z$-module and
$$A \cong \operatorname{End}_Z\left( Ae_i \right).$$
\end{Proposition}

\begin{proof} 
We first claim that for any $j \in Q_0$, $e_jAe_i$ is a reflexive $Z$-module, and so $Ae_i$ is a reflexive $Z$-module.  For $i,j \in Q_0$, 
$$\operatorname{Hom}_Z\left(e_iAe_j, Z \right) = \operatorname{Hom}_Z\left(e_iAe_j, Ze_j \right) = \operatorname{Hom}_Z\left(e_iAe_j,e_jAe_j \right) \cong e_jAe_i,$$
where the last isomorphism follows from Lemma \ref{isomorphism} with $k=j$.  Thus
$$\operatorname{Hom}_Z\left( e_iAe_j,Z \right) \cong e_jAe_i \ \ \ \text{ and } \ \ \ 
\operatorname{Hom}_Z\left( e_jAe_i, Z \right) \cong e_iAe_j,$$
proving our claim.  Furthermore,
$$\begin{array}{cclc}
A & = & \bigoplus_{j,k \in Q_0} e_kAe_j & \\
& \cong & \bigoplus_{j,k \in Q_0}\operatorname{Hom}_Z \left(e_jAe_i,e_kAe_i \right) & \text{ by Lemma \ref{isomorphism}}\\
& \cong & \operatorname{Hom}_Z \left( \bigoplus_j e_jAe_i, \bigoplus_k e_kAe_i \right) & \\
& = & \operatorname{End}_Z \left( Ae_i \right). &
\end{array}$$
\end{proof}

Note that Proposition \ref{endomorphism ring} holds after localization: 
$$A_{\mathfrak{m}} \cong Z_{\mathfrak{m}} \otimes_Z \operatorname{End}_Z(Ae_i) \cong \operatorname{End}_{Z_{\mathfrak{m}}}\left( A_{\mathfrak{m}}e_i \right).$$
In the following examples, recall that $R \subset B$ is isomorphic to $Z$.  We will denote by $R\left\{b_1,\ldots,b_n \right\}$ the (indecomposable) $R$-module minimally generated by $b_1,\ldots,b_n \in B$.

\begin{Example} \rm{
Consider the $Y^{4,0}$ algebra $A$ given in figure \ref{Y40}.  
$A$ is isomorphic to the endomorphism ring of the direct sum of the reflexive $R$-modules $R_i := \bar{\tau}(e_iAe_1) \subset B$, $i \in Q_0$, given in the figure.  
Note that the free $R$-module $R$ can be placed at any vertex.
\begin{figure}
$$\begin{array}{ccc}
\xy (-18.478,7.654)*+{8}="8";(-18.478,-7.654)*+{7}="7";(-7.654,-18.478)*+{6}="6";(7.654,-18.478)*+{5}="5";(18.478,-7.654)*+{4}="4";(18.478,7.654)*+{3}="3";(7.654,18.478)*+{2}="2";(-7.654,18.478)*+{1}="1";
{\ar@/^1pc/"1";"7"};{\ar@{->}"1";"3"};
{\ar@/^/"2";"3"};{\ar@/_.1pc/"2";"3"};{\ar@{->}"2";"4"};
{\ar@/_/"3";"8"};{\ar@/^/"3";"8"};{\ar@{->}"3";"5"};
{\ar@/^/"4";"5"};{\ar@/_.1pc/"4";"5"};
{\ar@/^/"5";"6"};{\ar@/_/"5";"6"};{\ar@/_/"5";"2"};{\ar@/^/"5";"2"};
{\ar@/^/"6";"7"};{\ar@/_/"6";"7"};{\ar@/^1pc/"6";"4"};
{\ar@/^/"7";"8"};{\ar@/_/"7";"8"};{\ar@/^1pc/"7";"5"};
{\ar@/^/"8";"1"};{\ar@/_.1pc/"8";"1"};{\ar@/^1pc/"8";"6"};{\ar@{->}"8";"2"};
\endxy & 
\ \ \ \ 
&
{\begin{array}{ccl} 
R_1 & := & R \\
R_2 & := & R\left\{x_1y_1^2y_2,x_2y_1^2y_2,x_1^3y_1^5,x_1^2x_2y_1^5,x_1x_2^2y_1^5,x_2^3y_1^5\right\}\\
R_3 & := & R\left\{y_1,x_1^2y_2^3,x_1x_2y_2^3,x_2^2y_2^3\right\}\\
R_4 & := & R\left\{x_1y_1y_2^3,x_2y_1y_2^3,x_1y_1^3y_2,x_2y_1^3y_2\right\}\\
R_5 & := & R\left\{y_1^2,y_2^2\right\}\\
R_6 & := & R\left\{x_1y_1y_2^2,x_2y_1y_2^2,x_1y_1^3,x_2y_1^3\right\}\\
R_7 & := & R\left\{y_2,x_1^2y_1^3,x_1x_2y_1^3,x_2^2y_1^3\right\}\\
R_8 & := & R\left\{x_1y_1y_2,x_2y_1y_2,x_1^3y_1^4,x_1^2x_2y_1^4,x_1x_2^2y_1^4,x_2^3y_1^4\right\}
\end{array}}
\end{array}$$
\caption{A $Y^{4,0}$ quiver and its corresponding $R$-modules $R_i:=\bar{\tau}(e_iAe_1) \subset B$.}
\label{Y40}
\end{figure}
}\end{Example}

\begin{Example} \rm{
The conifold quiver algebra $A$ given in example \ref{conifold} is the $Y^{1,0}$ algebra, with center
$$Z \cong R = k[x_1y_1, x_2y_2, x_1y_2, x_2y_1] \cong k[a,b,c,d]/(ab-cd).$$
It is standard \cite[section 1, example]{VdB2} to view $A$ as the endomorphism ring
$$A \cong \operatorname{End}_R(R \oplus I) = \left( \begin{array}{cc} R & I \\ I^{-1} & R \end{array} \right),$$
where $I = (a,c)=(x_1y_1,x_1y_2)$ and $I^{-1}= (a,d) = (x_1y_1,x_2y_1)$.
Our method realizes $A$ as a slightly different endomorphism ring:
$$A \cong \operatorname{End}_R\left(R \oplus R\left\{x_1,x_2\right\}\right) \cong \operatorname{End}_R\left(R \oplus R\left\{y_1,y_2\right\}\right).$$
} \end{Example}

We now prove the main result of this section.

\begin{Theorem} \label{nccr} 
Let $A$ be a square superpotential algebra, $Z$ its center, and $\mathfrak{m}$ the maximal ideal at the origin of $\operatorname{Max}Z$.  Then the localization $A_{\mathfrak{m}}$ is a noncommutative crepant resolution of $Z_{\mathfrak{m}}$, and consequently a local Calabi-Yau algebra of dimension 3.
\end{Theorem}

\begin{proof} 
The only simple $A_{\mathfrak{m}}$-modules are the vertex simples by Lemma \ref{localized vertex}, and so the theorem follows from Theorem \ref{Gorenstein}, Proposition \ref{q}, and Proposition \ref{endomorphism ring}.  
Moreover, by a result of Braun \cite[example 2.22]{Braun}, a noncommutative crepant resolution $A$ is locally Calabi-Yau if $k$ is algebraically closed and $Z$ is a normal Gorenstein finitely-generated $k$-algebra.
\end{proof}

\section{The $Y^{p,q}$ algebras} \label{The Ypq algebras}

We now consider a particular class of square superpotential algebras in detail, namely the $Y^{p,q}$ algebras defined in example \ref{Ypq}.  These algebras are conjecturally related to certain Sasaki-Einstein manifolds in the $\mathcal{N}=1$, $d=4$ AdS/CFT correspondence in string theory.

\subsection{Azumaya loci and (non-local) global dimensions} \label{Azumaya loci}

\begin{Proposition} \label{center prop} 
Let $Z$ be the center of a $Y^{p,q}$ algebra $A$.  Then there is some $0 \leq r \leq 2p$ such that
\begin{equation} \label{Ypq center} 
Z \cong k\left[ x_{\alpha_1}x_{\alpha_2}y_1y_2, \ y_1^p \prod_{\ell = 1}^{2p-r}x_{\beta_{\ell}}, \ y_2^p \prod_{\ell = 1}^{r}x_{\gamma_{\ell}} \ \mid \ \alpha_{\ell}, \beta_{\ell}, \gamma_{\ell} \in \{1,2\} \right],
\end{equation}
where in the McKay cases $r \in \{ 0, 2p \}$, set $y_{\alpha}^p \prod_{\ell = 1}^0 x_{\beta_{\ell}} := y_{\alpha}^p$.
\end{Proposition}

\begin{proof} 
For any $i \in Q_0$, $Z \cong Ze_i = e_iAe_i \cong \bar{\tau}(e_iAe_i) =: R$ by Theorem \ref{center} and Lemma \ref{bartau}.  
$R$ is therefore generated by the $\bar{\tau}$-images of cycles in $e_iAe_i$ without cyclic proper subpaths.
Denote by $R'$ the algebra on the right hand side of (\ref{Ypq center}).
Fix $i \in Q_0$ and $(0,0) \in \pi^{-1}(i) \subset \widetilde{Q}_0 = \mathbb{Z}^2$.

We first claim that $R' \subseteq R$.
As is clear from figure \ref{cases}, a vertex is the tail of an arrow with $\bar{\tau}$-image $x_1$ (resp.\ $x_1y_{\beta}$) if and only if it is also the tail of an arrow with $\bar{\tau}$-image $x_2$ (resp.\ $x_2y_{\beta}$).
Therefore, since the width of the fundamental domain of $Q$ is 2, if $c$ is a cycle whose $\bar{\tau}$-image is divisible by $x_{\alpha}$, then there is a cycle $c'$ satisfying
\begin{equation} \label{c'}
\bar{\tau}(c') = \bar{\tau}(c) \frac{x_{\alpha +1}}{x_{\alpha}}.
\end{equation}

Since the width of the fundamental domain of $Q$ is 2, either $(0, p) \in \pi^{-1}(i)$ or $(1,p) \in \pi^{-1}(i)$.
Suppose $(0,p) \in \pi^{-1}(i)$ (resp.\ $(1,p) \in \pi^{-1}(i)$), and let $c$ be a path in $\widetilde{Q}$ from $(0,0)$ to $(0,p)$ (resp.\ $(1,p)$) without cyclic proper subpaths.
Then $\pi(c) \in e_iAe_i$ is a cycle with $\bar{\tau}$-image $y_1^p x_1^s x_2^t$, where $s = t \leq p$ (resp.\ $s = t+1 \leq p$). 
Set $r := s+t \leq 2p$.
Since $s,t \leq p$, we have $\bar{\tau}(c) \mid \sigma^p$, and so by Lemma \ref{star} there exists a cycle $d$ in $Q$ satisfying $\bar{\tau}(d) = y_2^p x_1^{p-s}x_2^{p-t}$.
Furthermore, $(p - s) + (p-t) = 2p -r$.
Thus, by (\ref{c'}) there are cycles in $Q$ with $\bar{\tau}$-images $y_1^p \prod_{\ell = 1}^{2p-r} x_{\beta_{\ell}}$ and $y_2^p \prod_{\ell = 1}^rx_{\gamma_{\ell}}$ for each $\beta_{\ell}, \gamma_{\ell} \in \left\{ 1,2 \right\}$.

Finally, since the unit cycle at $i$ has $\bar{\tau}$-image $x_1x_2y_1y_2$, (\ref{c'}) implies that there are cycles with $\bar{\tau}$-images $x_1^2y_1y_2$ and $x_2^2y_1y_2$.

We now claim that $R \subseteq R'$.
It is sufficient to determine the $\bar{\tau}$-images of all the cycles in $e_iAe_i$ without cyclic proper subpaths, as these images form a generating set for $R$.
Denote by $J$ the set of vertices $j \in \pi^{-1}(i)$ for which there is a path $c^+$ in $\widetilde{Q}$ from $(0,0)$ to $j$ whose projection $c$ is a cycle in $Q$ without cyclic proper subpaths.
Since the lift of a cycle in $Q$ without cyclic proper subpaths is a path in $\widetilde{Q}$ without cyclic proper subpaths, such a path $c^+$ from $(0,0)$ to $j$ is unique (modulo $\partial W$) by Lemma \ref{star star}.
Therefore there is a bijection between the vertices in $J$ and the cycles in $e_iAe_i$ without cyclic proper subpaths, minus the unit cycle at $i$.

Set $u := (2,0)$ and $v := (v_1,p)$, with $v_1 \in \{0,1\}$ chosen so that $\pi^{-1}(i) = \mathbb{Z}u \oplus \mathbb{Z}v$.  
For $j \in \pi^{-1}(i)$, write $j = j_1 u + j_2 v$.
We claim that 
$$J = \left\{ \pm u, \ j_1u \pm v \ \mid \ -p \leq j_1 \leq p \right\}.$$

If $j_2 = 0$ then it is clear that $j_1 = \pm 1$.

Now suppose $j_2 \geq 1$.
Then there is a cycle $c$ without cyclic proper subpaths whose lift $c^+$ has height $j_2p$.
For each $0 \leq m \leq j_2$, there is some $s_m \in \mathbb{Z}$ such that $(s_m, mp) \in \widetilde{Q}_0$ is a vertex subpath of $c^+$.
Since the width of the fundamental domain of $Q$ is 2 and $c$ has no cyclic proper subpaths, 
\begin{equation} \label{pm 1}
\pi((s_m \pm 1,mp)) = \pi((0,0)) = \operatorname{t}(c) \ \text{ for } \ 0 < m < j_2.
\end{equation}
Therefore $j_2 \leq 2$.

Suppose to the contrary that $j_2 = 2$.
Let $d_1$ and $d_2$ be paths in $\widetilde{Q}$ without cyclic proper subpaths, respectively from $(0,0)$ to $(s_1+1,p)$, and from $(s_1+1,p)$ to $(s_2,2p)$.  
By (\ref{pm 1}), the projections $\pi(d_1)$ and $\pi(d_2)$ are cycles in $e_iAe_i$ whose $\bar{\tau}$-images are of the form $y_1^p \prod_{\ell = 1}^{2p-r}x_{\beta_{\ell}}$ in $R'$.
Since $B$ is a polynomial ring and $y_2$ does not divide $\bar{\tau}(d_1)$ and $\bar{\tau}(d_2)$ in $B$, $y_2$ does not divide their product $\bar{\tau}(d_2d_1) = \bar{\tau}(d_1) \bar{\tau}(d_2)$.
Therefore $\sigma$ does not divide $\bar{\tau}(d_2d_1)$.
Thus, by Lemma \ref{cycle in cover} $d_2d_1$ has no cyclic proper subpaths (in $\widetilde{Q}$).
Therefore, since $c^+$ and $d_2d_1$ have coincident heads and tails in $\widetilde{Q}$ and no cyclic proper subpaths, Lemma \ref{star star} implies
$$c \sim \pi(d_2d_1) = \pi(d_2)\pi(d_1).$$ 
This contradicts our assumption that $c$ has no cyclic proper subpaths.
Therefore $j_2$ must equal 1, and so for some $s,t \geq 0$,
$$\bar{\tau}(c) = y_1^p x_1^s x_2^t.$$

Suppose $t = 0$, so that $j = \frac s2 u + v$, and in particular $j_1 = \frac s2$.
Consider a subpath $ba$ of $c$, where $a$ and $b$ are arrows.
If $\bar{\tau}(a) = x_1$ then $\bar{\tau}(b) \not = x_1$ since the width of the fundamental domain of $Q$ is 2 and $\widetilde{Q}$ has vertical symmetry.
Thus $s$ can be at most $2p$; the maximum $s = 2p$ occurs if there is a path $c = b_pa_p\cdots b_2a_2b_1a_1$, where $a_{\ell}$ and $b_{\ell}$ are arrows satisfying $\bar{\tau}(a_{\ell}) = x_1$ and $\bar{\tau}(b_{\ell}) = x_1y_1$.
The cases $t \not = 0$ are obtained similarly by applying (\ref{c'}).

Finally, the case $j_2 \leq -1$ is similar to the case $j_2 \geq 1$, proving our claim.
\end{proof}

\begin{Lemma} \label{isolated singularity} 
If $p \not = q$, then the only singular point in $\operatorname{Max}Z$ is the origin.
\end{Lemma}

\begin{proof} 
Referring to (\ref{Ypq center}), set $s := 2p - r$ and denote by $\mathcal{G}$ the set of generators of $R = \bar{\tau}(e_iAe_i) \subset B$,
$$\begin{array}{rcl}
t_0 & = & y_1^p x_1^s \\
t_1 & = & y_1^p x_1^{s-1}x_2 \\
 & \vdots & \\
t_s & = & y_1^p x_2^s
\end{array} 
\ \ \ \ 
\begin{array}{rcl}
u_0 & = & y_2^p x_1^r \\
u_1 & = & y_2^p x_1^{r-1}x_2 \\
& \vdots & \\
u_r & = & y_2^p x_2^r
\end{array}
\ \ \ \ 
\begin{array}{rcl}
v_1 & = & x_1^2y_1y_2 \\
v_2 & = & x_1x_2 y_1 y_2 \\
v_3 & = & x_2^2 y_1y_2
\end{array}$$
Since $p \not = q$, we have $0 < r < 2p$.
Thus all the coordinate functions $g(x_1,x_2,y_1,y_2) \in \mathcal{G}$ vanish if $x_1 = x_2 = 0$ or $y_1 = y_2 = 0$.
In particular, the only ideal $\mathfrak{m}$ in $\operatorname{Max}R$ containing both $x_1 B \cap R$ and $x_2 B \cap R$, or both $y_1 B \cap R$ and $y_2 B \cap R$, is the origin $\left( g \ | \ g \in \mathcal{G} \right)R \in \operatorname{Max}R$.
Therefore it suffices to show that any point $\mathfrak{m} \in \operatorname{Max}R$ is smooth if it does not contain $x_{\alpha} B \cap R$ and $y_{\beta} B \cap R$ for some $\alpha, \beta \in \{1,2\}$.  
So without loss of generality suppose $\mathfrak{m}$ does not contain $x_1 B \cap R$ and $y_1 B \cap R$; in particular, $t_0 = x_1^sy_1^p \not \in \mathfrak{m}$, that is, $t_0(\mathfrak{m}) \not = 0$.

Denote by $\mathcal{R}$ a minimal generating set for the relations among the coordinate variables $g \in \mathcal{G}$.
Then by abuse of notation, $R \cong k\left[ \mathcal{G} \right]/\left( \mathcal{R} \right)$.
Consider the submatrix $K$ of the Jacobian of $R$ at $\mathfrak{m}$,
$$J(\mathfrak{m}) = \left[ \frac{\partial g}{\partial r} (\mathfrak{m}) \right]_{(g, r) \in \mathcal{G} \times \mathcal{R}},$$
given by Table (\ref{partial}).
In the table, for each $0 \leq n \leq r$ the indices $i_n, j_n \in \left\{ 1,2,3 \right\}$ are suitably chosen and the exponents $k_n, \ell_n$ satisfy $k_n + \ell_n = 2p$.
\begin{table}
$$\begin{array}{c||ccc|c|ccc|}
& \partial_{t_2} & \cdots & \partial_{t_s} & 
\begin{array}{cc} \partial_{v_2} & \partial_{v_3} \end{array} & 
\partial_{u_0} & \cdots & \partial_{u_r}\\
\hline
\hline
t_0t_2 -t_1t_1      & t_0   &        & 0   &   & &   & \\
\vdots              &       & \ddots &     & 0 & & 0 & \\
t_0t_s - t_1t_{s-1} & \star &        & t_0 &   & &   & \\
\hline
\begin{array}{r} t_0v_2 - t_1v_1 \\ t_0v_3 - t_2v_1 \end{array} 
                    &       &  \star &     & 
\begin{array}{cc} t_0 & 0 \\ 0 & t_0 \end{array}
                                               & & 0 & \\ 
\hline
t_0u_0-v_{i_0}^{k_0}v_{j_0}^{\ell_0} & &   & &       & t_0 &        & 0 \\
\vdots                               & & 0 & & \star &     & \ddots &     \\
t_0u_r-v_{i_r}^{k_r}v_{j_r}^{\ell_r} & &   & &       &  0  &        & t_0   
\end{array}$$
\caption{The partial derivatives of relations specifying the square submatrix $K$ of the Jacobian.}
\label{partial}
\end{table}
Since $K$ is a lower triangular $(2p + 2) \times (2p + 2)$ square matrix with nonzero diagonal entries $t_0(\mathfrak{m})$, the rank of $K$ is $2p + 2$.
But the rank of $J(\mathfrak{m})$ is at most the dimension of the ambient space $k[\mathcal{G}]$ minus the dimension of $R$, namely $(2p+5) - 3 = 2p+2$, so the rank of $J(\mathfrak{m})$ is precisely $2p+2$.
Therefore $\mathfrak{m}$ is a smooth point of $\operatorname{Max}R$.
The case where $x_1$ and $y_2$ (resp.\ $x_2, y_1$; $x_2, y_2$) are nonzero is similar with $u_0$ (resp.\ $t_s$; $u_r$) in place of $t_0$.
\end{proof}

\begin{figure}
$$(i) \ \ \ \ 
\xy (-10,-5)*+{\text{\scriptsize{$s_1$}}}="1";(0,-5)*+{\text{\scriptsize{$s_2$}}}="2";(10,-5)*+{\text{\scriptsize{$s_1$}}}="3";(-10,5)*+{\text{\scriptsize{$t_1$}}}="4";(0,5)*+{\text{\scriptsize{$t_2$}}}="5";(10,5)*+{\text{\scriptsize{$t_1$}}}="6";
{\ar"1";"2"};{\ar@{.>}"3";"2"};{\ar@{.>}"4";"1"};{\ar@{.>}"6";"3"};{\ar@{.>}"5";"2"};{\ar@{.>}"2";"4"};{\ar"2";"6"};{\ar"4";"5"};{\ar@{.>}"6";"5"};
\endxy
\ \ \ \ \ \ \
(ii) \ \ \ \
\xy (-10,-5)*+{\text{\scriptsize{$s_1$}}}="1";(0,-5)*+{\text{\scriptsize{$s_2$}}}="2";(10,-5)*+{\text{\scriptsize{$s_1$}}}="3";(-10,5)*+{\text{\scriptsize{$t_1$}}}="4";(0,5)*+{\text{\scriptsize{$t_2$}}}="5";(10,5)*+{\text{\scriptsize{$t_1$}}}="6";
{\ar"1";"4"};{\ar"4";"5"};{\ar@{.>}"5";"2"};{\ar@{.>}"2";"1"};{\ar@{.>}"2";"3"};{\ar"3";"6"};{\ar@{.>}"6";"5"};
\endxy
\ \ \ \ \ \ \
(iii) \ \ \ \ 
\xy (-10,-5)*+{\text{\scriptsize{$s_1$}}}="1";(0,-5)*+{\text{\scriptsize{$s_2$}}}="2";(10,-5)*+{\text{\scriptsize{$s_1$}}}="3";(-10,5)*+{\text{\scriptsize{$t_1$}}}="4";(0,5)*+{\text{\scriptsize{$t_2$}}}="5";(10,5)*+{\text{\scriptsize{$t_1$}}}="6";
{\ar"1";"4"};{\ar"4";"5"};{\ar@{.>}"5";"1"};{\ar@{.>}"1";"2"};{\ar@{.>}"2";"5"};{\ar@{.>}"5";"3"};{\ar@{.>}"3";"2"};{\ar"3";"6"};{\ar@{.>}"6";"5"};
\endxy$$
$$(iv) \ \ \ \
\xy (-10,-5)*+{\text{\scriptsize{$s_1$}}}="1";(0,-5)*+{\text{\scriptsize{$s_2$}}}="2";(10,-5)*+{\text{\scriptsize{$s_1$}}}="3";(-10,5)*+{\text{\scriptsize{$t_1$}}}="4";(0,5)*+{\text{\scriptsize{$t_2$}}}="5";(10,5)*+{\text{\scriptsize{$t_1$}}}="6";
{\ar@{<.}"1";"2"};{\ar@{<=}"3";"2"};{\ar@{<-}"4";"1"};{\ar@{<-}"6";"3"};{\ar@{<=}"5";"2"};{\ar@{<.}"2";"4"};{\ar@{<.}"2";"6"};{\ar@{<.}"4";"5"};{\ar@{<=}"6";"5"};
\endxy
\ \ \ \ \ \ \
(v) \ \ \ \ 
\xy (-10,-5)*+{\text{\scriptsize{$s_1$}}}="1";(0,-5)*+{\text{\scriptsize{$s_2$}}}="2";(10,-5)*+{\text{\scriptsize{$s_1$}}}="3";(-10,5)*+{\text{\scriptsize{$t_1$}}}="4";(0,5)*+{\text{\scriptsize{$t_2$}}}="5";(10,5)*+{\text{\scriptsize{$t_1$}}}="6";
{\ar@{<.}"1";"4"};{\ar@{<.}"4";"5"};{\ar@{<-}"5";"2"};{\ar@{<-}"2";"1"};{\ar@{<.}"2";"3"};{\ar@{<.}"3";"6"};{\ar@{<-}"6";"5"};
\endxy
\ \ \ \ \ \ \
(vi) \ \ \ \
\xy (-10,-5)*+{\text{\scriptsize{$s_1$}}}="1";(0,-5)*+{\text{\scriptsize{$s_2$}}}="2";(10,-5)*+{\text{\scriptsize{$s_1$}}}="3";(-10,5)*+{\text{\scriptsize{$t_1$}}}="4";(0,5)*+{\text{\scriptsize{$t_2$}}}="5";(10,5)*+{\text{\scriptsize{$t_1$}}}="6";
{\ar@{<.}"1";"4"};{\ar@{<.}"4";"5"};{\ar@{<-}"5";"1"};{\ar@{<.}"1";"2"};{\ar@{<.}"2";"5"};{\ar@{<.}"5";"3"};{\ar@{<.}"3";"2"};{\ar@{<.}"3";"6"};{\ar@{<-}"6";"5"};
\endxy$$
\caption{}
\label{cases}
\end{figure}

\begin{Theorem} \label{hom3} 
Let $A$ be a (non-localized) $Y^{p,q}$ algebra.  Then the following hold.
\begin{enumerate}
 \item If $p \not = q$ and $V$ is a simple $A$-module, then $V$ is either a vertex simple module or a large module.
 \item The Azumaya locus of $A$ coincides with the smooth locus of $Z$.
 \item $A$ is homologically homogeneous of global dimension 3.
\end{enumerate}
\end{Theorem}

\begin{proof}
(1) We only need to consider case (2) in Theorem \ref{simples}.  
Suppose the cycle $c \in e_iAe_i$ does not annihilate $V$, and suppose $c$ has no cyclic proper subpaths.  
Since $c$ is a path, $\bar{\tau}(c)$ cannot be of the form $x_1^sx_2^t$ or $y_1^sy_2^t$ since the underlying graph of $Q$ embeds into a surface.
Furthermore, since $p \not = q$, $Q$ is not McKay, and so without loss of generality we may assume $\bar{\tau}(c) = x_1^sy_1^t$ for some $s,t \geq 1$ by Proposition \ref{triangular}; the other cases are similar.  

We claim that each vertex in $Q$ is a subpath of $c$ modulo $\partial W$.
Referring to figure \ref{cases}, in all 6 cases at least one of $s_1$ or $s_2$ is a vertex subpath of $c$ since the width of the fundamental domain of $Q$ is 2.
So suppose $s_1$ is a vertex subpath of $c$.
The cyclic permutation $c_{s_1}$ of $c$ at $s_1$ also has $\bar{\tau}$-image $\bar{\tau}(c) = x_1^sy_1^t$.
Observe that in all cases except (iv), there is a path $a$ denoted by solid arrows from $s_1$ that passes through both vertices $t_1$ and $t_2$, whose $\bar{\tau}$-image is only divisible by $x_1$ and $y_1$.

(a) Consider all cases except (iv).
By Proposition \ref{astronaut} there is a unique cycle in $e_{s_1}Ae_{s_1}$ without cyclic proper subpaths whose $\bar{\tau}$-image is of the form $x_1^sy_1^t$, namely $c_{s_1}$, so $a$ must be a subpath of $c_{s_1}$ (modulo $\partial W$).
Thus $a$ is also a subpath of the cyclic permutation $c$ of $c_{s_1}$.  
Therefore $t_1$ and $t_2$ are both vertex subpaths of $c$.

(b) Now consider case (iv).
Observe that (iv), (v), or (vi) must be directly below (iv).

First suppose either (v) or (vi) is directly below (iv) in $\widetilde{Q}$.
By (a), in both (v) and (vi) the arrow from $t_2$ to $t_1$ is a subpath of $c$.  
But since (iv) is directly above, this arrow is the arrow from $s_2$ to $s_1$ in (iv), and so all the bold arrows in (iv), as well as the arrow from $s_1$ to $t_1$, are subpaths of $c$.
Therefore the vertices $t_1$ and $t_2$ are subpaths of $c$ as well.  

Now suppose (iv) is directly below (iv).
Since $Q$ is not McKay, $Q$ cannot consist entirely of `building blocks' of the form (iv).
Thus there is a row of the form (v) or (vi) below (iv) in $\widetilde{Q}$.
It therefore follows by induction that both $t_1$ and $t_2$ are subpaths of $c$.

(c) By (a) and (b), for each $j \in Q_0$ there are paths $c_1$ and $c_2$ such that $c = c_2e_jc_1$.
Thus if $c$ does not annihilate an $A$-module $V$ then $e_j$ also does not annihilate $V$. 
This implies $\operatorname{dim}_ke_jV \geq 1$.  
Therefore $\operatorname{dim}_ke_jV = 1$ for each $j \in Q_0$ by Lemma \ref{leq 1}, whence $\operatorname{dim}_kV = |Q_0|$, and so $V$ is large.

(2) First suppose $p \not =q$.  
Let $V$ be a simple $A$-module and set $\mathfrak{m} := \operatorname{ann}_ZV \in \operatorname{Max}Z$.  
We have just shown that $V$ is either a vertex simple module or a large module.  
Therefore, if $V$ is (not) a large module then $\mathfrak{m}$ is (not) in the Azumaya locus of $A$ by Lemma \ref{Azumaya locus of square algebras}, and $\mathfrak{m}$ is (not) a smooth point of $\operatorname{Max}Z$ by Lemma \ref{isolated singularity}.  
Thus the Azumaya locus is the open dense subset $\operatorname{Max}Z \setminus \{0\}$.

The case $p=q$ is similar, noting that there are clearly two distinct isoclasses of simple modules whose $Z$-annihilators are in the locus $\phi\left( \{x_1 = x_2 =0\} \right)$ (with $\phi$ as in Lemma \ref{impression}(\ref{U'})).  
In this case the Azumaya locus is $\operatorname{Max}Z \setminus \phi \left( \{x_1=x_2=0\} \right)$.

(3) The $Y^{p,p}$ algebras are McKay quiver algebras for certain finite abelian subgroups of $\operatorname{SL}_3(\mathbb{C})$, and the claim is well known in this case.  
So suppose $p \not = q$.  
If $V$ is a vertex simple $A$-module then $\operatorname{pd}_A(V) = 3$ by Proposition \ref{q}.  
If $V$ is a non-vertex simple, then $V$ is large by (1), so $\mathfrak{m} = \operatorname{ann}_ZV$ is in the Azumaya locus by Lemma \ref{Azumaya locus of square algebras}, so $\mathfrak{m}$ is in the smooth locus of $\operatorname{Max}Z$ by (2).
Thus $\operatorname{pd}_{Z_{\mathfrak{m}}}(Z_{\mathfrak{m}}/\mathfrak{m}_{\mathfrak{m}}) = 3$ by Proposition \ref{Krull dim}, whence $\operatorname{pd}_A(V) = 3$ by Corollary \ref{Large modules}(3).  
By \cite[Proposition III.6.7(a)]{Ba}, if $S$ is a noetherian ring module-finite over its center then $\operatorname{gl.dim}(S) = \operatorname{sup} \left\{ \operatorname{pd}_S(M) \ | \ M \text{ simple} \right\}$.  
But $A$ has these properties by Corollary \ref{square prime}, so $A$ has global dimension 3.  Also by Corollary \ref{square prime}, $\operatorname{Max}Z$ is irreducible, hence equidimensional, and thus $A$ is homologically homogeneous.
\end{proof}

To conclude this section, we show that the `$R$-charge' of an arrow determined by $a$-maximization is consistent with its impression given in Theorem \ref{square impression}.  Let $A$ be a superpotential algebra module-finite over its center $Z$; then the $R$-charge of an arrow $a \in Q_1$ is conjectured to be the volume of the `zero locus' of $a$ in $\operatorname{Max}Z$, that is, the locus consisting of the maximal ideals $\mathfrak{m} \in \operatorname{Max}Z$ such that $a \in \mathfrak{m}_{\mathfrak{m}}A_{\mathfrak{m}}$.  In physics terms, the $R$-charge of a field is conjectured to be the volume of the locus where symmetry is not broken in the vev moduli space.  This has been explored in \cite{BFZ, BB}, for example.  We verify that when $A$ is a $Y^{p,q}$ algebra the labeling of arrows given in figure \ref{labels} is consistent with the numerical $R$-charge assignments determined by $a$-maximization \cite{IW}, as first computed for the $Y^{2,1}$ quiver in \cite{BBC}, and then for general $(p,q)$ in \cite{BFHMS}.  

\begin{Proposition} 
The (numerical) $R$-charge assignments of the arrows in a $Y^{p,q}$ quiver determined by $a$-maximization are consistent with the labels given in figure \ref{labels}.
\end{Proposition}

\begin{proof}  
Denote an arrow $a$ by its label $\bar{\tau}(a)$ given in figure \ref{labels}.  The $R$-charge assignments as computed in \cite{BHK} are shown in table \ref{R charge}.
\begin{table}
$$\begin{array}{rcl}
R(x_1 y_1) = R(x_2 y_1) & = & (3q-2p+\sqrt{4p^2-3q^2})/3q\\
 &=& \frac 13 \left( -1 + \sqrt{13} \right)\\
R(x_1) = R(x_2) = & = & 2p(2p-\sqrt{4p^2-3q^2})/3q^2\\
 & = & \frac 43 \left( 4 - \sqrt{13} \right)\\
R(y_2) & = & (-4p^2+3q^2+2pq+(2p-q)\sqrt{4p^2-3q^2})/3q^2\\
 & = & -3 + \sqrt{13}\\
R(y_1) &= &(-4p^2+3q^2-2pq+(2p+q)\sqrt{4p^2-3q^2})/3q^2\\
 & = & \frac 13 \left( -17+ 5 \sqrt{13} \right)\\
R(x_1 y_2) = R(x_2y_2) & = & (3q +2p-\sqrt{4p^2-3q^2})/3q
\end{array}$$
\caption{The $Y^{p,q}$ $R$-charge assignments determined from $a$-maximization.}
\label{R charge}
\end{table}
To check consistency, one verifies that $R\left(x_{\alpha}y_{\beta}\right) = R\left(x_{\alpha}\right) + R\left(y_{\beta}\right)$ in each of the four cases $\alpha, \beta \in \{1,2\}$. 
\end{proof}

\subsection{Exceptional loci with zero volume: a proposal} \label{Exceptional divisors with zero volume}

Let $A$ be a $Y^{p,q}$ algebra with center $Z$ and let $V^i$ be the vertex simple $A$-module at $i \in Q_0$.  The origin $\mathfrak{m}$ of $\operatorname{Max}Z$ is then in the compliment of the Azumaya locus--the \textit{ramification locus}--by Theorem \ref{hom3}, and clearly $A/\mathfrak{m}A \cong \bigoplus_{i \in Q_0} V^i$.  In Theorem \ref{nccr} we found that $A_{\mathfrak{m}}$ is a noncommutative crepant resolution of $Z_{\mathfrak{m}}$, and so the points in the exceptional locus of the noncommutative resolution should in principle be the simple $A_{\mathfrak{m}}$-modules, which are the vertex simples $V^i_{\mathfrak{m}}$.  It is important to note that quiver stability does not appear sufficient to capture these points since $\bigoplus_{i \in Q_0} V^i$ is not stable for any stability parameter with dimension vector $(1,\ldots,1)$.  

In Proposition \ref{q} we showed that the vertex simple $A_{\mathfrak{m}}$-modules are smooth in the sense that $\operatorname{pd}_{A_{\mathfrak{m}}}(V_{\mathfrak{m}}^i)=3 = \operatorname{dim}Z_{\mathfrak{m}}$ for each $i \in Q_0$.  In this section we introduce a proposal that provides a geometric reason for this behavior.  Specifically, we propose that certain points in $\operatorname{Max}A$ that sit over the ramification locus of $\operatorname{Max}Z$ are the irreducible components of the exceptional locus of a resolution $Y \rightarrow \operatorname{Max}Z$ shrunk to zero size.  In physics terms, fractional branes probing a singularity see the variety they are immersed in as smooth since they are wrapping exceptional divisors that have been shrunk to point-like spheres.  In what follows we use the symplectic quotient construction on the impression of a $Y^{p,q}$ algebra, and set $k = \mathbb{C}$.

In Corollary \ref{square prime}, $\operatorname{Max}Z$ was shown to be a toric algebraic variety.  $Z$ is therefore the ring of invariants of some torus action on $B=\mathbb{C}[x_1,x_2,y_1,y_2]$, which we determine in the following lemma.  It is straightforward to verify with Proposition \ref{center prop}.

\begin{Lemma} 
The center $Z \subset B$ of a $Y^{p,q}$ algebra is the ring of invariants of the torus action
\begin{equation} \label{torus action}
\left(x_1,x_2,y_1,y_2 \right) \mapsto \left( \lambda^{-p}\omega^{-1}x_1, \ \lambda^{-p}\omega^{-1}x_2, \ \lambda^{2p-r} \omega^2y_1, \ \lambda^{r}y_2 \right)
\end{equation}
with torus $\mathbb{C}^* \times \mu_r \ni (\lambda, \omega)$, for some $0 \leq r \leq 2p$.
\end{Lemma}

In the special case $(p,q)=(1,0)$, $r = 1$ by Proposition \ref{center prop}, whence $\omega=1$.
Therefore $Z$ is the coordinate ring for the conifold (i.e., quadric cone) given in example \ref{conifold}.  
Moreover, in the case $(p,q)=(2,1)$ again $r = 1$, yielding $\omega =1$, and it is straightforward to check that $Z$ is the complex cone over the first del Pezzo surface $dP_1$ (i.e., $\mathbb{CP}^2$ blownup at one point), verifying an argument that this should indeed be the case given in \cite[section 2]{BHOP}.\footnote{Martelli and Sparks proved that the real cone over the $Y^{2,1}$ manifold is the complex cone over $dP_1$ \cite{MS}, and so it follows that the real cone over the $Y^{2,1}$ manifold coincides with the maximal spectrum of the $Y^{2,1}$ algebra away from the origin.}

Before considering the associated $Y^{p,q}$ moment map, recall the following standard construction.  The symplectic manifold $\left(\mathbb{C}^2, \omega = \frac {i}{2} \left( dx \wedge d\bar{x} + dy \wedge d\bar{y} \right) \right)$ admits a hamiltonian action of the maximal compact subgroup $\mathbb{T} := \{ t \in \mathbb{C}^* \ | \ |t| =1\}$ of $\mathbb{C}^*$ given by $(x,y) \mapsto (tx,ty)$.  The dual $\mathtt{g}^*$ of the Lie algebra of $\mathbb{T}$ is then $\mathbb{R}$, so there is a moment map
$$\mu: \mathbb{C}^2 \rightarrow \mathtt{g}^* \cong \mathbb{R}, \ \ \ \ \ \mu(x,y) = \frac 12 \left( |x|^2 + |y|^2 \right).$$
It follows that
$$\mu^{-1}(1/2)/\mathbb{T} = \left\{ (x,y) \in M \ | \ |x|^2+|y|^2=1 \right\}/\mathbb{T} = \left\{ \mathbb{CP}^1 \text{ with radius } 1 \right\},$$
and more generally
$$\mu^{-1}(|a|^2/2 )/\mathbb{T} = \left\{ (a x, a y) \in M \ | \ |x|^2 + |y|^2 = 1 \right\}/\mathbb{T} = \left\{ \mathbb{CP}^1 \text{ with radius } |a| \right\}.$$
Varying $|a|$ is then equivalent to varying the radius of the $\mathbb{CP}^1$.\footnote{In physics terms, in the Lagrangian of an $\mathcal{N}=1$ physical theory, $a$ is a Fayet-Iliopoulos parameter and the moment map constraints are the D-terms; see for example \cite{MP}, and in the case of the $Y^{p,q}$ manifolds (starting from a metric), \cite{MS}.}  In particular, $|a| \rightarrow 0$ is equivalent to the radius vanishing.

Now since the center of a $Y^{p,q}$ algebra is a normal toric variety, it is also a symplectic variety with a (non-degenerate) symplectic form obtained by pulling back the standard symplectic form on $\operatorname{Max}B = \mathbb{C}^4$.  There is a hamiltonian action on $\operatorname{Max}B$ by the maximal compact subgroup $\mathbb{T}:=U(1) \times \mu_r \ni (t, \omega)$,
$$(x_1,x_2,y_1,y_2) \mapsto \left( t^{-p} \omega^{-1} x_1,t^{-p}\omega^{-1} x_2,t^{2p-r}\omega^2 y_1,t^{r}y_2 \right).$$
Again the dual of the Lie algebra of $\mathbb{T}$ is $\mathtt{g}^* \cong \mathbb{R}$, and so there is a moment map
$$\mu: \operatorname{Max}B \rightarrow \mathbb{R}, \ \ \ \ \mu(x_1,x_2,y_1,y_2) = \frac 12 \left(-p|x_1|^2 -p|x_2|^2 +(2p-r)|y_1|^2 +(r)|y_2|^2 \right).$$
The singular variety $\operatorname{Max}Z$ is then the symplectic reduction at the origin, 
$$\operatorname{Max}Z = \mu^{-1}(0)/\mathbb{T},$$
and two different blowups of $\operatorname{Max}Z$ are given by $\mu^{-1}(\xi /2)/\mathbb{T}$ for $\xi > 0$ and $\xi < 0$, respectively.  From the previous example, $\sqrt{|\xi|} \in \mathtt{g}^*$ may be viewed as the radius of the exceptional locus in the respective blowup of $\operatorname{Max}Z$.

For the following, suppose $p\not =q$ (the case $p=q$ is similar).  Let $\mathcal{M}_{\xi}$ denote the space of all representations $\tau_{\mathfrak{m}}$ where $m \in \mu^{-1}(\xi^2/2) \subset \operatorname{Max}B$.  We find the following:
\begin{itemize}
 \item If $\xi >0$, $\mathcal{M}_{\xi}$ is parameterized by a blowup of $\operatorname{Max}Z$ at the origin (a $\mathbb{CP}^1$ family together with $\operatorname{Max}Z\setminus\{0\}$);
 \item If $\xi < 0$, $\mathcal{M}_{\xi}$ is parameterized by the flopped blowup; and
 \item If $\xi = 0$, $\mathcal{M}_0$ is parameterized $\operatorname{Max}Z\setminus\{0\}$, together with the direct sum of vertex simples.
\end{itemize}
Specifically, $\xi = 0$ determines the constraints
$$\begin{array}{cclcc}
x_1 = x_2 = 0 & \Longleftrightarrow & y_1 = y_2 = 0 & \text{ when } & p \not = q,\\
x_1 = x_2 = 0 & \Longleftrightarrow & y_1=0 & \text{ when } & p = q,
\end{array}$$
and we claim that these are the same constraints obtained by requiring the $A$-modules be simple.
Indeed, consider the case $p \not = q$.
If $a$ and $b$ are arrows with respective $\bar{\tau}$-images $x_1$ and $x_2$ that annihilate $V$, then each arrow with $\bar{\tau}$-image divisible by $x_1$ or $x_2$ annihilates $V$ by Proposition \ref{(B/r)^d}.
Thus, since the only simple modules are vertex simples or large by Theorem \ref{hom3}, $V$ must be a vertex simple.
Therefore every arrow annihilates $V$, and in particular any arrow whose $\bar{\tau}$-image is divisible by $y_1$ or $y_2$ annihilates $V$.
The case $p = q$ is similar.

As an example of what happens when $\xi \not = 0$, consider the $Y^{4,2}$ algebra given in example \ref{Ypq}.  
The supporting subquivers for the two $\mathbb{CP}^1$-families $\mathcal{M}_{\xi >0}$ and $\mathcal{M}_{\xi <0}$ are given in figure \ref{r not = 0}, where the dotted arrows are represented by zero.
The $\bar{\tau}$-images of the solid arrows give explicit coordinates on each $\mathbb{CP}^1$, which are respectively $[y_1:y_2^3]$ and $[x_1:x_2]$.  
\begin{figure}
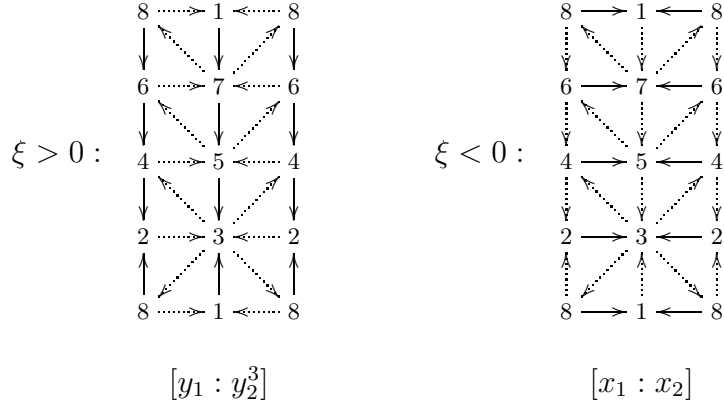

$$\begin{array}{ccccc}
\xi >0: & \xy
(-10,-20)*+{\text{\scriptsize{$8$}}}="a1";(-10,-10)*+{\text{\scriptsize{$2$}}}="a2";(-10,0)*+{\text{\scriptsize{$4$}}}="a3";(-10,10)*+{\text{\scriptsize{$6$}}}="a4";(-10,20)*+{\text{\scriptsize{$8$}}}="a5";
(0,-20)*+{\text{\scriptsize{$1$}}}="b1";(0,-10)*+{\text{\scriptsize{$3$}}}="b2";(0,0)*+{\text{\scriptsize{$5$}}}="b3";(0,10)*+{\text{\scriptsize{$7$}}}="b4";(0,20)*+{\text{\scriptsize{$1$}}}="b5";
(10,-20)*+{\text{\scriptsize{$8$}}}="c1";(10,-10)*+{\text{\scriptsize{$2$}}}="c2";(10,0)*+{\text{\scriptsize{$4$}}}="c3";(10,10)*+{\text{\scriptsize{$6$}}}="c4";(10,20)*+{\text{\scriptsize{$8$}}}="c5";
{\ar@{.>}"a1";"b1"};{\ar@{<.}"b1";"c1"};{\ar@{.>}"a2";"b2"};{\ar@{<.}"b2";"c2"};{\ar@{.>}"a3";"b3"};{\ar@{<.}"b3";"c3"};{\ar@{.>}"a4";"b4"};{\ar@{<.}"b4";"c4"};{\ar@{.>}"a5";"b5"};{\ar@{<.}"b5";"c5"};
{\ar@{->}"a1";"a2"};{\ar@{->}"b1";"b2"};{\ar@{->}"c1";"c2"};
{\ar@{<-}"a2";"a3"};{\ar@{<-}"b2";"b3"};{\ar@{<-}"c2";"c3"};
{\ar@{<-}"a3";"a4"};{\ar@{<-}"b3";"b4"};{\ar@{<-}"c3";"c4"};
{\ar@{<-}"a4";"a5"};{\ar@{<-}"b4";"b5"};{\ar@{<-}"c4";"c5"};
{\ar@{.>}"b2";"a1"};{\ar@{.>}"b2";"c1"};
{\ar@{.>}"b2";"a3"};{\ar@{.>}"b2";"c3"};
{\ar@{.>}"b3";"a4"};{\ar@{.>}"b3";"c4"};
{\ar@{.>}"b4";"a5"};{\ar@{.>}"b4";"c5"};
\endxy &
\ \ \ \ \ \ \  &
\xi < 0: & \xy
(-10,-20)*+{\text{\scriptsize{$8$}}}="a1";(-10,-10)*+{\text{\scriptsize{$2$}}}="a2";(-10,0)*+{\text{\scriptsize{$4$}}}="a3";(-10,10)*+{\text{\scriptsize{$6$}}}="a4";(-10,20)*+{\text{\scriptsize{$8$}}}="a5";
(0,-20)*+{\text{\scriptsize{$1$}}}="b1";(0,-10)*+{\text{\scriptsize{$3$}}}="b2";(0,0)*+{\text{\scriptsize{$5$}}}="b3";(0,10)*+{\text{\scriptsize{$7$}}}="b4";(0,20)*+{\text{\scriptsize{$1$}}}="b5";
(10,-20)*+{\text{\scriptsize{$8$}}}="c1";(10,-10)*+{\text{\scriptsize{$2$}}}="c2";(10,0)*+{\text{\scriptsize{$4$}}}="c3";(10,10)*+{\text{\scriptsize{$6$}}}="c4";(10,20)*+{\text{\scriptsize{$8$}}}="c5";
{\ar@{->}"a1";"b1"};{\ar@{<-}"b1";"c1"};{\ar@{->}"a2";"b2"};{\ar@{<-}"b2";"c2"};{\ar@{->}"a3";"b3"};{\ar@{<-}"b3";"c3"};{\ar@{->}"a4";"b4"};{\ar@{<-}"b4";"c4"};{\ar@{->}"a5";"b5"};{\ar@{<-}"b5";"c5"};
{\ar@{.>}"a1";"a2"};{\ar@{.>}"b1";"b2"};{\ar@{.>}"c1";"c2"};
{\ar@{<.}"a2";"a3"};{\ar@{<.}"b2";"b3"};{\ar@{<.}"c2";"c3"};
{\ar@{<.}"a3";"a4"};{\ar@{<.}"b3";"b4"};{\ar@{<.}"c3";"c4"};
{\ar@{<.}"a4";"a5"};{\ar@{<.}"b4";"b5"};{\ar@{<.}"c4";"c5"};
{\ar@{.>}"b2";"a1"};{\ar@{.>}"b2";"c1"};
{\ar@{.>}"b2";"a3"};{\ar@{.>}"b2";"c3"};
{\ar@{.>}"b3";"a4"};{\ar@{.>}"b3";"c4"};
{\ar@{.>}"b4";"a5"};{\ar@{.>}"b4";"c5"};
\endxy\\
\\
& \left[y_1:y_2^3\right] & & & \left[x_1:x_2\right]
\end{array}$$
\caption{The supporting subquivers for two $\mathbb{CP}^1$-families of modules over a $Y^{4,2}$ algebra, related by a flop.}
\label{r not = 0}
\end{figure}

Our proposal provides a geometric view of Van den Bergh's idea that the noncommutative resolution, loosely speaking, lies at the intersection between the various flops \cite{VdB2}.  However, the identification still needs to be made precise; progress in this direction is given in \cite{me}, but many questions remain.

\appendix
\section{Proof of Lemma \ref{appendix}}

Without loss of generality, take $\mathsf{u} = \mathsf{y}_1$ and suppose $\operatorname{h}(a) = \operatorname{h}(p)$. 
We proceed in a case-by-case analysis.  
The following argument will be used repeatedly: if $w = \mathsf{y}_{\alpha}\mathsf{x}_{\beta}\mathsf{y}_{\gamma}$ is a path with $\alpha \not = \gamma$, that is, $w$ equals say 
$\xy (-5,-5)*+{ \ }="1";(-5,5)*{\cdot}="3";(5,5)*{\cdot}="4";(5,-5)*+{ \ }="2";{\ar"1";"3"};{\ar"3";"4"};{\ar"4";"2"};(-5,-6)*{}="1a";(5,-6)*{}="2a";{\ar@{}|-1"1a";"1a"};{\ar@{}|-2"2a";"2a"};\endxy$,
$\xy (-5,-5)*+{ \ }="1";(-5,5)*{\cdot}="3";(5,5)*{}="4";(5,-5)*+{ \ }="2";{\ar"1";"3"};{\ar"3";"2"};(-5,-6)*{}="1a";(5,-6)*{}="2a";{\ar@{}|-1"1a";"1a"};{\ar@{}|-2"2a";"2a"};\endxy$, 
or $\xy (-5,-5)*+{ \ }="1";(-5,5)*{}="3";(5,5)*{\cdot}="4";(5,-5)*+{ \ }="2";{\ar"1";"4"};{\ar"4";"2"};(-5,-6)*{}="1a";(5,-6)*{}="2a";{\ar@{}|-1"1a";"1a"};{\ar@{}|-2"2a";"2a"};\endxy$,
 then modulo $\partial W$, $w$ also equals 
$\xy (-5,-5)*{\cdot}="1";(-5,5)*+{ \ }="3";(5,5)*+{ \ }="4";(5,-5)*{\cdot}="2";{\ar"3";"1"};{\ar"1";"2"};{\ar"2";"4"};(-5,6)*{}="1a";(5,6)*{}="2a";{\ar@{}|-1"1a";"1a"};{\ar@{}|-2"2a";"2a"};\endxy$, 
$\xy (-5,-5)*{}="1";(-5,5)*+{ \ }="3";(5,5)*+{ \ }="4";(5,-5)*{}="2";{\ar"3";"1"};{\ar"1";"4"};(-5,6)*{}="1a";(5,6)*{}="2a";{\ar@{}|-1"1a";"1a"};{\ar@{}|-2"2a";"2a"};\endxy$,
 or $\xy (-5,-5)*{}="1";(-5,5)*+{ \ }="3";(5,5)*+{ \ }="4";(5,-5)*{}="2";{\ar"3";"2"};{\ar"2";"4"};(-5,6)*{}="1a";(5,6)*{}="2a";{\ar@{}|-1"1a";"1a"};{\ar@{}|-2"2a";"2a"};\endxy$,
 so 
\begin{equation}
\label{diagonal}
\mathsf{y}_{\alpha}\mathsf{x}_{\beta}\mathsf{y}_{\gamma} \sim \mathsf{y}_{\gamma}\mathsf{x}_{\beta}\mathsf{y}_{\alpha}.
\end{equation}
Additionally, if say the factor $\mathsf{x}_{\beta}\mathsf{y}_{\gamma}$ is a diagonal arrow, then we also have $$\mathsf{y}_{\alpha}\mathsf{y}_{\gamma}\mathsf{x}_{\beta} \sim \mathsf{y}_{\alpha}\mathsf{x}_{\beta}\mathsf{y}_{\gamma} \sim \mathsf{y}_{\gamma}\mathsf{x}_{\beta}\mathsf{y}_{\alpha}.$$  
Similarly for $\mathsf{x} \leftrightarrow \mathsf{y}$.
\begin{itemize}
 \item $\mathsf{t}_1=\mathsf{y}_2$:
  \begin{itemize}
    \item If $n = 1$ or $\mathsf{t}_2= \mathsf{y}_2$ then $\mathsf{s}_1 = \mathsf{x}_{\alpha}$, hence $\mathsf{y}_1\mathsf{x}_{\alpha}$ is a diagonal arrow and $\mathsf{t}_1=\mathsf{y}_2$ is a vertical arrow.  Apply (\ref{diagonal}).
    \item If $\mathsf{t}_2= \mathsf{x}_{\alpha}$ then at least one of the factors $(\mathsf{x}_{\alpha}\mathsf{y}_2)$, $(\mathsf{y}_1\mathsf{s}_m)$, must be a diagonal arrow.  Both factors cannot be diagonal arrows since $\xy (-10,-5)*{\cdot}="1";(0,-5)*{\cdot}="2";(0,5)*{\cdot}="3";(10,-5)*{\cdot}="4";{\ar^{\mathsf{y}_1\mathsf{s}_m}"1";"3"};{\ar"3";"2"};{\ar"2";"1"};{\ar^{\mathsf{x}_{\alpha}\mathsf{y}_2}"3";"4"};\endxy$ is not a possible configuration.  Apply (\ref{diagonal}).
  \end{itemize}
 \item $\mathsf{t}_1 = \mathsf{x}_{\alpha}$:
   \begin{itemize}
    \item If $\mathsf{t}_2= \mathsf{y}_2$, either $\mathsf{y}_1$ or $\mathsf{t}_2=\mathsf{y}_2$ is a vertical arrow since otherwise either the configuration $\xy (-10,-5)*{\cdot}="1";(0,-5)*{\cdot}="2";(0,5)*{\cdot}="3";(10,-5)*{\cdot}="4";{\ar^{\mathsf{y}_1\mathsf{s}_m}"1";"3"};{\ar"3";"2"};{\ar"2";"1"};{\ar^{\mathsf{y}_2\mathsf{x}_{\alpha}}"3";"4"};\endxy$ 
    or
    $\xy (-10,-5)*{\cdot}="1";(0,-5)*{\cdot}="2";(0,5)*{\cdot}="3";(10,-5)*{\cdot}="4";{\ar^{\mathsf{x}_{\alpha}\mathsf{y}_1}"1";"3"};{\ar"3";"2"};{\ar"2";"1"};{\ar^{\mathsf{t}_3\mathsf{y}_2}"3";"4"};\endxy$
    would occur.
    \begin{itemize}
       \item If $\mathsf{y}_1$ is a vertical arrow, apply (\ref{diagonal}).  
       \item Suppose $\mathsf{y}_1$ is not a vertical arrow.  
         \begin{itemize} 
           \item If $\mathsf{x}_{\alpha}\not =\mathsf{s}_m = \mathsf{x}_{\beta}$ then $\mathsf{y}_2\mathsf{x}_{\alpha}\mathsf{y}_1\mathsf{s}_m$ is a unit cycle and we are done.  
           \item If $\mathsf{x}_{\alpha}=\mathsf{s}_m$ then the configuration $\xy (-10,-5)*{\cdot}="1";(0,-5)*{\cdot}="2";(0,5)*{\cdot}="3";(10,-5)*{\cdot}="4";(10,5)*{\cdot}="5";{\ar^{\mathsf{y}_1\mathsf{s}_m}"1";"3"};{\ar^{\mathsf{x}_{\alpha}}"3";"5"};{\ar^{\mathsf{y}_2}"5";"4"};{\ar"4";"3"};{\ar"2";"4"};{\ar"3";"2"};{\ar"2";"1"};\endxy$ 
       must occur since 
       $\xy (-10,-5)*{\cdot}="1";(0,-5)*{\cdot}="2";(0,5)*{\cdot}="3";(10,-5)*{\cdot}="4";(10,5)*{\cdot}="5";{\ar^{\mathsf{y}_1\mathsf{s}_m}"1";"3"};{\ar^{\mathsf{x}_{\alpha}}"3";"5"};{\ar^{\mathsf{y}_2}"5";"4"};{\ar"4";"2"};{\ar"2";"3"};\endxy$ 
       is not possible, and so
       $$\mathsf{y}_2\mathsf{x}_{\alpha}\mathsf{y}_1\mathsf{s}_m \sim \mathsf{x}_{\alpha}\mathsf{y}_2\mathsf{y}_1\mathsf{s}_m \sim \mathsf{x}_{\alpha}\mathsf{y}_2\mathsf{s}_m\mathsf{y}_1 \sim \mathsf{x}_{\alpha}\mathsf{y}_1\mathsf{s}_m\mathsf{y}_2.$$
       \end{itemize}
    \end{itemize}
    \item Suppose $\mathsf{t}_2 = \mathsf{x}_{\beta}$.  
      \begin{itemize}
        \item If $\mathsf{y}_1\mathsf{s}_m$ is a diagonal arrow then $\mathsf{t}_1=\mathsf{x}_{\alpha}$ must be a horizontal arrow.  Apply (\ref{diagonal}).  
        \item If $\mathsf{x}_{\alpha}\mathsf{y}_1$ is a diagonal arrow then $\mathsf{t}_2 = \mathsf{x}_{\beta}$ must be a horizontal arrow since otherwise the configuration $\xy (-10,-5)*{\cdot}="1";(0,-5)*{\cdot}="2";(0,5)*{\cdot}="3";(10,-5)*{\cdot}="4";{\ar^{\mathsf{x}_{\alpha}\mathsf{y}_1}"1";"3"};{\ar"3";"2"};{\ar"2";"1"};{\ar^{\mathsf{y}_2\mathsf{x}_{\beta}}"3";"4"};\endxy$ would occur.  Apply (\ref{diagonal}).
        \item Suppose $\mathsf{y}_1$ is a vertical arrow (that is, it is not `half of a diagonal arrow').
           \begin{itemize}
             \item If $\alpha \not = \beta$ then $\mathsf{t}_3\mathsf{x}_{\beta}$ is a diagonal arrow, hence $\mathsf{t}_3 = \mathsf{y}_2$.  Apply (\ref{diagonal}) twice: 
             $$(\mathsf{y}_2\mathsf{x}_{\beta})\mathsf{x}_{\alpha}\mathsf{y}_1 \sim (\mathsf{x}_{\beta}\mathsf{y}_2)\mathsf{x}_{\alpha}\mathsf{y}_1 \sim \mathsf{x}_{\alpha}\mathsf{y}_2\mathsf{x}_{\alpha}\mathsf{y}_1 \sim \mathsf{x}_{\alpha}\mathsf{y}_1\mathsf{x}_{\alpha}\mathsf{y}_2,$$
             where the last equality holds since $\mathsf{y}_1$ is a vertical arrow.
             \item Suppose $\alpha = \beta$, so $\bar{\tau}(\mathsf{t}_2\mathsf{t}_1\mathsf{u})=x_{\alpha}^2y_1$.  If the path $\mathsf{t}_2\mathsf{t}_1\mathsf{u}$ is in the configuration
             $\xy (-10,-5)*{\cdot}="1";(-10,5)*{\cdot}="2";(0,-5)*{\cdot}="3";(0,5)*{\cdot}="4";(10,5)*{\cdot}="5";{\ar^{\mathsf{u}}"1";"2"};{\ar^{\mathsf{t}_1}"2";"4"};{\ar^{\mathsf{t}_2}"4";"5"};{\ar"4";"1"};{\ar"1";"3"};{\ar"3";"4"};\endxy$
              then $\mathsf{x}_{\alpha}\mathsf{y}_1 \sim \mathsf{y}_1\mathsf{x}_{\alpha}$.  Otherwise $\mathsf{t}_2\mathsf{t}_1\mathsf{u}$ is in the configuration
              $\xy (-10,-5)*{\cdot}="1";(-10,5)*{\cdot}="2";(0,-5)*{\cdot}="3";(0,5)*{\cdot}="4";(10,-5)*{\cdot}="5";(10,5)*{\cdot}="6";{\ar^{\mathsf{u}}"1";"2"};{\ar^{\mathsf{t}_1}"2";"4"};{\ar^{\mathsf{t}_2}"4";"6"};{\ar"6";"5"};{\ar"5";"4"};{\ar"4";"3"};{\ar"3";"5"};{\ar"3";"1"};\endxy$
               since $\xy (-10,-5)*{\cdot}="1";(-10,5)*{\cdot}="2";(0,-5)*{\cdot}="3";(0,5)*{\cdot}="4";(10,-5)*{\cdot}="5";(10,5)*{}="6";{\ar^{\mathsf{u}}"1";"2"};{\ar^{\mathsf{t}_1}"2";"4"};{\ar^{\mathsf{t}_2}"4";"5"};{\ar"4";"3"};{\ar"3";"1"};\endxy$ 
               is not possible, and we may assume $\mathsf{u}$ is the leftmost $\mathsf{y}_1$ variable in $p$.  Repeating this argument we find that $p=\mathsf{x}_{\alpha}^n\mathsf{y}_1\mathsf{s}_m \cdots \mathsf{s}_1$ with $\mathsf{x}_{\alpha}^n\mathsf{y}_1$ in the configuration
               \begin{equation} \label{configuration} \xy
(-37,5)*{\cdot}="0b";(-27,5)*{\cdot}="1b";(-17,5)*{\cdot}="2b";(-7,5)*{\cdot}="3b";(0,5)*{\cdots}="mb";
(37,5)*{}="7b";(27,5)*{\cdot}="6b";(17,5)*{\cdot}="5b";(7,5)*{\cdot}="4b"; (-37,-5)*{\cdot}="0a";(-27,-5)*{\cdot}="1a";(-17,-5)*{\cdot}="2a";(-7,-5)*{\cdot}="3a";(0,-5)*{\cdots}="m";
(37,-5)*{\cdot}="7a";(27,-5)*{\cdot}="6a";(17,-5)*{\cdot}="5a";(7,-5)*{\cdot}="4a";
  (31,8)*{}="j";
 {\ar@{}^{\operatorname{h}(p)}"j";"j"};
 {\ar^{\mathsf{u}}"0a";"0b"};{\ar^{\mathsf{t}_1}"0b";"1b"};{\ar@{->}"1b";"1a"};{\ar@{->}"1a";"0a"};
 {\ar^{\mathsf{t}_2}"1b";"2b"};{\ar^{\mathsf{t}_3}"2b";"3b"};{\ar^{\mathsf{t}_{n-1}}"4b";"5b"};{\ar^{\mathsf{t}_n}"5b";"6b"};
 {\ar@{->}"1a";"2a"};{\ar@{->}"2a";"3a"};{\ar@{->}"4a";"5a"};{\ar@{->}"5a";"6a"};
 {\ar@{->}"2a";"1b"};{\ar@{->}"3a";"2b"};{\ar@{->}"5a";"4b"};{\ar@{->}"6a";"5b"};{\ar_a"7a";"6b"};
 {\ar@{->}"2b";"2a"};{\ar@{->}"3b";"3a"};{\ar@{->}"4b";"4a"};{\ar@{->}"5b";"5a"};{\ar@{->}"6b";"6a"};
 \endxy \end{equation}
By assumption there is an arrow $a$ with head at $\operatorname{h}(p)$ satisfying $u = y_1 \mid \bar{\tau}(a)$ and $\bar{\tau}(a) \mid \bar{\tau}(p)$.
From (\ref{configuration}) we see that $\bar{\tau}(a) = x_{\gamma}y_1$ with $\gamma \not = \alpha$.
Therefore $x_{\gamma} \mid \bar{\tau}(p)$.  
Let $b$ be the leftmost arrow in $p$ such that $x_{\gamma}$ divides $\bar{\tau}(b)$.  We may apply the arguments in all the above cases with $\mathsf{y}_1 \mapsto \mathsf{x}_{\gamma}$, with the exception of configuration (\ref{configuration}), to show that $\mathsf{x}_{\gamma}$ can be moved leftward so that it is adjacent to $\mathsf{u} = \mathsf{y}_1$ in $p$ modulo $\partial W$, and the result follows.  But it is not possible that both $\mathsf{y}_1$ and $\mathsf{x}_{\gamma}$ are in the configuration (\ref{configuration}):
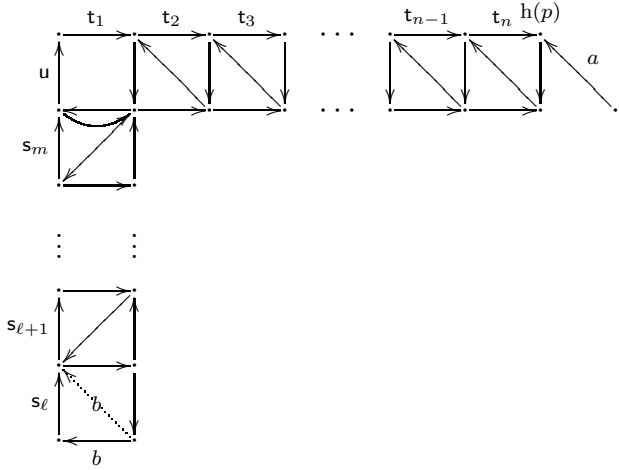
$$\xy
(-37,27)*{\cdot}="0b";(-27,27)*{\cdot}="1b";(-17,27)*{\cdot}="2b";(-7,27)*{\cdot}="3b";(0,27)*{\cdots}="mb";
(37,27)*{}="7b";(27,27)*{\cdot}="6b";(17,27)*{\cdot}="5b";(7,27)*{\cdot}="4b"; (-37,17)*{\cdot}="0a";(-27,17)*{\cdot}="1a";(-17,17)*{\cdot}="2a";(-7,17)*{\cdot}="3a";(0,17)*{\cdots}="m";
(37,17)*{\cdot}="7a";(27,17)*{\cdot}="6a";(17,17)*{\cdot}="5a";(7,17)*{\cdot}="4a";
  (31,30)*{}="j";
(-37,-27)*{\cdot}="1";(-27,-27)*{\cdot}="2";
(-37,-17)*{\cdot}="3";(-27,-17)*{\cdot}="4";
(-37,-7)*{\cdot}="5";(-27,-7)*{\cdot}="6";
(-37,0)*{\vdots}="md";(-27,0)*{\vdots}="md2";
(-37,7)*{\cdot}="7";(-27,7)*{\cdot}="8";
 {\ar@{}^{\operatorname{h}(p)}"j";"j"};
 {\ar^{\mathsf{u}}"0a";"0b"};{\ar^{\mathsf{t}_1}"0b";"1b"};{\ar@{->}"1b";"1a"};{\ar@{->}"1a";"0a"};
 {\ar^{\mathsf{t}_2}"1b";"2b"};{\ar^{\mathsf{t}_3}"2b";"3b"};{\ar^{\mathsf{t}_{n-1}}"4b";"5b"};{\ar^{\mathsf{t}_n}"5b";"6b"};
 {\ar@{->}"1a";"2a"};{\ar@{->}"2a";"3a"};{\ar@{->}"4a";"5a"};{\ar@{->}"5a";"6a"};
 {\ar@{->}"2a";"1b"};{\ar@{->}"3a";"2b"};{\ar@{->}"5a";"4b"};{\ar@{->}"6a";"5b"};{\ar_a"7a";"6b"};
 {\ar@{->}"2b";"2a"};{\ar@{->}"3b";"3a"};{\ar@{->}"4b";"4a"};{\ar@{->}"5b";"5a"};{\ar@{->}"6b";"6a"};
{\ar^b"2";"1"};{\ar^{\mathsf{s}_{\ell}}"1";"3"};{\ar@{.>}|-b"2";"3"};{\ar"4";"2"};{\ar"3";"4"};{\ar"6";"3"};{\ar^{\mathsf{s}_{\ell+1}}"3";"5"};{\ar"5";"6"};{\ar"4";"6"};{\ar"7";"8"};{\ar^{\mathsf{s}_m}"7";"0a"};{\ar"8";"1a"};{\ar"1a";"7"};{\ar@/_/"0a";"1a"};
 \endxy$$
           \end{itemize}
        \end{itemize}
      \item Suppose $n=1$.
        \begin{itemize}
          \item If $\mathsf{x}_{\alpha}\mathsf{y}_1$ is a diagonal arrow then $\mathsf{x}_{\alpha}\mathsf{y}_1 \sim \mathsf{y}_1\mathsf{x}_{\alpha}$.
          \item If $\mathsf{y}_1\mathsf{s}_m$ is a diagonal arrow, then apply (\ref{diagonal}).
          \item If $\mathsf{y}_1$ is a vertical arrow, then apply the above case $p \sim \mathsf{x}_{\alpha}^n\mathsf{y}_1\mathsf{s}_m \cdots \mathsf{s}_1$ with $n=1$.
        \end{itemize}
   \end{itemize}
\end{itemize}
This completes the proof.

\section{A math-physics dictionary for quivers}

In reverse geometric engineering \cite{DB1}, a type of quiver algebra called a \textit{superpotential algebra} is constructed from the (classical) vacuum equations of motion of an $\mathcal{N}=1$ supersymmetric quiver gauge theory.  In the original physics proposal/conjecture of Berenstein, Douglas, and Leigh (see \cite{DB1,BD} and references therein), the center of a superpotential algebra is the coordinate ring for an affine tangent cone (or at least some affine chart) on a 3 complex-dimensional singular\footnote{The Calabi-Yau variety need not be singular, but often theories with singularities are able to more closely model nature by, for example, breaking supersymmetry; see \cite{BHOP}.} Calabi-Yau variety--the hypothesized hidden internal space of our universe.\footnote{According to the AdS/CFT correspondence, this variety does not necessarily need to be actual physical space, but may instead just be a parameter space for something similar to mass (`vacuum expectation values') for certain fields that live in our $(3+1)$-dimensional spacetime manifold.}  The algebra itself is then viewed as a noncommutative ring of functions on the space of its simple modules, just as is the case in commutative algebraic geometry (when $k = \bar{k}$).  
They conjectured that, at least in physically relevant examples, this space is a \textit{noncommutative resolution} of the algebra's singular center since D-branes supposedly see the variety they are embedded in as smooth \cite{DGM}.

The $Y^{p,q}$ quivers are of interest to physicists since they encode the gauge theory in the conjectured AdS/CFT correspondence when the horizon is a $Y^{p,q}$ Sasaki-Einstein 5-manifold, given by metric data on the topological space $S^2 \times S^3$.  The $Y^{p,q}$ quiver gauge theories were constructed to model these geometries using symmetry arguments in a process known as geometric engineering, by Benvenuti, Franco, Hanany, Martelli, Sparks, and Kazakopoulos \cite{BFHMS, BHK}.  
In this paper we instead start with the $Y^{p,q}$ quiver gauge theories and derive their dual geometries by the methods of reverse geometric engineering; such a geometry is conjectured to coincide with the real cone over a $Y^{p,q}$ manifold (the horizon), but this is still unknown for $p >2$.  

The following is a partial dictionary between quiver gauge theories, specifically in regards to the mesonic branch since that is the focus of this paper, and quiver representation theory.  We begin with the following:
\begin{itemize}
  \item quiver gauge theory $\Leftrightarrow$ a quiver algebra and its representations;\\
        in particular, a $d=4$, $\mathcal{N}=1$ supersymmetric quiver gauge theory $\Leftrightarrow$ a path algebra modulo $F$-flatness constraints, i.e., a superpotential algebra
  \item complexified $U(n)$ gauge group $\Leftrightarrow$ general linear group\\
   (In this context, by $U(n)$ physicists usually mean $U(n)$
 \textit{complexified}, that is, if $H_1$ and $H_2$ are elements
 of the Lie algebra \textsl{u}$(n)$, then $\text{exp}(H_1 +
 iH_2) \in GL_n(\mathbb{C})$, and for any $L \in GL_n(\mathbb{C})$
 there is some such $H_1$ and $H_2$ such that $L =
 \text{exp}(H_1 + iH_2)$.)
  \item gauge invariance (under complexified gauge group)
 $\Leftrightarrow$ isomorphism classes of quiver representations
  \item Seiberg dual gauge theories, that is, different gauge theories
 in the UV which flow to the same fixed point in the IR
 $\Leftrightarrow$ different superpotential algebras that have the same
centers \textit{or} different superpotential algebras whose bounded derived categories of modules are equivalent \cite{BD}
\end{itemize}
In table \ref{table} we sketch an $\mathcal{N} = 1$, $d=4$ AdS/SCFT correspondence, or more generally a procedure for geometric and reverse geometric engineering, for a superpotential algebra $A$.  Note that the universe is thought to be a product $\mathcal{M} \times X$, where $\mathcal{M}$ is $3+1$ dimensional Minkowski space and $X$ is a compact 3 complex-dimensional (possibly singular) Calabi-Yau variety.  A D$n$-brane (with $n$ odd) fixes the endpoints of a string; mathematically it is a sheaf, or a complex of sheaves, supported on an $n+1$ real dimensional subvariety of $\mathcal{M} \times X$.  Here we only consider D3-branes which extend into $\mathcal{M}$ and are point-like, i.e., sky scraper sheaves, on $X$.  More generally though one also includes various $5$- and $7$-branes (such as in the physical realizations of dimer models), and D3-branes are allowed to wrap nontrivial cycles in $X$.  

\begin{table} 
\begin{center}\begin{tabular}{|l|l|l|}
\hline
\textit{Gauge theory on }$\mathcal{M}$ & \textit{Geometry and physics of }$X$ & \textit{Quiver representations}\\
\hline \hline
                        & a stack of $|Q_0|$   & vertices of $Q$\\
                        & fractional branes at &          \\
                        & the apex of the      &          \\
                        & tangent cone $C_p(X)$&          \\
                        & at a point $p \in X$&            \\
\hline
$U(1)$ gauge group on   &                     & the vertex simple    \\
the fractional brane at &                     & $A$-module $V^i$ (or \\
vertex $i$              &                     & its annihilator)  \\
\hline
(complexified) $U(n)$   &                     & $A$-module $V_{\rho}$ with\\
gauge group at vertex $i$&                     & $\operatorname{dim}_{\mathbb{C}} e_iV_{\rho} = n$\\
\hline
bifundamental field     & open oriented string& an arrow with tail at \\
transforming in the     & stretching from the & $i$ and head at $j$ \\
fundamental representation& fractional brane at &  \\
of the gauge group at   & vertex $i$ to the   &  \\
vertex $j$ and the anti-& fractional brane at & \\
fundamental representation& vertex $j$          & \\
of the gauge group at   &&\\
vertex $i$              &&\\
\hline
vev of a bifundamental &                     & matrix representation of\\
field                  &                     & the corresponding arrow\\
\hline
a possible configuration& a point in $C_p(X)$& an isoclass of simple \\
of vev's modulo the $F$-& (or a bulk D3-brane & $A$-modules (or the \\
flatness constraints   &  at a point in $C_p(X)$) & corresponding primitive\\
                       &                         & ideal)\\
\hline 
(mesonic) chiral ring  & coordinate ring for & center of $A$\\
                       & $C_p(X)$        & \\
\hline
mesonic field          &                     & cycle in the quiver\\                       
\hline
\end{tabular}\end{center}
\caption{A math-physics quiver dictionary for geometric engineering.}
\label{table}
\end{table}

\bibliographystyle{hep}
\def\cprime{$'$} \def\cprime{$'$}

\end{document}